\documentclass[12pt]{article}
\usepackage{amsmath}
\usepackage{amsthm}
\usepackage{amssymb}
\usepackage{amsfonts}
\usepackage{enumerate}
\usepackage{graphicx}
\usepackage{tikz}
\usetikzlibrary{decorations.pathreplacing}
\usepackage{caption}
\usepackage{mathtools}
\usepackage{subcaption}
\usepackage{fullpage}
\newtheorem{construction}{Construction}[section]
\newtheorem{theorem}[construction]{Theorem}
\newtheorem{corollary}[construction]{Corollary}
\newtheorem{lemma}[construction]{Lemma}
\newtheorem{remark}[construction]{Remark}

\theoremstyle{definition}
\newtheorem{definition}[construction]{Definition}
\newtheorem{conditions}[construction]{Conditions}
\newcounter{example}
\newenvironment{example}{\begin{description}\refstepcounter{example}%
\item[Example \theexample:]}{\end{description}}

\newenvironment{packed_item}{
\begin{itemize}
  \setlength{\itemsep}{1pt}
  \setlength{\parskip}{0pt}
  \setlength{\parsep}{0pt}
}{\end{itemize}}
\definecolor{Red}{rgb}{1,0.4,0.4}   
\definecolor{Green}{rgb}{.1,.5,.1}
\definecolor{Blue}{rgb}{.1,.1,.5}
\definecolor{blue}{RGB}{0,0,255}
\definecolor{Yellow}{rgb}{.8,.4, 0}
\definecolor{X}{rgb}{.8,.4, 0}
\definecolor{H}{rgb}{0,0,1}
\definecolor{light}{rgb}{.67,.84,.90}%
\definecolor{Cyan}{rgb}{0,1,1}%
\definecolor{Purple}{rgb}{.5,0,.5}%
\definecolor{Purple2}{rgb}{.5,.2,.5}%
\definecolor{white}{rgb}{1.0,1.0,1.0}%
\definecolor{Purple2}{rgb}{.8,.4, 0}
\definecolor{Amarillo}{RGB}{225,191,73}
\definecolor{Celeste}{RGB}{117,170,219}
\definecolor{Castano}{RGB}{132,53,17}
\definecolor{Black}{RGB}{0,0,0}
\definecolor{White}{RGB}{255,255,255}
\DeclareMathOperator{\lcm}{lcm}
\newcommand{\RGDD}{\ensuremath{\mbox{\sf RGDD}}}

\newcommand{\sgcd}[1]{\ensuremath{s\left(\rule[-1ex]{0pt}{3ex}#1\right)}}

 \renewcommand{\-}{\ensuremath{\text{--}}}
\title{A Generalization of the Hamilton-Waterloo Problem on Complete Equipartite Graphs}
\author{Melissa Keranen and Adri\'{a}n Pastine\\
Michigan Technological University\\
Michigan Technological University,
Houghton, MI 49931, U.S.A.}

\begin{document}
\maketitle
\abstract{The Hamilton-Waterloo problem asks for which $s$ and $r$ the complete graph $K_n$ can be decomposed into $s$ copies of a given 2-factor $F_1$ and $r$ copies of a given 2-factor $F_2$ (and one copy of a 1-factor if $n$ is even). In this paper we generalize the problem to complete equipartite graphs $K_{(n:m)}$ and show that $K_{(xyzw:m)}$ can be decomposed into $s$ copies of a 2-factor consisting of cycles of length $xzm$; and $r$ copies of a 2-factor consisting of cycles of length $yzm$, whenever $m$ is odd, $s,r\neq 1$, $\gcd(x,z)=\gcd(y,z)=1$ and $xyz\neq 0 \pmod 4$. We also give some more general constructions where the cycles in a given two factor may have different lengths. We use these constructions to find solutions to the Hamilton-Waterloo problem for complete graphs.}
\section{Introduction.}
The Oberwolfach Problem was first posed by Ringel in 1967 during a conference in Oberwolfach. The question was whether it was possible to seat the $v$ conference attendees at $n$ round tables for dinner during $\frac{v-1}{2}$ nights, in such a way that every attendee sits next to every other attendee exactly once. This is equivalent to asking whether the complete graph $K_v$ can be decomposed into $\frac{v-1}{2}$ copies of a $2$-factor $F$ (in a $2$-factor every component is a cycle, which represents a round table). To achieve this decomposition $v$ needs to be odd, because the vertices (attendees) need to have even degree. Later a version with $v$ even was studied. In this case, the attendees will never sit next to their spouses (and we are assuming that every attendee has a spouse). This is equivalent to asking for a decomposition of $K_v$ into $\frac{v-2}{2}$ copies of a $2$-factor $F$, and one copy of a $1$-factor (each attendee together with their spouse).

In \cite{Liu1} Liu first worked on the generalization of the Oberwolfach problem, where instead of avoiding their spouses, the attendees avoid all the other members of their delegation. The assumption was that all the delegations had the same number of people. Thus we are seeking to decompose the complete equipartite graph $K_{(m:n)}$ with $n$ partite sets (delegations) of size $m$ each (members of a delegations) into $\frac{(n-1)m}{2}$ copies of a $2$-factor $F$. Here $(n-1)m$ has to be even. In \cite{HH} Hoffman and Holliday worked on the equipartite generaliztion of the Oberwolfach problem when $(n-1)m$ is odd, decomposing into $\frac{(n-1)m-1}{2}$ copies of a $2$-factor $F$, and one copy of a $1$-factor.

The Hamilton-Waterloo problem is a generalization of the Oberwolfach problem, in which the conference is being held at two different cities. Because the table arrangements are different, we have two $2$-factors, $F_1$ and $F_2$. The Hamilton-Waterloo problem then asks whether the complete graph $K_v$ can be decomposed into $r$ copies of the $2$-factor $F_1$ (tables at Hamilton) and $s$ copies of the $2$-factor $F_2$ (tables at Waterloo), such that $s+r=\frac{v-1}{2}$, when $v$ is odd, and having $s+r=\frac{v-2}{2}$ and a $1$-factor when $v$ is even.

The uniform Oberwolfach problem (when all the tables have the same size, i.e. all the cycles of the $2$-factor have the same size) has been completely solved by Alspach and Haagkvist \cite{AH} and Alspach, Schellenberg, Stinson and Wagner \cite{ASSW}. For the non-uniform Oberwolfach problem many results have been obtained, for a survey of results up to 2006 see \cite{BR}. The uniform Oberwolfach problem over equipartite graphs has been completely solved by Liu \cite{Liu2} and Hoffman and Holliday \cite{HH}. In the non-uniform case Bryant, Danziger and Pettersson \cite{BDP} completely solved the case when the $2$-factor is bipartite. For the Hamilton-Waterloo problem most of the results are uniform, see for example \cite{AKKPO} or \cite{BDT}. In particular, Burgess, Danziger and Traetta \cite{BDT} proved the following theorem.

\begin{theorem}{\normalfont\cite{BDT}}\label{theoremBDT}
If $m$ and $n$ are odd integers with $n\geq m\geq 3$ and $t>1$, then there is a decomposition of $K_{mnt}$ into $s$ $C_{m}$-factors and $r$ $C_n$-factors if and only if $t$ is odd, $s,r\geq 0$ and $s+r=(mnt-1)/2$, except possibly when $r=1$ or $3$, or $(m,n,r)=(5,9,5),(5,9,7),(7,9,5),(7,9,7),(3,13,5)$.
\end{theorem}
Theorem \ref{theoremBDT} covers most of the odd ordered uniform cases. The authors in \cite{BDT} point out that it is possible to have solutions where the number of vertices is not a multimple of $mn$. Thus if $l=\lcm(m,n)$ and the number of vertices is a multiple of $l$, but not divisible by $mn$, then Theorem \ref{theoremBDT} cannot be applied. The constructions given in this paper can be applied to cover some of these cases.

There are some results in the non-uniform case, some examples are Bryant, Danzinger \cite{BD}, Bryant, Danzinger, Dean \cite{BDD} and Haggkvist \cite{H}.

The Hamilton-Waterloo problem can be generalized for complete equipartite graphs in the same way as the Oberwolfach problem was generalized, but not much work has been done in this direction. Asplund, Kamin, Keranen, Pastine and \"{O}zkan \cite{AKKPO} gave some constructions for complete equipartite graphs with $3$ parts. Burgess, Danziger and Traetta \cite{BDT} studied the case when the graph consists of $m$ partite sets of size $n$, and the cycle sizes are $m$ and $n$. In both papers the constructions were done in order to get a result on the Hamilton-Waterloo problem for complete graphs. The focus of this paper is to give a generalization of the Hamilton-Waterloo problem for complete equipartite graphs with an odd number of partite sets. We obtain results both in the uniform and non-uniform cases.
\section{Basic Definitions and Results}
Let $G$ be a multipartite graph with $k$ partite sets, $G_0,G_1,\ldots, G_{k-1}$. We identify each vertex  $g$ of $G$ as an ordered pair $(g,i)$, where $g\in G_i$.

\begin{example}
Consider the graph in Figure \ref{identificationfigure}. 
\begin{figure}
\begin{center}
\begin{tikzpicture}[every node/.style={draw,shape=circle}]
\draw (0,0) node [scale=.7, minimum size=.8cm](1){a};
\draw (1,0) node [scale=.7, minimum size=.8cm](2){b};
\draw (2,0) node [scale=.7, minimum size=.8cm](3){c};
\draw (0,-1) node[scale=.7, minimum size=.8cm] (4){d};
\draw (0,-2) node [scale=.7, minimum size=.8cm](5){f};
\draw (1,-1) node [scale=.7, minimum size=.8cm](6){e};
\draw (1) to (2);
\draw (2) to (4);
\draw (4) to (3);
\draw (3) to  [bend left=30](5);
\draw (5) to (6);
\draw (6) to (1);
\end{tikzpicture}
\caption{}
\label{identificationfigure}
\end{center}
\end{figure}
If we consider each column as a partite set, we have $3$ partite sets; $G_0$, with vertices $a,d$ and $f$; $G_1$ with vertices $b$ and $e$; and $G_2$ with $c$ as its only vertex. Then the vertices are $(a,0)$, $(b,1)$, $(c,2)$, $(d,0)$, $(e,1)$, $(f,0)$.

When it is convenient, we will just denote the vertices $(0,0)$, $(0,1)$, $(0,2)$, $(1,0)$, $(1,1)$, $(2,0)$, as in Figure \ref{identificationfigure2}.
\begin{figure}
\begin{center}
\begin{tikzpicture}[every node/.style={draw,shape=circle}]
\draw (0,0) node [scale=.7](1){0,0};
\draw (1,0) node [scale=.7](2){0,1};
\draw (2,0) node [scale=.7](3){0,2};
\draw (0,-1) node[scale=.7] (4){1,0};
\draw (0,-2) node [scale=.7](5){2,0};
\draw (1,-1) node [scale=.7](6){1,1};
\draw (1) to (2);
\draw (2) to (4);
\draw (4) to (3);
\draw (3) to  [bend left=30](5);
\draw (5) to (6);
\draw (6) to (1);
\end{tikzpicture}
\caption{}
\label{identificationfigure2}
\end{center}
\end{figure}
\end{example}
\begin{definition}\label{def1}
Let $G$ and $H$ be multipartite graphs. Then we define the \textit{partite product} of $G$ and $H$, $G\otimes   H$ as follows:
\begin{itemize}
\item $V(G\otimes   H)=\{(g,h,i)| (g,i)\in V(G)$ and $(h,i)\in V(H)\}$.
\item $E(G\otimes   H)=\{\{(g_1,h_1,i),(g_2,h_2,j)\}|\{(g_1,i),(g_2,j)\}\in E(G)$ and $\{(h_1,i),(h_2,j)\}\in E(H)\}$.
\end{itemize}
\end{definition}

Notice that this definition is quite similar to that of the direct product. The main difference is that we are doing this product ``just in $1$ coordinate''. To see that they are different it suffices to count the number of vertices in the product. If $G=H=K_{3,3,3}$ then $|V(G\times H)|=81$ but $|V(G\otimes   H)|=27$.

Indeed, if the $k$ partite sets of $G$ and $H$ have sizes $g_0,g_1,\ldots,g_{k-1}$ and $h_0,h_1,\ldots, h_{k-1}$, respectively, then $|V(G\otimes   H)|=g_0h_0+g_1h_1+\dots+g_{k-1}h_{k-1}$, whereas $|V(G\times H)|=\left(\sum g_i\right)\left(\sum h_i\right)$.
\begin{remark}
The partite product depends on the multipartite representation chosen for a graph. For example, the graphs $G$ and $H$ in Figure \ref{isomorphicgraphs} are isomorphic, but they behave differently in the product (where we understand that each column is a part of the multipartite graph).
\end{remark}

\begin{figure}
\begin{center}
\begin{tikzpicture}[every node/.style={draw,shape=circle}]
\draw (0,0) node [scale=.5](1){$0,0$};
\draw (1,0) node [scale=.5](2){$0,1$};
\draw (2,0) node [scale=.5](3){$0,2$};
\draw (0,-1) node[scale=.5] (4){$1,0$};
\draw (0,-2) node [scale=.5](5){$2,0$};
\draw (1,-1) node [scale=.5](6){$1,1$};
\draw (1) to (2);
\draw (2) to (4);
\draw (4) to (3);
\draw (3) to  [bend left=30](5);
\draw (5) to (6);
\draw (6) to (1);

\draw (4,0) node [scale=.5](11){$0,0$};
\draw (5,0) node [scale=.5](22){$0,1$};
\draw (6,0) node [scale=.5](33){$0,2$};
\draw (4,-1) node[scale=.5] (44){$1,0$};
\draw (5,-1) node [scale=.5](55){$1,1$};
\draw (6,-1) node [scale=.5](66){$1,2$};
\draw (11) to (22);
\draw (22) to (33);
\draw (44) to (33);
\draw (44) to (55);
\draw (55) to (66);
\draw (66) to (11);

\node[left=10pt,fill=none,draw=none] at (0,-.75) {$G=$};
\node[left=10pt,fill=none,draw=none] at (4,-.75) {$H=$};

\draw (-4,-4) node [scale=.5,inner sep=1pt](000){$0,0,0$};
\draw (-2,-4) node [scale=.5,inner sep=1pt](001){$0,0,1$};
\draw (0,-4) node [scale=.5,inner sep=1pt](002){$0,0,2$};

\draw (-4,-5) node [scale=.5,inner sep=1pt](010){$0,1,0$};
\draw (-2,-5) node [scale=.5,inner sep=1pt](011){$0,1,1$};

\draw (-4,-6) node [scale=.5,inner sep=1pt](020){$0,2,0$};

\draw (-4,-7) node [scale=.5,inner sep=1pt](100){$1,0,0$};
\draw (-2,-7) node [scale=.5,inner sep=1pt](101){$1,0,1$};

\draw (-4,-8) node [scale=.5,inner sep=1pt](110){$1,1,0$};
\draw (-2,-8) node [scale=.5,inner sep=1pt](111){$1,1,1$};

\draw (-4,-9) node [scale=.5,inner sep=1pt](120){$1,2,0$};

\draw (-4,-10) node [scale=.5,inner sep=1pt](200){$2,0,0$};

\draw (-4,-11) node [scale=.5,inner sep=1pt](210){$2,1,0$};

\draw (-4,-12) node [scale=.5,inner sep=1pt](220){$2,2,0$};

\draw (000) to (001);
\draw (000) to (011);
\draw (000) to (101);
\draw (000) to (111);

\draw (100) to (001);
\draw (100) to (011);
\draw (010) to (001);
\draw (010) to (101);

\draw (110) to (002);
\draw (220) [bend right=15] to (002);

\draw (020) to (111);
\draw (020) to (011);

\draw (110) to (001);

\draw (120) to (011);

\draw (200) to (101);
\draw (200) to (111);

\draw (210) to (101);

\draw (220) to (111);

\node[left=10pt,fill=none,draw=none] at (-4,-4.75) {$G\otimes G=$};

\draw (6,-4) node [scale=.5,inner sep=1pt](a000){$0,0,0$};
\draw (8,-4) node [scale=.5,inner sep=1pt](a001){$0,0,1$};
\draw (10,-4) node [scale=.5,inner sep=1pt](a002){$0,0,2$};
\draw (6,-5) node [scale=.5,inner sep=1pt](a010){$0,1,0$};
\draw (8,-5) node [scale=.5,inner sep=1pt](a011){$0,1,1$};
\draw (10,-5) node [scale=.5,inner sep=1pt](a012){$0,1,2$};
\draw (6,-6) node [scale=.5,inner sep=1pt](a100){$1,0,0$};
\draw (8,-6) node [scale=.5,inner sep=1pt](a101){$1,0,1$};
\draw (10,-6) node [scale=.5,inner sep=1pt](a102){$1,0,2$};
\draw (6,-7) node [scale=.5,inner sep=1pt](a110){$1,1,0$};
\draw (8,-7) node [scale=.5,inner sep=1pt](a111){$1,1,1$};
\draw (10,-7) node [scale=.5,inner sep=1pt](a112){$1,1,2$};

\draw (a000) to (a001);
\draw (a000) to (a112);

\draw (a010) to (a011);
\draw (a010) to (a102);

\draw (a100) to (a101);
\draw (a100) to (a012);

\draw (a110) to (a111);
\draw (a110) to (a002);

\draw (a001) to (a002);
\draw (a011) to (a012);
\draw (a101) to (a102);
\draw (a111) to (a112);
\node[left=10pt,fill=none,draw=none] at (6,-4.75) {$H\otimes H=$};
\end{tikzpicture}
\caption{}
\label{isomorphicgraphs}
\end{center}
\end{figure}

The next result follows directly from Definition \ref{def1}
\begin{lemma}\label{commutative}
The product is commutative, that is, $G\otimes   H=H\otimes   G$.
\end{lemma}
Most of our results will be concerning complete multipartite graphs. We will denote by $K_{(n:m)}$ the complete multipartite graph with $m$ parts, each of size $n$.
\begin{lemma}\label{k1kidentity}
If $G$ is $k$-partite, then $G\otimes   K_{(1:k)}$ is isomorphic to $G$.
\end{lemma}
\begin{proof}
Here each part of $K_{(1:k)}$ has just 1 vertex and so $|V(G)|=|V\left(G\otimes   K_{(1:k)})\right |$. Because all the vertices of $K_{(1:k)}$ are neighbors, two vertices $(g_1,k_1,i), (g_2,k_2,j)$ of $G\otimes   K_{(1:k)}$ are neighbors if and only if $(g_1,i)$ and $(g_2,j)$ are neighbors. Therefore $G\otimes   K_{(1:k)}$ is isomorphic to $G$.
\end{proof}

\begin{definition}
The \textit{complete cyclic multipartite graph} $C_{(x:k)}$ is the graph with $k$ parts of size $x$, where two vertices $(g,i)$ and $(h,j)$ are neighbors if and only if $i-j=\pm 1\pmod{k}$, with this subtraction being done modulo $k$. The \textit{directed complete cyclic multipartite graph} $\overrightarrow{C}_{(x:k)}$ is the graph with $k$ parts of size $x$, with arcs of the form $\big((g,i),(h,i+1)\big)$ for every $0\leq g,h\leq x-1$, $0\leq i\leq k-1$.
\end{definition}
It should be noted that any decomposition of $\overrightarrow{C}_{(x:k)}$ gives a decomposition of $C_{(x:k)}$. Notice that $C_{(1:k)}$ is the cycle with $k$ vertices and $C_{(x:3)}$ is isomorphic to $K_{(x:3)}$. 
The next three results are easy to see, so the proofs are left to the reader.
\begin{lemma}
Let $G$ and $H$ be $k$-partite graphs. If each part of $G$ has $\frac{|V(G)|}{k}$ vertices and each part of $H$ has $\frac{|V(H)|}{k}$ vertices, then:
\begin{itemize}
\item Each part of $G\otimes   H$ has $\frac{|V(G)|\times |V(H)|}{k^2}$ vertices.
\item $|V\left(G\otimes   H\right)|=\frac{|V(G)|\times |V(H)|}{k}$.
\end{itemize}
\end{lemma}

\begin{lemma}
\begin{itemize}
\item $K_{(x:k)}\otimes   K_{(y:k)}$ is isomorphic to $K_{(xy:k)}$.
\item $C_{(x:k)}\otimes   C_{(y:k)}$ is isomorphic to $C_{(xy:k)}$.
\item $\overrightarrow{C}_{(x:k)}\otimes   \overrightarrow{C}_{(y:k)}$ is isomorphic to $\overrightarrow{C}_{(xy:k)}$.
\end{itemize}
\end{lemma}

\begin{lemma}
The complete cyclic multipartite graph is the product of the complete multipartite graph by the cycle. This is: $K_{(x:k)}\otimes   C_{(1:k)}=C_{(x:k)}$.
\end{lemma}
\section{Product and Decompositions.}
We can consider a decomposition of a graph as a partition of the edge set or as a union of edge disjoint subgraphs. This means that a decomposition of $G$ into $H_1,\ldots,H_s$ is given by $E(G)=\cup E(H_i)$ or by $G=\oplus H_i$. We will think of $\oplus$ as a boolean sum, which means that $H_i\oplus H_i=\emptyset$.

We have the following easy result. 
\begin{theorem}[Distribution]\label{distributive}
Let $G=\oplus_i G_i$ and $H=\oplus_j H_j$ be $k$-partite graphs. Then $G\otimes   H=\left(\oplus_i G_i\right)\otimes  \left(\oplus_j H_j\right)$. Furthermore, the following distributive property holds:
\[
\left(\oplus_i G_i\right)\otimes  \left(\oplus_j H_j\right)=\oplus_i\left(G_i\otimes  \oplus_j H_j\right)=\oplus_i\oplus_j\left( G_i\otimes   H_j\right)
\]
\end{theorem}
\begin{proof}
It is enough to prove that 
\[
G\otimes  \left(H_1 \oplus H_2\right)=\left(G\otimes   H_1\right)\oplus\left(G\otimes   H_2\right),\]
where $E(H_1)\cap E(H_2)=\emptyset$.

Let 
\[
\{(g_1,h_1,i),(g_2,h_2,j)\}\in E\left(G\otimes  \left(H_1 \oplus H_2\right)\right).
\] 
This means that 
\[
\{(h_1,i),(h_2,j)\}\in E(H_1)\cup E(H_2).\]
But since $E(H_1)\cap E(H_2)=\emptyset$, without loss of generality we may assume 
\[
\{(h_1,i),(h_2,j)\}\in E(H_1),\]
and so 
\[
\{(g_1,h_1,i),(g_2,h_2,j)\}\in E\left(G\otimes   H_1\right)\subset E\left(\left(G\otimes   H_1\right)\oplus\left(G\otimes   H_2\right)\right).\]
Hence
\[
E\left(G\otimes  \left(H_1 \oplus H_2\right)\right)\subset E\left(\left(G\otimes   H_1\right)\oplus\left(G\otimes   H_2\right)\right).
\]

Let now 
\[
\{(g_1,h_1,i),(g_2,h_2,j)\}\in E\left(\left(G\otimes   H_1\right)\oplus\left(G\otimes   H_2\right)\right).\]
This means that 
\[
\{(g_1,h_1,i),(g_2,h_2,j)\}\in E\left(\left(G\otimes   H_1\right)\right) \cup E\left(\left(G\otimes   H_2\right)\right).\]
Without loss of generality we may assume 
\[
\{(g_1,h_1,i),(g_2,h_2,j)\}\in E\left(\left(G\otimes   H_1\right)\right),
\]
which gives 
\[
\{(h_1,i),(h_2,j)\}\in E(H_1)\subset E(H_1\oplus H_2),\text{ and}
\]
\[
\{(g_1,h_1,i),(g_2,h_2,j)\}\in E\left(G\otimes   \left(H_1\oplus H_2\right)\right).
\]
Hence
\[
E\left(\left(G\otimes   H_1\right)\oplus\left(G\otimes   H_2\right)\right)\subset E\left(G\otimes  \left(H_1 \oplus H_2\right)\right).\]

Therefore 
\[
G\otimes  \left(H_1 \oplus H_2\right)=\left(G\otimes   H_1\right)\oplus\left(G\otimes   H_2\right),
\]
and by induction we get that the product and additions in 
\[
G\otimes   H=\left(\oplus_i G_i\right)\otimes  \left(\oplus_j H_j\right)\]
are distributive.
\end{proof}

\begin{corollary}
Let $G$ and $H$ be multipartite graphs with $k$ partite sets.
\begin{itemize}
\item If $G$ can be decomposed into isomorphic copies of $\Gamma$ and $H$ can be decomposed into isomorphic copies of $K_{(1:k)}=K_k$, then $G\otimes   H$ can be decomposed into isomorphic copies of $\Gamma$.
\item If $G$ can be factored into isomorphic copies of $\Gamma$ and $H$ can be factored into unions of copies of $K_{(1:k)}=K_k$, then $G\otimes   H$ can be factored into unions of copies of $\Gamma$.
\end{itemize}
\end{corollary}
\begin{proof}
If $G$ is decomposed into copies of $\Gamma$, it means that $G=\oplus G_i$, where each $G_i$ is isomorphic to $\Gamma$. If $H$ is decomposed into copies of $K_{(1:k)}$ (or union of them), it means that $H=\oplus H_i$, where each $H_i$ is isomorphic to $K_{(1:k)}$ (or union of them).

By the Distribution Theorem we only need to show that $G_i\otimes   H_i$ is isomorphic to $\Gamma$ or to $G_i$. But by Lemma \ref{k1kidentity} we know this is true.
\end{proof}

It is interesting to notice that the set of $k$-partite graphs, with $\oplus$ as a sum and $\otimes  $ as a product form a commutative ring. The empty graph is the $0$ element, and $K_{(1:k)}$ is the $1$ element. All of the elements are additive involutions, Theorem \ref{distributive} gives us the distribution laws, and Lemma \ref{commutative} shows that the product is commutative. 
\section{Product of cycles.}
In this section we will concern ourselves with the product of two or more cycles.
Since our product depends on what kind of partition we are using, we need to ask something more from our cycles in order to get results.

\begin{definition}
Given a graph $G$ we will say that $C$ is a \textit{$C_n$-factor of $G$} if $C$ is a $2$-factor of $G$ where each connected component is of size $n$. This means that $C$ is a spanning subgraph of $G$ and $C$ is a union of disjoint cycles of size $n$.
When it is understood that the graph is $G$, then we will just call $C$ a $C_n$-factor (instead of a $C_n$-factor of $G$).
Similarly given a directed graph $\overrightarrow{G}$ we will say that $\overrightarrow{C}$ is a \textit{$\overrightarrow{C}_n$-factor of $G$} if $\overrightarrow{C}$ is a $2$-factor of $\overrightarrow{G}$ where each connected component is a directed cycle of size $n$. 
When it is understood that the graph is $\overrightarrow{G}$, the we will just call $\overrightarrow{C}$ a $\overrightarrow{C}_n$-factor.
\end{definition}

The following lemmas give us an idea of how directed cycles work under the product. They also illustrate why we introduce $\overrightarrow{C}_{(x:k)}$ instead of just working with $C_{(x:k)}$.

\begin{lemma}\label{productdirectedcycles}
Let $\overrightarrow{C}$ be a directed cycle of length $n$ of $\overrightarrow{C}_{(x:k)}$, and let $\overrightarrow{C}'$ be a directed cycle of length $m$ of $\overrightarrow{C}_{(y:k)}$. Then $\overrightarrow{C}\otimes \overrightarrow{C}'$ is a set of $\frac{\gcd(n,m)}{k}$ disjoint directed cycles of $\overrightarrow{C}_{(xy:k)}$ of length $l=\frac{nm}{\gcd(n,m)}$ and $xyk-\frac{nm}{k}$ isolated vertices.
\end{lemma}
\begin{proof}
Notice that $\overrightarrow{C}$ has $xk-n$ isolated vertices, because it is a subgraph of $\overrightarrow{C}_{(x:k)}$. Likewise, $\overrightarrow{C}'$ has $yk-m$ isolated vertices.

Let $(x_0,y_0,i)$ be a vertex in $\overrightarrow{C}\otimes   \overrightarrow{C}'$. If either $(x_0,i)$ or $(y_0,i)$ are isolated vertices, then $(x_0,y_0,i)$ is isolated. If neither $(x_0,i)$ and $(y_0,i)$ are isolated, then they respectively have an arrow coming from $(x_1,i-1)$ and $(y_1,i-1)$; and an arrow going to $(x_2,i+1)$ and $(y_2,i+1)$, for some $0\leq x_1,x_2\leq x-1$, $0\leq y_1,y_2\leq y-1$. Hence $(x_0,y_0,i)$ has an arrow coming from $(x_1,y_1,i-1)$ and an arrow going to $(x_2,y_2,i+1)$. So $\overrightarrow{C}\otimes \overrightarrow{C}'$ is composed of directed cycles and isolated vertices.

Assume without loss of generality $n\leq m$.
Let $i_0,i_1,\ldots, i_{n-1}$ be the non-isolated vertices of $\overrightarrow{C}$, in the order they appear, with $i_0\in \overrightarrow{C}_0$; and let $j_0,j_1,\ldots,j_{m-1}$ be the non-isolated vertices of $\overrightarrow{C}'$, in the order they appear with $j_0\in \overrightarrow{C}'_0$. Then the directed cycle starting at $(i_0,j_0)$ in $\overrightarrow{C}\otimes \overrightarrow{C}'$ consists of the vertices:
\[
(i_0,j_0),(i_1,j_1),\ldots, (i_{n-1},j_{n-1}), (i_0,j_{n}),\ldots (i_{n-1},i_{m-1})
\]
This directed cycle has length $l=\frac{nm}{gcd(n,m)}$.
Notice that $\overrightarrow{C}\otimes \overrightarrow{C}'$ has $\frac{nm}{k}$ non-isolated vertices, which means that the number of directed cycles is $\frac{nm}{k}\frac{gcd(n,m)}{nm}=\frac{\gcd(n,m)}{k}$. Thus
$\overrightarrow{C}\otimes \overrightarrow{C}'$ is a set of $\frac{\gcd(n,m)}{k}$ disjoint directed cycles of $\overrightarrow{C}_{(xy:k)}$ of length $l=\frac{nm}{\gcd(n,m)}$ and $xyk-\frac{nm}{k}$ isolated vertices.
\end{proof}
\begin{lemma}\label{productofbalanced}
Let $\overrightarrow{G}$ and $\overrightarrow{H}$ be a $\overrightarrow{C}_n$-factor and a $\overrightarrow{C}_m$-factor of $\overrightarrow{C}_{(x:k)}$ and $\overrightarrow{C}_{(y:k)}$, respectively. Then $\overrightarrow{G}\otimes   \overrightarrow{H}$ is a $\overrightarrow{C}_{l}$-factor of $\overrightarrow{C}_{(xy:k)}$, where $l=\frac{nm}{gcd(n,m)}$.
\end{lemma}
\begin{proof}
Notice that neither $\overrightarrow{G}$ nor $\overrightarrow{H}$ have isolated vertices.
Let $(x_0,y_0,i)$ be a vertex in $\overrightarrow{G}\otimes   \overrightarrow{H}$. We know that $(x_0,i)$ has an arrow coming from $(x_1,i-1)$ and an arrow going to $(x_2,i+1)$, for exactly one pair $0\leq x_1,x_2\leq x-1$, because $\overrightarrow{G}$ is a $\overrightarrow{C}_n$-factor. Likewise $(y_0,i)$ has has an arrow coming from $(y_1,i-1)$ and an arrow going to $(y_2,i+1)$, for exactly one pair $0\leq y_1,y_2\leq y-1$. Hence $(x_0,y_0,i)$ has an arrow coming from $(x_1,y_1,i-1)$ and an arrow going to $(x_2,y_2,i+1)$, so each vertex in $\overrightarrow{G}\otimes \overrightarrow{H}$ is in exactly one directed cycle.

Let $\overrightarrow{G}=\oplus_i \overrightarrow{C}(i)$ and $\overrightarrow{H}=\oplus_j \overrightarrow{C}'(j)$, where each $\overrightarrow{C}(i)$ is a directed cycle of length $n$, and each $\overrightarrow{C}'(j)$ is a directed cycle of length $m$.

Then by Theorem \ref{distributive} we get:
\begin{align*}
\overrightarrow{G}\otimes \overrightarrow{H}&=\left(\oplus_i \overrightarrow{C}(i)\right)\otimes\left(\oplus_j \overrightarrow{C}'(j)\right)\\
&=\oplus_i\oplus_j\left(\overrightarrow{C}(i)\otimes \overrightarrow{C}'(j)\right)\\
\end{align*}

But by Lemma \ref{productdirectedcycles} we know that $\overrightarrow{C}(i)\otimes \overrightarrow{C}'(j)$ is composed of $\frac{\gcd(n,m)}{k}$ directed cycles of length $l=\frac{nm}{gcd(n,m)}$.
Hence each directed cycle in $\overrightarrow{G}\otimes   \overrightarrow{H}$ has size $l=\frac{nm}{gcd(n,m)}$ and  $\overrightarrow{G}\otimes   \overrightarrow{H}$ is a $\overrightarrow{C}_{l}$-factor.
\end{proof}

\begin{definition}
Given a graph $G$ we will say that $F$ is a $[n_1^{e_1},n_2^{e_2},\ldots,n_p^{e_p}]$-factor of $G$ if $F$ is a $2$-factor of $G$ with $e_i$ connected components of size $n_i$, $i=1,2,\ldots, p$.
\end{definition}

If $e_i$ is not listed, we will just assume that it is $1$. Also, we allow $n_i=n_j$, so that the number of cycles of a certain size is just the sum of the exponents of that number in the expression $[n_1^{e_1},\ldots,n_p^{e_p}]$.
\begin{example}
A $[3^2,3^3,5^2,11,13]$-factor is a subgraph of a graph on $59$ vertices, consisting of $5$ cycles of size $3$, $2$ cycles of size $5$, $1$ cycle of size $11$ and $1$ cycle of size $13$. This subgraph can also be written as a $[3,3,3,3,3,5,5,11,13]$-factor or as a $[3^5,5^2,11,13]$-factor.
\end{example}
Notice that a $C_n$-factor of a graph on $m$ vertices is a $[n^{\frac{m}{n}}]$-factor.
\section{From $C_{(v:n)}$ to $K_{(v:m)}$}
It is advantageous to find solutions to the the Hamilton-Waterloo problem on complete multipartite graphs  because they can then be used to obtain solutions to the Hamilton-Waterloo problem on complete graphs:

\begin{lemma}\label{buildcompletegraph}
Let $m$, $n$, $x$, $y$ and $v$ be positive integers. Suppose the following conditions are satisfied:
\begin{itemize}
\item There exists a decomposition of $K_v$ into $s_\alpha$ $C_{xn}$-factors and $r_\alpha$ $C_{yn}$-factors.
\item There exists a decomposition of $K_{(v:m)}$ into $s_\beta$ $C_{xn}$-factors and $r_\beta$ $C_{yn}$-factors.
\end{itemize}

Then there exists a decomposition of $K_{vm}$ into $s=s_\alpha+s_\beta$ $C_{xn}$-factors and $r=r_\alpha+r_\beta$ $C_{yn}$-factors.
\end{lemma}
\begin{proof}
Partition the vertices of $K_{vm}$ into $m$ sets $A_1,\ldots, A_m$ of size $v$ each. 
The graph that contains the edges between vertices belonging to a same partite set is the union of $m$ disjoint copies of $K_v$. We can decompose each copy of $K_v$ into $s_\alpha$ $C_{xn}$-factors and $r_\alpha$ $C_{yn}$-factors.
The graph that contains the edges between vertices belonging to different parts is isomorphic to $K_{(v:m)}$. We can decompose this graph into $s_\beta$ $C_{xn}$-factors and $r_\beta$ $C_{yn}$-factors.

Thus we have a decomposition of $K_{vm}$ into $s=s_\alpha+s_\beta$ $C_{xn}$-factors and $r=r_\alpha+r_\beta$ $C_{yn}$-factors. 
\end{proof}

One could write a version of the lemma for non-uniform solutions as follows, where by $mK_v$ we understand the graph consisting of $m$ disconnected copies of $K_v$.
\begin{lemma}\label{buildcompletegraphnonuniform}
Let $m$, and $v$ be positive integers. Let $F_1$ and $F_2$ be two $2$-factors on $vm$ vertices. Suppose the following conditions are satisfied:
\begin{itemize}
\item There exists a decomposition of $mK_v$ into $s_\alpha$ copies of $F_1$ and $r_\alpha$ copies of $F_2$.
\item There exists a decomposition of $K_{(v:m)}$ into $s_\beta$ copies of $F_1$ and $r_\beta$ copies of $F_2$.
\end{itemize}

Then there exists a decomposition of $K_{vm}$ into $s=s_\alpha+s_\beta$ copies of $F_1$ and $r=r_\alpha+r_\beta$ copies of $F_2$.
\end{lemma}

In order to use these lemmas we need two types of ingredients. The first one is the decomposition of the complete graph $K_v$. In \cite{AH,ASSW} uniform decompositions were given:
\begin{theorem}
\label{OP}
\cite{AH,ASSW}
There exists a decomposition of $K_v$ into $C_n$-factors if and
only if $v \equiv 0 \pmod{n}$, $(v,n) \not = (6,3)$ and $(v,n) \not = (12,3)$.
\end{theorem}

The other ingredient is the decomposition of the complete multipartite graph $K_{(v:m)}$. In \cite{BDT} the authors used decompositions of $C_{(v:m)}$ to obtain decompositions of $K_{(v:m)}$. We describe this type of construction in a formal fashion in the following lemma.
\begin{lemma}\label{cvntokvm}
Let $m$, $x_1,\ldots,x_p$, $y_1,\ldots,y_p$, and $v$ be positive integers. Let $s_1,\ldots,s_{\frac{m-1}{2}}$ be non-negative integers. Suppose the following conditions are satisfied:
\begin{itemize}
\item There exists a decomposition of $K_m$ into $[n_1,\ldots,n_p]$-factors.
\item For every $1\leq i\leq p$, and for every $1\leq t \leq \frac{m-1}{2}$ there exists a decomposition of $C_{(v:n_i)}$ into $s_t$ $C_{x_in_i}$-factors and $r_t$ $C_{y_in_i}$-factors.
\end{itemize}

Let
\[
s=\sum_{t=1}^{\frac{(m-1)}{2}}s_t\text{\hspace{.3in} and \hspace{.3in}}r=\sum_{t=1}^{\frac{(m-1)}{2}}r_t
\]

Then there exists a decomposition of $K_{(v:m)}$ into $s$ $[(x_1n_1)^{\frac{v}{x_1}},\ldots,(x_pn_p)^{\frac{v}{x_p}}]$-factors and $r$ $[(y_1n_1)^{\frac{v}{y_1}},\ldots,(y_pn_p)^{\frac{v}{y_p}}]$-factors.
\end{lemma}
\begin{proof}
Using the decomposition of $K_m$ into $[n_1,n_2,\ldots,n_p]$-factors, we have:
\[
K_{(1:m)}=K_m=\bigoplus_{t=1}^{\frac{m-1}{2}}N_t
\]

where each $N_t$ is a $[n_1,n_2,\ldots,n_p]$-factor.

We know that $K_{(v:m)}=K_{(v:m)}\otimes   K_{(1:m)}$. This gives:
\[
\begin{array}{ccl}
K_{(v:m)}&=&K_{(v:m)}\otimes   K_{(1:m)}\\
&=&K_{(v:m)}\otimes   \left(\bigoplus_{t=1}^{\frac{m-1}{2}}N_t\right)\\
&=&\bigoplus_{t=1}^{\frac{m-1}{2}}\left(K_{(v:m)}\otimes   N_t\right)\\
\end{array}
\]

Notice that $K_{(v:m)}\otimes   N_t$ is a spanning subgraph of $K_{(v:m)}$. Since $N_t$ is a $[n_1,n_2,\ldots,n_p]$-factor, we have $N_t=\bigoplus_{i=1}^p C(w_i)$, where $C(w_i)$ is a cycle of size $n_i$.
Therefore,
\[
\begin{array}{ccl}
K_{(v:m)}\otimes   N_t&=K_{(v:m)}\otimes   \left(\bigoplus_{i=1}^p C(w_i)\right)\\
&=\bigoplus_{i=1}^p \left(K_{(v:m)}\otimes  C(w_i)\right)\\
\end{array}
\]

But $K{(v:m)}\otimes   C(w_i)$ is isomorphic to $C_{(v:n_i)}$ because $C(w_i)$ is isomorphic to $C_{(1:n_i)}$. So for each $w_i$ we can decompose $K{(v:m)}\otimes   C(w_i)$ into $s_t$ $C_{x_in_i}$-factors and $r_t$ $C_{y_in_i}$-factors. Since $C_{(v:n_i)}$ has $vn_i$ vertices, a $C_{x_in_i}$-factor has $\frac{v}{x_i}$ cycles, hence it is a $[(x_in_i)^{\frac{v}{x_i}}]$-factor. Likewise a $C_{y_in_i}$-factor is a $[(y_in_i)^{\frac{v}{y_i}}]$-factor.

 Taking the union of one $C_{x_in_i}$-factor for each $i$ gives a $[(x_1n_1)^{\frac{v}{x_1}},\ldots,(x_pn_p)^{\frac{v}{x_p}}]$-factor. Thus we get a decomposition of $K{(v:m)}\otimes  N_t$ into $s_t$ $[(x_1n_1)^{\frac{v}{x_1}},\ldots,(x_pn_p)^{\frac{v}{x_p}}]$-factors and $r_t$ $[(y_1n_1)^{\frac{v}{y_1}},\ldots,(y_pn_p)^{\frac{v}{y_p}}]$-factors.

Doing this for every $1\leq t \leq \frac{m-1}{2}$, we end up with a decomposition of $K_{(v:m)}$ into $s$ $[(x_1n_1)^{\frac{v}{x_1}},\ldots,(x_pn_p)^{\frac{v}{x_p}}]$-factors and $r$ $[(y_1n_1)^{\frac{v}{y_1}},\ldots,(y_pn_p)^{\frac{v}{y_p}}]$-factors.
\end{proof}

It is important to note that Theorem \ref{OP} can be used to obtain decompositions of $K_m$ into $C_n$-factors. Thus the focus of the next three sections is to find decompositions of $C_{(v:n)}$. Any decomposition of $\overrightarrow{C}_{(v:n)}$ is equivalent to a decomposition of $C_{(v:n)}$, by simply removing the direction of each edge. We will work with the partite product on directed graphs to find decompositions of $\overrightarrow{C}_{(v:n)}$, and thus obtain decompositions of $C_{(v:n)}$.
\section{Hamilton Waterloo problem on directed complete cyclic multipartite graphs.}
For the entirety of this section, we assume $x$ is odd. In this section we will decompose $\overrightarrow{C}_{(x:n)}$ into $\overrightarrow{C}_n$-factors and $\overrightarrow{C}_{xn}$-factors (Hamilton Cycles), and $\overrightarrow{C}_{(4x:n)}$ into $\overrightarrow{C}_n$-factors and $\overrightarrow{C}_{2xn}$-factors.

Suppose $G_{\alpha}$ and $G_{\beta}$ are two parts of size $x$ in an equipartite directed graph $G$. We say an arc in $G$ has difference $d$ if $\left((g_1,\alpha),(g_2,\beta)\right)\in E(G)$ and $g_2-g_1\equiv d \pmod{x}$. If $\left\lbrace\left((g_1,\alpha),(g_2,\beta)\right):g_2-g_1\equiv d \pmod x \right\rbrace \subset E(G)$, then we say that difference $d$ between parts $\alpha$ and $\beta$ is covered by $G$.

Let the partite sets of $\overrightarrow{C}_{(x:n)}$ be $G_0,G_1,\ldots,G_{n-1}$. Write $n-1=2^{e_1}+2^{e_2}+\ldots+2^{e_k}$ with $e_{\alpha} > e_{\beta}$ if $\alpha < \beta$. Notice that $k\leq \frac{n}{2}$. Working modulo $x$, let $T_x(i)$ be the directed subgraph of $\overrightarrow{C}_{(x:n)}$ obtained by taking differences:
\begin{itemize}
\item $2^{e_j}i$ between $G_{j-1}$ and $G_{j}$ for $1\leq j \leq k$.
\item $-2i$ between $G_{j-1}$ and $G_{j}$ for $k+1\leq j \leq 2k-1$.
\item $-i$ between $G_{j-1}$ and $G_{j}$ for $2k\leq j \leq n-1$.
\item $-i$ between $G_{n-1}$ and $G_{0}$.
\end{itemize}

\begin{example}
We construct $T_5(1)$ with $n=7$. We have $n-1=6=2^2+2^1$, and $k=2$.

This means that from the first column to the second one we add $2^2$, from the second to the third we add $2$; since $k=2$, from the third to the forth we subtract $2$, and for the rest of the arcs we just subtract $1$. Figure \ref{t51figure} shows one directed cycle of this construction.

\begin{figure}
\begin{center}
\begin{tikzpicture}[every node/.style={draw,shape=circle}]
\draw (0,0) node [scale=.5] (000){0,0};
\draw (1,0) node [scale=.5] (010){0,1};
\draw (2,0) node [scale=.5] (020){0,2};
\draw (3,0) node [scale=.5] (030){0,3};
\draw (4,0) node [scale=.5] (040){0,4};
\draw (5,0) node [scale=.5] (050){0,5};
\draw (6,0) node [scale=.5] (060){0,6};
\draw (0,-1) node [scale=.5](001){1,0};
\draw (1,-1) node [scale=.5](011){1,1};
\draw (2,-1) node [scale=.5](021){1,2};
\draw (3,-1) node [scale=.5](031){1,3};
\draw (4,-1) node [scale=.5](041){1,4};
\draw (5,-1) node [scale=.5](051){1,5};
\draw (6,-1) node [scale=.5](061){1,6};
\draw (0,-2) node [scale=.5](002){2,0};
\draw (1,-2) node [scale=.5](012){2,1};
\draw (2,-2) node [scale=.5](022){2,2};
\draw (3,-2) node [scale=.5](032){2,3};
\draw (4,-2) node [scale=.5](042){2,4};
\draw (5,-2) node [scale=.5](052){2,5};
\draw (6,-2) node [scale=.5](062){2,6};
\draw (0,-3) node [scale=.5](003){3,0};
\draw (1,-3) node [scale=.5](013){3,1};
\draw (2,-3) node [scale=.5](023){3,2};
\draw (3,-3) node [scale=.5](033){3,3};
\draw (4,-3) node [scale=.5](043){3,4};
\draw (5,-3) node [scale=.5](053){3,5};
\draw (6,-3) node [scale=.5](063){3,6};
\draw (0,-4) node [scale=.5](004){4,0};
\draw (1,-4) node [scale=.5](014){4,1};
\draw (2,-4) node [scale=.5](024){4,2};
\draw (3,-4) node [scale=.5](034){4,3};
\draw (4,-4) node [scale=.5](044){4,4};
\draw (5,-4) node [scale=.5](054){4,5};
\draw (6,-4) node [scale=.5](064){4,6};

\draw (000) [->,out=300,in=120,looseness=1] to (014);

\draw (014) [->] to (021);

\draw (021) [->] to (034);

\draw (034) [->] to (043);

\draw (043) [->] to (052);

\draw (052) [->] to (061);

\draw[dashed,->] (061) .. controls +(135:1) and +(-45:1) .. (000);
\end{tikzpicture}
\caption{one directed cycle in $T_5(1)$, with $n=7$.}
\label{t51figure}
\end{center}
\end{figure}
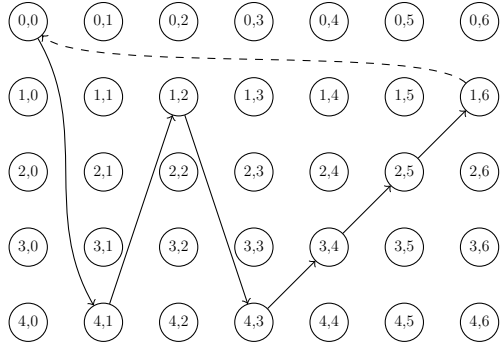

The rest of the arcs are obtained by developing this base directed cycle modulo $x$, i.e. if $\left((\alpha,j),(\beta,j+1)\right)\in E(T_x(i))$ then $\left((\alpha+h,j),(\beta+h,j+1)\right)\in E(T_x(i))$ for all $h\in \mathbb{Z}_x$.
\end{example}
\begin{lemma}\label{bbb}
$T_x(i)$ is a $\overrightarrow{C}_n$-factor for any $i$.
\end{lemma}
\begin{proof}
It suffices to show that the construction gives a base directed cycle of length $n$. The directed cycle containing the vertex $(0,0)$ can be tracked by considering the first coordinate of each vertex that is visited while passing from $G_0$ to $G_1$ to $G_2 \ldots$ to $G_0$. If we add the respective differences of the edges between $G_0$ and $G_1$, $G_1$ and $G_2, \ldots$, $G_{n-1}$ and $G_0$, we must show that this total sum is $0$. Because $2^{e_1}+2^{e_2}+\ldots+2^{e_k}=n-1$, we have for the sum:
\[
i(n-1)-2i(k-1)-i(n-2k+1)=i(n-1)-i(n-1)=0.
\]

\end{proof}

Let $F_h(G)$ be the directed subgraph of the directed graph $G$ that contains only the arcs between parts $h-1$ and $h$. That is 
\[
E(F_h(G))=\left\{\left((g_1,h-1)(g_2,h)\right)|\{(g_1,h-1)(g_2,h)\}\in E(G)\right\}.
\] 
In particular $F_n(G)$ contains the arcs between $G_{n-1}$ and $G_0$.

Let $H_x(i,s)=T_x(i)\oplus F_n(T_x(i))\oplus F_n(T_x(s))$. This means that $H_x(i,s)$ is the directed subgraph of $\overrightarrow{C}_{(x:n)}$ obtained by taking the same arcs as $T_x(i)$ between $G_j$ and $G_{j+1}$ for $0\leq j \leq n-2$ and $T_x(s)$ between $G_{n-1}$ and $G_0$.

\begin{example}
Figure \ref{Hexample} ilustrates $T_3(0)$, $T_3(1)$, $H_3(0,1)$ and $H_3(1,0)$ with $n=3$.

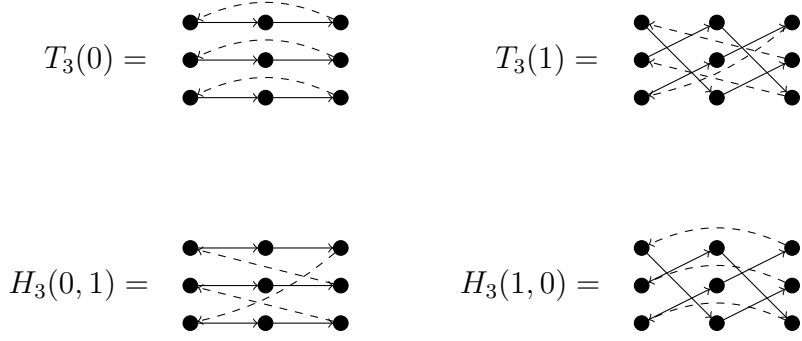
\begin{figure}
\begin{center}
\begin{tikzpicture}[every node/.style={draw,shape=circle,fill=black}]
\draw (0,0) node [scale=.5] (000){};
\draw (1,0) node [scale=.5] (010){};
\draw (2,0) node [scale=.5] (020){};
\draw (0,-.5) node [scale=.5](001){};
\draw (1,-.5) node [scale=.5](011){};
\draw (2,-.5) node [scale=.5](021){};
\draw (0,-1) node [scale=.5](002){};
\draw (1,-1) node [scale=.5](012){};
\draw (2,-1) node [scale=.5](022){};

\draw (000) [->] to (010);
\draw (010) [->] to (020);
\draw (020) [bend right=25, dashed,->] to (000);
\draw (001) [->] to (011);
\draw (011) [->] to (021);
\draw (021) [bend right=25, dashed,->] to (001);
\draw (002) [->] to (012);
\draw (012) [->] to (022);
\draw (022) [bend right=25, dashed,->] to (002);

\draw (6,0) node [scale=.5] (100){};
\draw (7,0) node [scale=.5] (110){};
\draw (8,0) node [scale=.5] (120){};
\draw (6,-.5) node [scale=.5](101){};
\draw (7,-.5) node [scale=.5](111){};
\draw (8,-.5) node [scale=.5](121){};
\draw (6,-1) node [scale=.5](102){};
\draw (7,-1) node [scale=.5](112){};
\draw (8,-1) node [scale=.5](122){};

\draw (100) [->] to (112);
\draw (112) [->] to (121);
\draw (121) [dashed,->] to (100);
\draw (101) [->] to (110);
\draw (110) [->] to (122);
\draw (122) [dashed,->] to (101);
\draw (102) [->] to (111);
\draw (111) [->] to (120);
\draw (120) [bend left=13, dashed,->] to (102);

\draw (0,-3) node [scale=.5] (200){};
\draw (1,-3) node [scale=.5] (210){};
\draw (2,-3) node [scale=.5] (220){};
\draw (0,-3.5) node [scale=.5](201){};
\draw (1,-3.5) node [scale=.5](211){};
\draw (2,-3.5) node [scale=.5](221){};
\draw (0,-4) node [scale=.5](202){};
\draw (1,-4) node [scale=.5](212){};
\draw (2,-4) node [scale=.5](222){};

\draw (200) [->] to (210);
\draw (210) [->] to (220);

\draw (201) [->] to (211);
\draw (211) [->] to (221);

\draw (202) [->] to (212);
\draw (212) [->] to (222);
\draw (220) [bend left=13, dashed,->] to (202);
\draw (222) [dashed,->] to (201);
\draw (221) [dashed,->] to (200);

\draw (6,-3) node [scale=.5] (300){};
\draw (7,-3) node [scale=.5] (310){};
\draw (8,-3) node [scale=.5] (320){};
\draw (6,-3.5) node [scale=.5](301){};
\draw (7,-3.5) node [scale=.5](311){};
\draw (8,-3.5) node [scale=.5](321){};
\draw (6,-4) node [scale=.5](302){};
\draw (7,-4) node [scale=.5](312){};
\draw (8,-4) node [scale=.5](322){};

\draw (320) [bend right=25, dashed,->] to (300);
\draw (321) [bend right=25, dashed,->] to (301);
\draw (322) [bend right=25, dashed,->] to (302);

\draw (300) [->] to (312);
\draw (312) [->] to (321);

\draw (301) [->] to (310);
\draw (310) [->] to (322);

\draw (302) [->] to (311);
\draw (311) [->] to (320);

\node[left=10pt,fill=none,draw=none] at (0,-.5) {$T_3(0)=$};
\node[left=10pt,fill=none,draw=none] at (6,-.5) {$T_3(1)=$};
\node[left=10pt,fill=none,draw=none] at (0,-3.5) {$H_3(0,1)=$};
\node[left=10pt,fill=none,draw=none] at (6,-3.5) {$H_3(1,0)=$};
\end{tikzpicture}
\caption{$T_3(0)$, $T_3(1)$, $H_3(0,1)$, $H_3(1,0)$.}
\label{Hexample}
\end{center}
\end{figure}
\end{example}

\begin{lemma}\label{HisforHamilton}
If $gcd(x,i-s)=1$ then $H_x(i,s)$ is a directed Hamiltonian cycle.
\end{lemma}
\begin{proof}
Because the arcs are given by differences it is clear that each vertex has in-degree and out-degree both equal to $1$.
We need to show that all of the vertices are connected. We will first show that there is a directed path between any 2 vertices of $G_0$.
Without loss of generality, we will show that $(0,0)$ is connected to $(\alpha,0)$ for any $\alpha$.
Because the arcs between groups $G_j$ and $G_{j+1}$ are the same as the arcs in $T_x(i)$ for $j=0,\ldots,n-2$ it is easy to see that there is a path from $(0,0)$ to $(i,n-1)$. Now the arc leaving from $(i,n-1)$ has its other end as $(i-s,0)$. So $(0,0)$ is connected to $(i-s,0)$. If we continue on this path, every time we arrive back in $G_0$, we arrive at the vertex $(\alpha'(i-s),0)$. Because $\gcd(x,i-s)=1$, the order of $i-s$ in the cyclic group $\mathbb{Z}_x$ is $x$. Thus any $\alpha$ modulo $x$ can be written as $\alpha'(i-s)$.
Hence $(0,0)$ is connected to all the vertices of $G_0$. 

Because we are defining arcs by differences, every vertex in $G_1$ is connected to a vertex in $G_0$, every vertex in $G_2$ to a vertex in $G_1$, and so on. Therefore all the vertices are connected, and the directed cycle is Hamiltonian as we wanted to prove.
\end{proof}

Next we show how to decompose $\overrightarrow{C}_{(x:n)}$ by using the $H_x(i,j)$ graphs. First we will decompose $\overrightarrow{C}_{(x:n)}$ into $\overrightarrow{C}_n$-factors using the $T_x(i)$ graphs. Then we will show how to switch some edges in the $T_x(i)$ graphs to obtain $H_x(i,j)$ graphs. It is important to notice that $H_x(i,i)=T_x(i)$.
\begin{lemma}
Let $x$ be odd, and let $\phi$ be a bijection on $\{0,...,x-1\}$. Then
\[
\overrightarrow{C}_{(x:n)}=\bigoplus_{i=0}^{x-1} T_x(i)=\bigoplus_{i=0}^{x-1} H_x(i,\phi(i))
\]
\end{lemma}
\begin{proof}
To prove the first equality,
\[
\overrightarrow{C}_{(x:n)}=\bigoplus_{i=0}^{x-1} T_x(i),
\]
we need to show that any difference between consecutive parts is covered by one of the $T_x(i)$ graphs. Notice that all the differences in the $T_x$ graphs are given by a power of $2$ times $i$, or $-2i$ or $-i$. It is clear that between parts which use difference $-i$, we cover all the differences, with difference $\delta$ being covered in $T_x(x-\delta)$. Because $x$ is odd, $\gcd(x,2^e)=1$, so the order of $2^e$ in the cyclic group $C_x$ is $x$ for any $1\leq e\leq x-1$. This means that any $\delta\equiv 2^{e_j}i\pmod{x}$ can be written as $2^e\delta'$. Therefore, the difference $\delta$ between the remaining pairs of consecutive groups is covered in some $T_x(\delta')$. A similar calculation can be used for differences of the form $-2i$, because $\gcd(x,x-2)=1$. Here we write $\delta$ as $-2\delta'$, and the difference is covered in $T_x(\delta')$.

For the second equality we have
\begin{align*}
\bigoplus_{i=0}^{x-1} H_x(i,\phi(i))&=\bigoplus_{i=0}^{x-1}\left( T_x(i)\oplus F_n(T_x(i))\oplus F_n (T_x(\phi(i)))\right)\\
&=\bigoplus_{i=0}^{x-1} T_x(i) \bigoplus_{i=0}^{x-1} F_n(T_x(i)) \bigoplus_{i=0}^{x-1} F_n(T_x(\phi(i)))
\end{align*}
and 
\begin{align*}
\bigoplus_{i=0}^{x-1} F_n(T_x(\phi(i)))&=\bigoplus_{i=0}^{x-1} F_n(T_x(i))
\end{align*} because $\phi$ is a bijection. So:
\begin{align*}
\bigoplus_{i=0}^{x-1} T_x(i) \bigoplus_{i=0}^{x-1} F_n(T_x(i)) \bigoplus_{i=0}^{x-1} F_n(T_x(\phi(i)))&=\bigoplus_{i=0}^{x-1} T_x(i) \bigoplus_{i=0}^{x-1} F_n(T_x(i)) \bigoplus_{i=0}^{x-1} F_n(T_x(i))\\
&=\bigoplus_{i=0}^{x-1} T_x(i)
\end{align*}
Hence:

\[
\bigoplus_{i=0}^{x-1} T_x(i)=\bigoplus_{i=0}^{x-1} H_x(i,\phi(i))
\]
\end{proof}

In some cases we have $H_x(i,i)$, notice that this is the same as $T_x(i)$. Decomposing $\overrightarrow{C}_{(x:n)}$ into $s$ directed Hamilton cycles and $x-s$ $\overrightarrow{C}_n$-factors is now equivalent to finding a bijection $\phi$ with $gcd(x,i-\phi(i))=1$ for $s$ elements of $\{0,...,x-1\}$ and $\phi(i)=i$ for the rest.
We will use these functions extensively throughout the paper, so we will refer to them as ``the phi-functions'':
Let $2\leq s \leq x$. Define $\phi_s:\mathbb{Z}_x\rightarrow \mathbb{Z}_x$ as follows:
\begin{center}
\underline{\textbf{The phi-functions}}
\[
\phi_s(i)=\left\lbrace \begin{array}{lcl} 

i+1 & \text{for}  & i\leq (s-3),\, i\text{ even}\\
i-1 & \text{for}  & i\leq (s-3),\, i\text{ odd} \\
i+1 & \text{for}  & i=s-2\\
i-2 & \text{for}  & i=s-1,\, s\text{ odd} \\
i-1 & \text{for}  & i=s-1,\, s\text{ even} \\
i &\text{for} & s\leq i \leq x-1\\

\end{array}\right.
\]
\end{center}
For example, if $s=7$ and $x=11$, then 
\[
\phi_7=(01)(23)(456)(7)(8)(9)(10)
\]

Notice that $\phi_s$ has $x-s$ fixed points. For any non-fixed point we have $i-\phi(i)\in \{\pm 1,2\}$, and so $\gcd(x,i-\phi(i))=1$ if $x$ is odd.

\begin{theorem}\label{1varfun}
Let $x$ be odd. Let $s\in \{0,...,x\}$, $s\neq 1$. Then $\overrightarrow{C}_{(x:n)}$ can be decomposed into $s$ directed Hamiltonian cycles and $x-s$ $\overrightarrow{C}_n$-factors.
\end{theorem}
\begin{proof}
If $s=0$ we just use the identity mapping.
Otherwise we use the phi-function $\phi_s$.
Hence the discussion that precedes this theorem shows that
\[
\overrightarrow{C}_{(x:n)}=\bigoplus_{i=0}^{x-1} H_x(i,\phi_s(i))
\]
is a decomposition of $\overrightarrow{C}_{(x:n)}$ into $s$ directed Hamiltonian cycles and $x-s$ $\overrightarrow{C}_n$-factors.
\end{proof}

Next we turn to the case of $x$ even. We begin by considering $\overrightarrow{C}_{(4:n)}$.
In $\cite{AKKPO}$ the graphs in Figure \ref{trianglesk43} were used to decompose $K_{(4:3)}$ into triangle factors and $C_6$-factors. 
We extend the ideas used in \cite{AKKPO} to decompose $\overrightarrow{C}_{(4:n)}$ into $\overrightarrow{C}_n$-factors and $\overrightarrow{C}_{2n}$-factors.
We will define directed subgraphs $\gamma_{i,j}$ and build the directed graphs $\Gamma(j)$ as the sum of some of these directed subgraphs. In each of these directed subgraphs, the vertices in the top row will be said to have height 0, in the second row height 1, and so on.
We begin with the base directed subgraphs from Figure \ref{babygammas} . In each subgraph, the rows are indexed by their height.
The following result is easy to verify by inspection of the graphs $\gamma_{i,j}$.

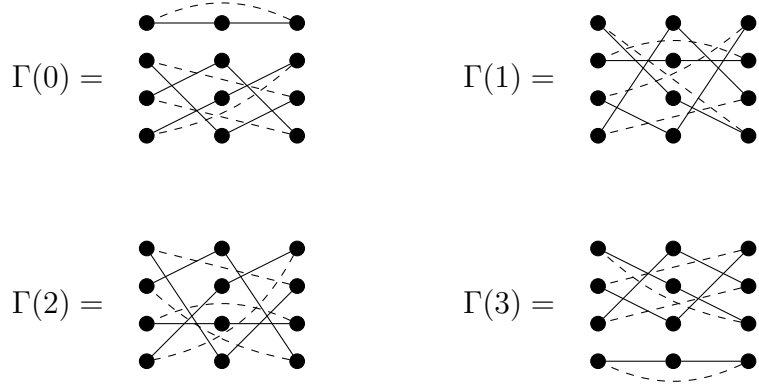
\begin{figure}
\begin{center}
\begin{tikzpicture}[every node/.style={draw,shape=circle,fill=black}]
\draw (0,0) node [scale=.5] (000){};
\draw (1,0) node [scale=.5] (010){};
\draw (2,0) node [scale=.5] (020){};
\draw (0,-.5) node [scale=.5](001){};
\draw (1,-.5) node [scale=.5](011){};
\draw (2,-.5) node [scale=.5](021){};
\draw (0,-1) node [scale=.5](002){};
\draw (1,-1) node [scale=.5](012){};
\draw (2,-1) node [scale=.5](022){};
\draw (0,-1.5) node [scale=.5](003){};
\draw (1,-1.5) node [scale=.5](013){};
\draw (2,-1.5) node [scale=.5](023){};

\draw (000) to (010);
\draw (010) to (020);
\draw (020) [bend right=25, dashed] to (000);
\draw (001) to (013);
\draw (013) to (022);
\draw (022) [dashed] to (001);
\draw (002) to (011);
\draw (011) to (023);
\draw (023) [dashed] to (002);
\draw (003) to (012);
\draw (012) to (021);
\draw (021) [bend left=15, dashed] to (003);

\draw (6,0) node [scale=.5] (100){};
\draw (7,0) node [scale=.5] (110){};
\draw (8,0) node [scale=.5] (120){};
\draw (6,-.5) node [scale=.5](101){};
\draw (7,-.5) node [scale=.5](111){};
\draw (8,-.5) node [scale=.5](121){};
\draw (6,-1) node [scale=.5](102){};
\draw (7,-1) node [scale=.5](112){};
\draw (8,-1) node [scale=.5](122){};
\draw (6,-1.5) node [scale=.5](103){};
\draw (7,-1.5) node [scale=.5](113){};
\draw (8,-1.5) node [scale=.5](123){};

\draw (101) to (111);
\draw (111) to (121);
\draw (121) [bend right=25, dashed] to (101);
\draw (100) to (112);
\draw (112) to (123);
\draw (123) [dashed] to (100);
\draw (102) to (113);
\draw (113) to (120);
\draw (120) [bend left=13, dashed] to (102);
\draw (103) to (110);
\draw (110) to (122);
\draw (122) [dashed] to (103);

\draw (0,-3) node [scale=.5] (200){};
\draw (1,-3) node [scale=.5] (210){};
\draw (2,-3) node [scale=.5] (220){};
\draw (0,-3.5) node [scale=.5](201){};
\draw (1,-3.5) node [scale=.5](211){};
\draw (2,-3.5) node [scale=.5](221){};
\draw (0,-4) node [scale=.5](202){};
\draw (1,-4) node [scale=.5](212){};
\draw (2,-4) node [scale=.5](222){};
\draw (0,-4.5) node [scale=.5](203){};
\draw (1,-4.5) node [scale=.5](213){};
\draw (2,-4.5) node [scale=.5](223){};

\draw (202) to (212);
\draw (212) to (222);
\draw (222) [bend right=25, dashed] to (202);
\draw (201) to (210);
\draw (210) to (223);
\draw (223) [bend left=15, dashed] to (201);
\draw (200) to (213);
\draw (213) to (221);
\draw (221) [dashed] to (200);
\draw (203) to (211);
\draw (211) to (220);
\draw (220) [bend left=25, dashed] to (203);

\draw (6,-3) node [scale=.5] (300){};
\draw (7,-3) node [scale=.5] (310){};
\draw (8,-3) node [scale=.5] (320){};
\draw (6,-3.5) node [scale=.5](301){};
\draw (7,-3.5) node [scale=.5](311){};
\draw (8,-3.5) node [scale=.5](321){};
\draw (6,-4) node [scale=.5](302){};
\draw (7,-4) node [scale=.5](312){};
\draw (8,-4) node [scale=.5](322){};
\draw (6,-4.5) node [scale=.5](303){};
\draw (7,-4.5) node [scale=.5](313){};
\draw (8,-4.5) node [scale=.5](323){};

\draw (303) to (313);
\draw (313) to (323);
\draw (323) [bend left=25, dashed] to (303);
\draw (301) to (312);
\draw (312) to (320);
\draw (320) [dashed] to (301);
\draw (302) to (310);
\draw (310) to (321);
\draw (321) [dashed] to (302);
\draw (300) to (311);
\draw (311) to (322);
\draw (322) [bend left=15, dashed] to (300);
\node[left=10pt,fill=none,draw=none] at (0,-.75) {$\Gamma(0)=$};
\node[left=10pt,fill=none,draw=none] at (6,-.75) {$\Gamma(1)=$};
\node[left=10pt,fill=none,draw=none] at (0,-3.75) {$\Gamma(2)=$};
\node[left=10pt,fill=none,draw=none] at (6,-3.75) {$\Gamma(3)=$};
\end{tikzpicture}
\caption{Triangle factors for $K_{(4:3)}$.}
\label{trianglesk43}
\end{center}
\end{figure}
\begin{figure}
\begin{center}
\begin{tikzpicture}[every node/.style={draw,shape=circle,fill=black}, every path/.style={->}]
\draw (0,0) node [scale=.5] (000){};
\draw (1,0) node [scale=.5] (010){};
\draw (2,0) node [scale=.5] (020){};
\draw (3,0) node [scale=.5] (030){};
\draw (0,-.5) node [scale=.5](001){};
\draw (1,-.5) node [scale=.5](011){};
\draw (2,-.5) node [scale=.5](021){};
\draw (3,-.5) node [scale=.5](031){};
\draw (0,-1) node [scale=.5](002){};
\draw (1,-1) node [scale=.5](012){};
\draw (2,-1) node [scale=.5](022){};
\draw (3,-1) node [scale=.5](032){};
\draw (0,-1.5) node [scale=.5](003){};
\draw (1,-1.5) node [scale=.5](013){};
\draw (2,-1.5) node [scale=.5](023){};
\draw (3,-1.5) node [scale=.5](033){};

\draw (000) to (010);
\draw (010) to (020);
\draw (020) to (030);
\draw (001) to (013);
\draw (013) to (022);
\draw (022) to (031);
\draw (002) to (011);
\draw (011) to (023);
\draw (023) to (032);
\draw (003) to (012);
\draw (012) to (021);
\draw (021) to (033);

\draw (6,0) node [scale=.5] (100){};
\draw (7,0) node [scale=.5] (110){};
\draw (8,0) node [scale=.5] (120){};
\draw (9,0) node [scale=.5] (130){};
\draw (6,-.5) node [scale=.5](101){};
\draw (7,-.5) node [scale=.5](111){};
\draw (8,-.5) node [scale=.5](121){};
\draw (9,-.5) node [scale=.5](131){};
\draw (6,-1) node [scale=.5](102){};
\draw (7,-1) node [scale=.5](112){};
\draw (8,-1) node [scale=.5](122){};
\draw (9,-1) node [scale=.5](132){};
\draw (6,-1.5) node [scale=.5](103){};
\draw (7,-1.5) node [scale=.5](113){};
\draw (8,-1.5) node [scale=.5](123){};
\draw (9,-1.5) node [scale=.5](133){};

\draw (101) to (111);
\draw (111) to (121);
\draw (121) to (131);
\draw (100) to (112);
\draw (112) to (123);
\draw (123) to (130);
\draw (102) to (113);
\draw (113) to (120);
\draw (120) to (132);
\draw (103) to (110);
\draw (110) to (122);
\draw (122) to (133);

\draw (0,-3) node [scale=.5] (200){};
\draw (1,-3) node [scale=.5] (210){};
\draw (2,-3) node [scale=.5] (220){};
\draw (3,-3) node [scale=.5] (230){};
\draw (0,-3.5) node [scale=.5](201){};
\draw (1,-3.5) node [scale=.5](211){};
\draw (2,-3.5) node [scale=.5](221){};
\draw (3,-3.5) node [scale=.5](231){};
\draw (0,-4) node [scale=.5](202){};
\draw (1,-4) node [scale=.5](212){};
\draw (2,-4) node [scale=.5](222){};
\draw (3,-4) node [scale=.5](232){};
\draw (0,-4.5) node [scale=.5](203){};
\draw (1,-4.5) node [scale=.5](213){};
\draw (2,-4.5) node [scale=.5](223){};
\draw (3,-4.5) node [scale=.5](233){};

\draw (202) to (212);
\draw (212) to (222);
\draw (222) to (232);
\draw (201) to (210);
\draw (210) to (223);
\draw (223) to (231);
\draw (200) to (213);
\draw (213) to (221);
\draw (221) to (230);
\draw (203) to (211);
\draw (211) to (220);
\draw (220) to (233);

\draw (6,-3) node [scale=.5] (300){};
\draw (7,-3) node [scale=.5] (310){};
\draw (8,-3) node [scale=.5] (320){};
\draw (9,-3) node [scale=.5] (330){};
\draw (6,-3.5) node [scale=.5](301){};
\draw (7,-3.5) node [scale=.5](311){};
\draw (8,-3.5) node [scale=.5](321){};
\draw (9,-3.5) node [scale=.5](331){};
\draw (6,-4) node [scale=.5](302){};
\draw (7,-4) node [scale=.5](312){};
\draw (8,-4) node [scale=.5](322){};
\draw (9,-4) node [scale=.5](332){};
\draw (6,-4.5) node [scale=.5](303){};
\draw (7,-4.5) node [scale=.5](313){};
\draw (8,-4.5) node [scale=.5](323){};
\draw (9,-4.5) node [scale=.5](333){};

\draw (303) to (313);
\draw (313) to (323);
\draw (323) to (333);
\draw (301) to (312);
\draw (312) to (320);
\draw (320) to (331);
\draw (302) to (310);
\draw (310) to (321);
\draw (321) to (332);
\draw (300) to (311);
\draw (311) to (322);
\draw (322) to (330);
\draw (0,0) node [scale=.5] (000){};
\draw (1,0) node [scale=.5] (010){};
\draw (2,0) node [scale=.5] (020){};
\draw (3,0) node [scale=.5] (030){};
\node[left=5pt,fill=none,draw=none] at (0,0) {$0$};
\node[left=5pt,fill=none,draw=none] at (0,-.5) {$1$};
\node[left=5pt,fill=none,draw=none] at (0,-1) {$2$};
\node[left=5pt,fill=none,draw=none] at (0,-1.5) {$3$};

\node[left=5pt,fill=none,draw=none] at (6,0) {$0$};
\node[left=5pt,fill=none,draw=none] at (6,-.5) {$1$};
\node[left=5pt,fill=none,draw=none] at (6,-1) {$2$};
\node[left=5pt,fill=none,draw=none] at (6,-1.5) {$3$};

\node[left=5pt,fill=none,draw=none] at (6,-3) {$0$};
\node[left=5pt,fill=none,draw=none] at (6,-3.5) {$1$};
\node[left=5pt,fill=none,draw=none] at (6,-4) {$2$};
\node[left=5pt,fill=none,draw=none] at (6,-4.5) {$3$};

\node[left=5pt,fill=none,draw=none] at (0,-3) {$0$};
\node[left=5pt,fill=none,draw=none] at (0,-3.5) {$1$};
\node[left=5pt,fill=none,draw=none] at (0,-4) {$2$};
\node[left=5pt,fill=none,draw=none] at (0,-4.5) {$3$};

\node[left=20pt,fill=none,draw=none] at (0,-.75) {$\gamma_{0,0}=$};
\node[left=20pt,fill=none,draw=none] at (6,-.75) {$\gamma_{0,1}=$};
\node[left=20pt,fill=none,draw=none] at (0,-3.75) {$\gamma_{0,2}=$};
\node[left=20pt,fill=none,draw=none] at (6,-3.75) {$\gamma_{0,3}=$};
\end{tikzpicture}
\end{center}

\vspace{.25in}

\begin{center}
\begin{tikzpicture}[every node/.style={draw,shape=circle,fill=black}, every path/.style={->}]
\draw (2,0) node [scale=.5] (100){};
\draw (3,0) node [scale=.5] (110){};
\draw (4,0) node [scale=.5] (120){};
\draw (5,0) node [scale=.5] (130){};
\draw (6,0) node [scale=.5] (140){};
\draw (2,-.5) node [scale=.5](101){};
\draw (3,-.5) node [scale=.5](111){};
\draw (4,-.5) node [scale=.5](121){};
\draw (5,-.5) node [scale=.5](131){};
\draw (6,-.5) node [scale=.5](141){};
\draw (2,-1) node [scale=.5](102){};
\draw (3,-1) node [scale=.5](112){};
\draw (4,-1) node [scale=.5](122){};
\draw (5,-1) node [scale=.5](132){};
\draw (6,-1) node [scale=.5](142){};
\draw (2,-1.5) node [scale=.5](103){};
\draw (3,-1.5) node [scale=.5](113){};
\draw (4,-1.5) node [scale=.5](123){};
\draw (5,-1.5) node [scale=.5](133){};
\draw (6,-1.5) node [scale=.5](143){};

\draw (100) to (110);
\draw (110) to (120);
\draw (120) to (130);
\draw (130) to (140);
\draw (101) to (113);
\draw (113) to (122);
\draw (122) to (133);
\draw (133) to (141);
\draw (102) to (111);
\draw (111) to (123);
\draw (123) to (131);
\draw (131) to (142);
\draw (103) to (112);
\draw (112) to (121);
\draw (121) to (132);
\draw (132) to (143);

\draw (9,0) node [scale=.5] (200){};
\draw (10,0) node [scale=.5] (210){};
\draw (11,0) node [scale=.5] (220){};
\draw (12,0) node [scale=.5] (230){};
\draw (13,0) node [scale=.5] (240){};
\draw (9,-.5) node [scale=.5](201){};
\draw (10,-.5) node [scale=.5](211){};
\draw (11,-.5) node [scale=.5](221){};
\draw (12,-.5) node [scale=.5](231){};
\draw (13,-.5) node [scale=.5](241){};
\draw (9,-1) node [scale=.5](202){};
\draw (10,-1) node [scale=.5](212){};
\draw (11,-1) node [scale=.5](222){};
\draw (12,-1) node [scale=.5](232){};
\draw (13,-1) node [scale=.5](242){};
\draw (9,-1.5) node [scale=.5](203){};
\draw (10,-1.5) node [scale=.5](213){};
\draw (11,-1.5) node [scale=.5](223){};
\draw (12,-1.5) node [scale=.5](233){};
\draw (13,-1.5) node [scale=.5](243){};

\draw (201) to (211);
\draw (211) to (221);
\draw (221) to (231);
\draw (231) to (241);
\draw (200) to (212);
\draw (212) to (223);
\draw (223) to (232);
\draw (232) to (240);
\draw (202) to (213);
\draw (213) to (220);
\draw (220) to (233);
\draw (233) to (242);
\draw (203) to (210);
\draw (210) to (222);
\draw (222) to (230);
\draw (230) to (243);

\node[left=5pt,fill=none,draw=none] at (2,0) {$0$};
\node[left=5pt,fill=none,draw=none] at (2,-.5) {$1$};
\node[left=5pt,fill=none,draw=none] at (2,-1) {$2$};
\node[left=5pt,fill=none,draw=none] at (2,-1.5) {$3$};

\node[left=5pt,fill=none,draw=none] at (9,0) {$0$};
\node[left=5pt,fill=none,draw=none] at (9,-.5) {$1$};
\node[left=5pt,fill=none,draw=none] at (9,-1) {$2$};
\node[left=5pt,fill=none,draw=none] at (9,-1.5) {$3$};

\node[left=20pt,fill=none,draw=none] at (2,-0.75) {$\gamma_{1,0}=$};
\node[left=20pt,fill=none,draw=none] at (9,-0.75) {$\gamma_{1,1}=$};

\draw (2,-3) node [scale=.5] (300){};
\draw (3,-3) node [scale=.5] (310){};
\draw (4,-3) node [scale=.5] (320){};
\draw (5,-3) node [scale=.5] (330){};
\draw (6,-3) node [scale=.5] (340){};
\draw (2,-3.5) node [scale=.5](301){};
\draw (3,-3.5) node [scale=.5](311){};
\draw (4,-3.5) node [scale=.5](321){};
\draw (5,-3.5) node [scale=.5](331){};
\draw (6,-3.5) node [scale=.5](341){};
\draw (2,-4) node [scale=.5](302){};
\draw (3,-4) node [scale=.5](312){};
\draw (4,-4) node [scale=.5](322){};
\draw (5,-4) node [scale=.5](332){};
\draw (6,-4) node [scale=.5](342){};
\draw (2,-4.5) node [scale=.5](303){};
\draw (3,-4.5) node [scale=.5](313){};
\draw (4,-4.5) node [scale=.5](323){};
\draw (5,-4.5) node [scale=.5](333){};
\draw (6,-4.5) node [scale=.5](343){};

\draw (302) to (312);
\draw (312) to (322);
\draw (322) to (332);
\draw (332) to (342);
\draw (301) to (310);
\draw (310) to (323);
\draw (323) to (330);
\draw (330) to (341);
\draw (300) to (313);
\draw (313) to (321);
\draw (321) to (333);
\draw (333) to (340);
\draw (303) to (311);
\draw (311) to (320);
\draw (320) to (331);
\draw (331) to (343);

\draw (9,-3) node [scale=.5] (400){};
\draw (10,-3) node [scale=.5] (410){};
\draw (11,-3) node [scale=.5] (420){};
\draw (12,-3) node [scale=.5] (430){};
\draw (13,-3) node [scale=.5] (440){};
\draw (9,-3.5) node [scale=.5](401){};
\draw (10,-3.5) node [scale=.5](411){};
\draw (11,-3.5) node [scale=.5](421){};
\draw (12,-3.5) node [scale=.5](431){};
\draw (13,-3.5) node [scale=.5](441){};
\draw (9,-4) node [scale=.5](402){};
\draw (10,-4) node [scale=.5](412){};
\draw (11,-4) node [scale=.5](422){};
\draw (12,-4) node [scale=.5](432){};
\draw (13,-4) node [scale=.5](442){};
\draw (9,-4.5) node [scale=.5](403){};
\draw (10,-4.5) node [scale=.5](413){};
\draw (11,-4.5) node [scale=.5](423){};
\draw (12,-4.5) node [scale=.5](433){};
\draw (13,-4.5) node [scale=.5](443){};

\draw (403) to (413);
\draw (413) to (423);
\draw (423) to (433);
\draw (433) to (443);
\draw (401) to (412);
\draw (412) to (420);
\draw (420) to (432);
\draw (432) to (441);
\draw (402) to (410);
\draw (410) to (421);
\draw (421) to (430);
\draw (430) to (442);
\draw (400) to (411);
\draw (411) to (422);
\draw (422) to (431);
\draw (431) to (440);

\node[left=5pt,fill=none,draw=none] at (2,-3) {$0$};
\node[left=5pt,fill=none,draw=none] at (2,-3.5) {$1$};
\node[left=5pt,fill=none,draw=none] at (2,-4) {$2$};
\node[left=5pt,fill=none,draw=none] at (2,-4.5) {$3$};

\node[left=5pt,fill=none,draw=none] at (9,-3) {$0$};
\node[left=5pt,fill=none,draw=none] at (9,-3.5) {$1$};
\node[left=5pt,fill=none,draw=none] at (9,-4) {$2$};
\node[left=5pt,fill=none,draw=none] at (9,-4.5) {$3$};

\node[left=20pt,fill=none,draw=none] at (2,-3.75) {$\gamma_{1,2}=$};
\node[left=20pt,fill=none,draw=none] at (9,-3.75) {$\gamma_{1,3}=$};

\end{tikzpicture}
\end{center}

\vspace{.25in}

\begin{center}
\begin{tikzpicture}[every node/.style={draw,shape=circle,fill=black}, every path/.style={->}]

\draw (1,0) node [scale=.5] (100){};
\draw (2,0) node [scale=.5] (110){};
\draw (3,0) node [scale=.5] (120){};
\draw (4,0) node [scale=.5] (130){};
\draw (5,0) node [scale=.5] (140){};
\draw (6,0) node [scale=.5] (150){};
\draw (1,-.5) node [scale=.5](101){};
\draw (2,-.5) node [scale=.5](111){};
\draw (3,-.5) node [scale=.5](121){};
\draw (4,-.5) node [scale=.5](131){};
\draw (5,-.5) node [scale=.5](141){};
\draw (6,-.5) node [scale=.5](151){};
\draw (1,-1) node [scale=.5](102){};
\draw (2,-1) node [scale=.5](112){};
\draw (3,-1) node [scale=.5](122){};
\draw (4,-1) node [scale=.5](132){};
\draw (5,-1) node [scale=.5](142){};
\draw (6,-1) node [scale=.5](152){};
\draw (1,-1.5) node [scale=.5](103){};
\draw (2,-1.5) node [scale=.5](113){};
\draw (3,-1.5) node [scale=.5](123){};
\draw (4,-1.5) node [scale=.5](133){};
\draw (5,-1.5) node [scale=.5](143){};
\draw (6,-1.5) node [scale=.5](153){};

\draw (100) to (110);
\draw (110) to (120);
\draw (120) to (130);
\draw (130) to (140);
\draw (140) to (150);
\draw (101) to (113);
\draw (113) to (122);
\draw (122) to (131);
\draw (131) to (143);
\draw (143) to (151);
\draw (102) to (111);
\draw (111) to (123);
\draw (123) to (132);
\draw (132) to (141);
\draw (141) to (152);
\draw (103) to (112);
\draw (112) to (121);
\draw (121) to (133);
\draw (133) to (142);
\draw (142) to (153);

\draw (9,0) node [scale=.5] (200){};
\draw (10,0) node [scale=.5] (210){};
\draw (11,0) node [scale=.5] (220){};
\draw (12,0) node [scale=.5] (230){};
\draw (13,0) node [scale=.5] (240){};
\draw (14,0) node [scale=.5] (250){};
\draw (9,-.5) node [scale=.5](201){};
\draw (10,-.5) node [scale=.5](211){};
\draw (11,-.5) node [scale=.5](221){};
\draw (12,-.5) node [scale=.5](231){};
\draw (13,-.5) node [scale=.5](241){};
\draw (14,-.5) node [scale=.5](251){};
\draw (9,-1) node [scale=.5](202){};
\draw (10,-1) node [scale=.5](212){};
\draw (11,-1) node [scale=.5](222){};
\draw (12,-1) node [scale=.5](232){};
\draw (13,-1) node [scale=.5](242){};
\draw (14,-1) node [scale=.5](252){};
\draw (9,-1.5) node [scale=.5](203){};
\draw (10,-1.5) node [scale=.5](213){};
\draw (11,-1.5) node [scale=.5](223){};
\draw (12,-1.5) node [scale=.5](233){};
\draw (13,-1.5) node [scale=.5](243){};
\draw (14,-1.5) node [scale=.5](253){};

\draw (201) to (211);
\draw (211) to (221);
\draw (221) to (231);
\draw (231) to (241);
\draw (241) to (251);
\draw (200) to (212);
\draw (212) to (223);
\draw (223) to (230);
\draw (230) to (242);
\draw (242) to (250);
\draw (202) to (213);
\draw (213) to (220);
\draw (220) to (232);
\draw (232) to (243);
\draw (243) to (252);
\draw (203) to (210);
\draw (210) to (222);
\draw (222) to (233);
\draw (233) to (240);
\draw (240) to (253);

\node[left=5pt,fill=none,draw=none] at (1,0) {$0$};
\node[left=5pt,fill=none,draw=none] at (1,-.5) {$1$};
\node[left=5pt,fill=none,draw=none] at (1,-1) {$2$};
\node[left=5pt,fill=none,draw=none] at (1,-1.5) {$3$};

\node[left=5pt,fill=none,draw=none] at (9,0) {$0$};
\node[left=5pt,fill=none,draw=none] at (9,-.5) {$1$};
\node[left=5pt,fill=none,draw=none] at (9,-1) {$2$};
\node[left=5pt,fill=none,draw=none] at (9,-1.5) {$3$};

\node[left=20pt,fill=none,draw=none] at (1,-0.75) {$\gamma_{2,0}=$};
\node[left=20pt,fill=none,draw=none] at (9,-0.75) {$\gamma_{2,1}=$};

\draw (1,-3) node [scale=.5] (300){};
\draw (2,-3) node [scale=.5] (310){};
\draw (3,-3) node [scale=.5] (320){};
\draw (4,-3) node [scale=.5] (330){};
\draw (5,-3) node [scale=.5] (340){};
\draw (6,-3) node [scale=.5] (350){};
\draw (1,-3.5) node [scale=.5](301){};
\draw (2,-3.5) node [scale=.5](311){};
\draw (3,-3.5) node [scale=.5](321){};
\draw (4,-3.5) node [scale=.5](331){};
\draw (5,-3.5) node [scale=.5](341){};
\draw (6,-3.5) node [scale=.5](351){};
\draw (1,-4) node [scale=.5](302){};
\draw (2,-4) node [scale=.5](312){};
\draw (3,-4) node [scale=.5](322){};
\draw (4,-4) node [scale=.5](332){};
\draw (5,-4) node [scale=.5](342){};
\draw (6,-4) node [scale=.5](352){};
\draw (1,-4.5) node [scale=.5](303){};
\draw (2,-4.5) node [scale=.5](313){};
\draw (3,-4.5) node [scale=.5](323){};
\draw (4,-4.5) node [scale=.5](333){};
\draw (5,-4.5) node [scale=.5](343){};
\draw (6,-4.5) node [scale=.5](353){};

\draw (302) to (312);
\draw (312) to (322);
\draw (322) to (332);
\draw (332) to (342);
\draw (342) to (352);
\draw (301) to (310);
\draw (310) to (323);
\draw (323) to (331);
\draw (331) to (340);
\draw (340) to (351);
\draw (300) to (313);
\draw (313) to (321);
\draw (321) to (330);
\draw (330) to (343);
\draw (343) to (350);
\draw (303) to (311);
\draw (311) to (320);
\draw (320) to (333);
\draw (333) to (341);
\draw (341) to (353);

\draw (9,-3) node [scale=.5] (400){};
\draw (10,-3) node [scale=.5] (410){};
\draw (11,-3) node [scale=.5] (420){};
\draw (12,-3) node [scale=.5] (430){};
\draw (13,-3) node [scale=.5] (440){};
\draw (14,-3) node [scale=.5] (450){};
\draw (9,-3.5) node [scale=.5](401){};
\draw (10,-3.5) node [scale=.5](411){};
\draw (11,-3.5) node [scale=.5](421){};
\draw (12,-3.5) node [scale=.5](431){};
\draw (13,-3.5) node [scale=.5](441){};
\draw (14,-3.5) node [scale=.5](451){};
\draw (9,-4) node [scale=.5](402){};
\draw (10,-4) node [scale=.5](412){};
\draw (11,-4) node [scale=.5](422){};
\draw (12,-4) node [scale=.5](432){};
\draw (13,-4) node [scale=.5](442){};
\draw (14,-4) node [scale=.5](452){};
\draw (9,-4.5) node [scale=.5](403){};
\draw (10,-4.5) node [scale=.5](413){};
\draw (11,-4.5) node [scale=.5](423){};
\draw (12,-4.5) node [scale=.5](433){};
\draw (13,-4.5) node [scale=.5](443){};
\draw (14,-4.5) node [scale=.5](453){};

\draw (403) to (413);
\draw (413) to (423);
\draw (423) to (433);
\draw (433) to (443);
\draw (443) to (453);
\draw (401) to (412);
\draw (412) to (420);
\draw (420) to (431);
\draw (431) to (442);
\draw (442) to (451);
\draw (402) to (410);
\draw (410) to (421);
\draw (421) to (432);
\draw (432) to (440);
\draw (440) to (452);
\draw (400) to (411);
\draw (411) to (422);
\draw (422) to (430);
\draw (430) to (441);
\draw (441) to (450);

\node[left=5pt,fill=none,draw=none] at (1,-3) {$0$};
\node[left=5pt,fill=none,draw=none] at (1,-3.5) {$1$};
\node[left=5pt,fill=none,draw=none] at (1,-4) {$2$};
\node[left=5pt,fill=none,draw=none] at (1,-4.5) {$3$};

\node[left=5pt,fill=none,draw=none] at (9,-3) {$0$};
\node[left=5pt,fill=none,draw=none] at (9,-3.5) {$1$};
\node[left=5pt,fill=none,draw=none] at (9,-4) {$2$};
\node[left=5pt,fill=none,draw=none] at (9,-4.5) {$3$};

\node[left=20pt,fill=none,draw=none] at (1,-3.75) {$\gamma_{2,2}=$};
\node[left=20pt,fill=none,draw=none] at (9,-3.75) {$\gamma_{2,3}=$};

\end{tikzpicture}
\caption{}
\label{babygammas}
\end{center}
\end{figure}
\begin{lemma}\label{fixedheight}
For any $i,j,h$, the directed path beginning at height $h$ in the first column of $\gamma_{i,j}$ finishes at height $h$ in the last column of $\gamma_{i,j}$.
\end{lemma}

Let $n=3b+a$ with $0\leq a < 3$, $b\geq 2$.
For $0\leq t \leq b-2$, we define $\gamma_{0,j}(t)$ as the directed graph $\gamma_{0,j}$ on the parts $G_{3t-1},G_{3t},G_{3t+1},G_{3t+2}$, with calculations done in $\mathbb{Z}_n$. We define $\gamma_{a,j}(n)$ as the directed graph $\gamma_{a,j}$ on the parts $G_{3b-4},G_{3b-3},\ldots,G_{3b+a-1}$.
Let $\Gamma(j)=\bigoplus_{t=0}^{b-2}\gamma_{0,j}(t)\oplus \gamma_{a,j}(n)$. Notice that $\gamma_{0,j}(0)$ is on the parts $G_{-1}, G_{0}, G_{1}$ and $G_{2}$. This means that $F_n(\Gamma(j))$ is the matching between the first and second columns in $\gamma_{0,j}(0)$.
\begin{example}
Let $n=7$. We will construct $\Gamma(0)$. Since $n=6+1$, we have $b=2$ and $a=1$. This means that $\Gamma(0)=\gamma_{0,0}(0)\oplus\gamma_{1,0}(7)$, where $\gamma_{0,0}(0)$ is on the parts $G_{-1}=G_{7-1}=G_6,G_0,G_1$ and $G_2$; and $\gamma_{1,0}(7)$ is on the parts $G_2,G_3,G_4,G_5$ and $G_6$. So we get the picture in Figure \ref{Gamma0a}, where the arcs from $\gamma_{0,0}(0)$ are dashed.
\begin{figure}
\begin{center}
\begin{minipage}[b]{.5\linewidth}
\begin{tikzpicture}[every node/.style={draw,shape=circle,fill=black}, every path/.style={->}]
\draw (0,0) node [scale=.5] (000){};
\draw (1,0) node [scale=.5] (010){};
\draw (2,0) node [scale=.5] (020){};
\draw (3,0) node [scale=.5] (030){};
\draw (0,-.5) node [scale=.5](001){};
\draw (1,-.5) node [scale=.5](011){};
\draw (2,-.5) node [scale=.5](021){};
\draw (3,-.5) node [scale=.5](031){};
\draw (0,-1) node [scale=.5](002){};
\draw (1,-1) node [scale=.5](012){};
\draw (2,-1) node [scale=.5](022){};
\draw (3,-1) node [scale=.5](032){};
\draw (0,-1.5) node [scale=.5](003){};
\draw (1,-1.5) node [scale=.5](013){};
\draw (2,-1.5) node [scale=.5](023){};
\draw (3,-1.5) node [scale=.5](033){};

\draw (000)[dashed] to (010);
\draw (010)[dashed] to (020);
\draw (020)[dashed] to (030);
\draw (001)[dashed] to (013);
\draw (013) [dashed]to (022);
\draw (022)[dashed] to (031);
\draw (002)[dashed] to (011);
\draw (011)[dashed] to (023);
\draw (023)[dashed] to (032);
\draw (003)[dashed] to (012);
\draw (012)[dashed] to (021);
\draw (021)[dashed] to (033);

\draw (4,0) node [scale=.5] (110){};
\draw (5,0) node [scale=.5] (120){};
\draw (6,0) node [scale=.5] (130){};
\draw (7,0) node [scale=.5] (140){};
\draw (4,-.5) node [scale=.5](111){};
\draw (5,-.5) node [scale=.5](121){};
\draw (6,-.5) node [scale=.5](131){};
\draw (7,-.5) node [scale=.5](141){};
\draw (4,-1) node [scale=.5](112){};
\draw (5,-1) node [scale=.5](122){};
\draw (6,-1) node [scale=.5](132){};
\draw (7,-1) node [scale=.5](142){};
\draw (4,-1.5) node [scale=.5](113){};
\draw (5,-1.5) node [scale=.5](123){};
\draw (6,-1.5) node [scale=.5](133){};
\draw (7,-1.5) node [scale=.5](143){};

\draw (030) to (110);
\draw (110) to (120);
\draw (120) to (130);
\draw (130) to (140);
\draw (031) to (113);
\draw (113) to (122);
\draw (122) to (133);
\draw (133) to (141);
\draw (032) to (111);
\draw (111) to (123);
\draw (123) to (131);
\draw (131) to (142);
\draw (033) to (112);
\draw (112) to (121);
\draw (121) to (132);
\draw (132) to (143);

\node[above=10pt,fill=none,draw=none] at (0,0) {$G_6$};
\node[above=10pt,fill=none,draw=none] at (1,0) {$G_0$};
\node[above=10pt,fill=none,draw=none] at (2,0) {$G_1$};
\node[above=10pt,fill=none,draw=none] at (3,0) {$G_2$};
\node[above=10pt,fill=none,draw=none] at (4,0) {$G_3$};
\node[above=10pt,fill=none,draw=none] at (5,0) {$G_4$};
\node[above=10pt,fill=none,draw=none] at (6,0) {$G_5$};
\node[above=10pt,fill=none,draw=none] at (7,0) {$G_6$};

\end{tikzpicture}
\subcaption{}\label{Gamma0a}
\end{minipage}

\vspace{.5in}

\begin{minipage}[b]{.5\linewidth}
\begin{tikzpicture}[every node/.style={draw,shape=circle,fill=black}, every path/.style={->}]

\draw (1,0) node [scale=.5] (010){};
\draw (2,0) node [scale=.5] (020){};
\draw (3,0) node [scale=.5] (030){};

\draw (1,-.5) node [scale=.5](011){};
\draw (2,-.5) node [scale=.5](021){};
\draw (3,-.5) node [scale=.5](031){};

\draw (1,-1) node [scale=.5](012){};
\draw (2,-1) node [scale=.5](022){};
\draw (3,-1) node [scale=.5](032){};

\draw (1,-1.5) node [scale=.5](013){};
\draw (2,-1.5) node [scale=.5](023){};
\draw (3,-1.5) node [scale=.5](033){};

\draw (4,0) node [scale=.5] (110){};
\draw (5,0) node [scale=.5] (120){};
\draw (6,0) node [scale=.5] (130){};
\draw (7,0) node [scale=.5] (140){};
\draw (4,-.5) node [scale=.5](111){};
\draw (5,-.5) node [scale=.5](121){};
\draw (6,-.5) node [scale=.5](131){};
\draw (7,-.5) node [scale=.5](141){};
\draw (4,-1) node [scale=.5](112){};
\draw (5,-1) node [scale=.5](122){};
\draw (6,-1) node [scale=.5](132){};
\draw (7,-1) node [scale=.5](142){};
\draw (4,-1.5) node [scale=.5](113){};
\draw (5,-1.5) node [scale=.5](123){};
\draw (6,-1.5) node [scale=.5](133){};
\draw (7,-1.5) node [scale=.5](143){};

\draw (140)[dashed, bend right=15] to (010);
\draw (010)[dashed] to (020);
\draw (020)[dashed] to (030);
\draw (141)[dashed, bend left=35] to (013);
\draw (013) [dashed]to (022);
\draw (022)[dashed] to (031);
\draw (142)[dashed, bend right=17] to (011);
\draw (011)[dashed] to (023);
\draw (023)[dashed] to (032);
\draw[dashed] (143) .. controls +(135:1) and +(-45:1) .. (012);
\draw (012)[dashed] to (021);
\draw (021)[dashed] to (033);

\draw (030) to (110);
\draw (110) to (120);
\draw (120) to (130);
\draw (130) to (140);
\draw (031) to (113);
\draw (113) to (122);
\draw (122) to (133);
\draw (133) to (141);
\draw (032) to (111);
\draw (111) to (123);
\draw (123) to (131);
\draw (131) to (142);
\draw (033) to (112);
\draw (112) to (121);
\draw (121) to (132);
\draw (132) to (143);

\node[above=10pt,fill=none,draw=none] at (1,0) {$G_0$};
\node[above=10pt,fill=none,draw=none] at (2,0) {$G_1$};
\node[above=10pt,fill=none,draw=none] at (3,0) {$G_2$};
\node[above=10pt,fill=none,draw=none] at (4,0) {$G_3$};
\node[above=10pt,fill=none,draw=none] at (5,0) {$G_4$};
\node[above=10pt,fill=none,draw=none] at (6,0) {$G_5$};
\node[above=10pt,fill=none,draw=none] at (7,0) {$G_6$};

\end{tikzpicture}
\subcaption{}
\label{Gamma0b}
\end{minipage}
\caption{}
\label{Gamma0}
\end{center}
\end{figure}

Notice that the first and last columns are both $G_6$. Because the directed graph $\overrightarrow{C}_{(4:n)}$ has $G_0$ as the first column, we connect the vertices from $G_0$ to the last column instead, obtaining the picture in Figure \ref{Gamma0b}; where the dashed arcs still belong to $\gamma_{0,0}(0)$.

\end{example}

Using Lemma \ref{fixedheight} we have the following result.
\begin{lemma}\label{GisforGamma}
\[
\overrightarrow{C}_{(4:n)}=\Gamma(0)\oplus \Gamma(1) \oplus \Gamma(2) \oplus \Gamma(3)=\bigoplus_{j=0}^3 \Gamma(j)
\]
is a $\overrightarrow{C}_n$-factorization of $\overrightarrow{C}_{(4:n)}$.
\end{lemma}
\begin{proof}
It is easy to verify from the pictures that for any given $0\leq i\leq 2$, the directed graphs $\gamma_{i,0},\gamma_{i,1},\gamma_{i,2}$ and $\gamma_{i,3}$ are arc disjoint.
By Lemma \ref{fixedheight}, in each $\gamma_{i,j}$ the directed paths start and end at the same height. Thus when we connect all the directed paths in each factor $\Gamma(j)$, we obtain four directed cycles of length $n$.
Therefore $\bigoplus_{j=0}^3 \Gamma(j)$ is a $\overrightarrow{C}_n$-factorization of $\overrightarrow{C}_{(4:n)}$.
\end{proof}

To construct directed cycles of size $2n$, we perform switches on the edges between columns. Define $\lambda_{i,j}=\gamma_{0,i}(0)\oplus F_n(\gamma_{0,i}(0))\oplus F_n(\gamma_{0,j}(0))$. Keep in mind that $\gamma_{0,i}(0)$ is on the parts $G_{n-1},G_0,G_1,G_2$, and so $F_n(\gamma_{0,i}(0))$ only consists of the edges between parts $G_{n-1}$ and $G_0$.

\begin{lemma}\label{2heights}
If the directed path that starts at height $h_1$ in part $G_{n-1}$ in $\lambda_{i,j}$ ends at height $h_2$ in $G_2$, then the directed path that starts at height $h_2$ in $G_{n-1}$ ends at height $h_1$ in $G_2$. Even more, if $i\neq j$ then no directed path starts and ends at the same height.
\end{lemma}
\begin{proof}
We build tables that show for each possible combination of $i$ and $j$, the starting and ending heights of the directed paths $\lambda_{i,j}$. We have one table for each $i$, with the rows indexed by the options for $j$, and the columns indexed by the options for the starting height of each directed path. The entry in the table gives the finishing height.
\[
i=0\quad
\begin{array}{|c|c|c|c|c|}
\hline
&\text{height $0$} & \text{height $1$}  & \text{height $2$} & \text{height $3$}\\
\hline
 j=0 & 0 & 1 & 2 & 3\\
 \hline
j=1 & 3 & 2 & 1 & 0\\
 \hline
j=2 & 1 & 0 & 3 & 2\\
 \hline
j=3 & 2 & 3 & 0 &1\\
 \hline
 \end{array}
 \]
 \[
i=1\quad
\begin{array}{|c|c|c|c|c|}
\hline
&\text{height $0$} & \text{height $1$}  & \text{height $2$} & \text{height $3$}\\
\hline
 j=0 & 3 & 2 & 1 & 0\\
 \hline
j=1 & 0 & 1 & 2 & 3\\
 \hline
j=2 & 2 & 3 & 0 & 1\\
 \hline
j=3 & 1 & 0 & 3 &2\\
 \hline
 \end{array}
 \]
 \[
i=2\quad
\begin{array}{|c|c|c|c|c|}
\hline
&\text{height $0$} & \text{height $1$}  & \text{height $2$} & \text{height $3$}\\
\hline
 j=0 & 1& 0 & 3 & 2\\
 \hline
j=1 & 2 & 3 & 0 & 1\\
 \hline
j=2 & 0 & 1 & 2 & 3\\
 \hline
j=3 & 3 & 2 & 1 &0\\
 \hline
 \end{array}
 \]
 \[
i=3\quad
\begin{array}{|c|c|c|c|c|}
\hline
&\text{height $0$} & \text{height $1$}  & \text{height $2$} & \text{height $3$}\\
\hline
 j=0 & 2 & 3 & 0 & 1\\
 \hline
j=1 & 1 & 0 & 3 & 2\\
 \hline
j=2 & 3 & 2 & 1 & 0\\
 \hline
j=3 & 0 & 1 & 2 & 3\\
 \hline
 \end{array}
 \]

Notice that whenever $i=j$ we have $\lambda_{i,i}=\gamma_{0,i}(0)$, in which case we already know that the starting and ending heights of each directed path are the same. When $i\neq j$ the starting and ending heights are never the same, but if the starting height in $\lambda_{i,j}$ is $h_1$ and the ending height is $h_2$, then the directed path with starting height $h_2$ has ending height $h_1$.
Therefore the result is proven.
\end{proof}

Let $\Lambda(i,j)=\Gamma(i)\oplus F_n(\Gamma(i))\oplus F_n(\Gamma(j))$.
\begin{lemma}\label{LisforLambda}
If $i\neq j$, 
$\Lambda(i,j)$ is a $\overrightarrow{C}_{2n}$-factor.
\end{lemma}
\begin{proof}
Notice that 
\begin{align*}
\Lambda(i,j)&=\Gamma(i)\oplus F_n(\Gamma(i))\oplus F_n(\Gamma(j))\\
&=\left(\bigoplus_{t=0}^{b-2}\gamma_{0,i}(t)\right)\oplus \gamma_{a,i}(n)\oplus F_n(\Gamma(i))\oplus F_n(\Gamma(j))
\end{align*}

Because $F_n(\Gamma(i))$ is the matching between the first and second columns in $\gamma_{0,i}(0)$, we have $F_n(\Gamma(i))=F_n(\gamma_{0,i}(0))$. Therefore,

\begin{align*}
\Lambda(i,j)&=\left(\bigoplus_{t=0}^{b-2}\gamma_{0,i}(t)\right)\oplus \gamma_{a,i}(n)\oplus F_n(\Gamma(i))\oplus F_n(\Gamma(j))\\
&=\left(\bigoplus_{t=0}^{b-2}\gamma_{0,i}(t)\right)\oplus \gamma_{a,i}(n)\oplus F_n(\gamma_{0,i}(0))\oplus F_n(\gamma_{0,j}(0))\\
&=\left(\bigoplus_{t=1}^{b-2}\gamma_{0,i}(t)\right)\oplus \gamma_{a,i}(n)\oplus\left( \gamma_{0,i}(0)\oplus F_n(\gamma_{0,i}(0))\oplus F_n(\gamma_{0,j}(0))\right)\\
&=\left(\bigoplus_{t=1}^{b-2}\gamma_{0,i}(t)\right)\oplus \gamma_{a,i}(n)\oplus \lambda_{i,j}\\
\end{align*}

Consider the directed cycle that contains the vertex at height $h_1$ in $G_{n-1}$.
From Lemma \ref{2heights} we know that in $\lambda_{i,j}$ the directed path that starts at height $h_1$ in $G_{n-1}$ finishes at height $h_2$ in $G_2$. By Lemma \ref{fixedheight}, the directed paths through all the $\gamma_{i,j}(l)$, with $l\in\{1,2,\ldots, b-1,n\}$ start and end at the same heights. So when we reach $G_{n-1}$ again, it is at at height $h_2$. We leave $G_2$ at height $h_1$ this time, and as we move through all the $\gamma_{i,j}(l)$, with $l\in\{1,2,\ldots, b-1,n\}$, the heights never change. Therefore, we reach $G_{n-1}$ again at height $h_1$, closing the directed cycle.
This produces one directed cycle of size $2n$. By repeating the process with the directed cycle starting at one of the vertices that we have not used yet, we get the second directed cycle. 
Therefore $\Lambda(i,j)$ consists of two directed cycles of length $2n$. 
\end{proof}
Notice that if $i=j$, then $\Lambda(i,j)=\Gamma(i)$ consists of $4$ directed cycles of length $n$.
\begin{theorem}\label{even}
If $s\in \{0,2,3,4\}$, then $\overrightarrow{C}_{(4:n)}$ can be decomposed into $s$ $\overrightarrow{C}_{2n}$-factors and $4-s$ $\overrightarrow{C}_{n}$-factors.
\end{theorem}
\begin{proof}
Let $\pi$ be a permutation of the set $\{0,1,2,3\}$ with exactly $4-s$ fixed points. Then 
\[
\overrightarrow{C}_{(4:n)}=\bigoplus_{j=0}^3 \Gamma(j)=\bigoplus_{j=0}^3 \left(\Gamma(j)\oplus F_n(\Gamma(j))\oplus F_n(\Gamma(\pi(j))\right)=\bigoplus_{j=0}^3 \Lambda(j,\pi(j))
\]

Since $\Lambda(j,\pi(j))$ is a $\overrightarrow{C}_{2n}$-factor if $j\neq \pi(j)$ and a $\overrightarrow{C}_n$-factor otherwise, the theorem is proven.
\end{proof}
\begin{remark}
Notice that if $n=5$, we have $b=1$ and $a=2$. This means we have $\Gamma(j)=\gamma_{a,j}(5)=\gamma_{2,j}(5)$, which is on the parts $G_{-1}=G_4,G_0,G_1,G_2,G_3,G_4$. This will actually close the directed cycle. The results given in Lemmas \ref{2heights} and \ref{LisforLambda} only apply to $b\geq 2$, but it can be shown that the same results are true with $b=1$ by applying similar techniques on $\gamma_{a,j}$ instead of $\gamma_{0,j}$.
\end{remark}

There is one more basic decomposition that we will use, based on the resolvable gregarious decomposition of $K_{(w:n)}$ from \cite{BHR}. We make use of the constructions given in Lemma $3.1$ and Corollary $3.2$ of \cite{BHR}, and apply them to $\overrightarrow{C}_{(w:n)}$ instead of $K_{(w:n)}$.

\begin{definition}
A quasigroup $(Q,*)$ is a set $Q$ with a binary operation $*$ such that for each $a$ and $b$ in $Q$, there exist unique elements $x$ and $y$ in $Q$ such that:
\begin{itemize}
\item $a*x=b$;
\item $y*a=b$.
\end{itemize}
\end{definition}
\begin{definition}
Two quasigroups on the same set $(Q,*)$, $(Q,\circ)$ are said to be orthogonal if $i*j\neq i\circ j$ for every $i,j$ in $Q$.
\end{definition}
The reader may be familiar with Latin Squares, which are the multiplication tables of quasigroups, and mutually orthogonal Latin Squares, which are the multiplication tables of orthogonal quasigroups.
In \cite{BS}, \cite{BSP} it was shown that if $|Q|\not\in\{1,2,6\}$ then there are at least $2$ orthogonal quasigroups on $Q$.
Again, the decomposition in the following theorem is obtained by modifying the construction from Lemma $3.2$ in \cite{BHR} to work with $\overrightarrow{C}_{(w:n)}$ instead of $K_{(w:n)}$:
\begin{theorem}\label{gregarious}
Let $w\not\in\{2,6\}$ and $n$ odd. Then there is a decomposition of $\overrightarrow{C}_{(w:n)}$ into $\overrightarrow{C}_n$-factors.
\end{theorem}
\begin{proof}
Since $w\not\in\{2,6\}$ there exist two orthogonal quasigroups (Latin Squares) $(Q,\circ)$ and $(Q,*)$ of order $w$, with $Q=\{0,1,2,\dots,w-1\}$. We take directed cycles of the form:
\[
(i,0)(j,1)(i,2)(j,3)\ldots(i,n-3)(j,n-2)(k,n-1),\quad \text{ where }\quad
0\leq i,j \leq w-1, k=i\circ j.
\]
This produces a decomposition of $\overrightarrow{C}_{(w:n)}$ into $w^2$ directed cycles of size $n$.
To form a $\overrightarrow{C}_n$-factor, given $l\in Q$ we take all cycles arising from the pairs $i,j$ with $i*j=l$ in the second quasigroup $(Q,*)$. Thus we have a decomposition of $\overrightarrow{C}_{(w:n)}$ into $w$ $\overrightarrow{C}_n$-factors.
\end{proof}

\section{Multivariable functions}

\begin{definition}\label{definition71}
Let $x$ and $y$ be odd. We define $T_{(xy)}(i,\alpha)$ to be the directed subgraph of $\overrightarrow{C}_{(xy:n)}$ obtained by taking $T_{(xy)}(i,\alpha)=T_{x}(i)\otimes   T_y(\alpha)$. We also define 
\[
H_{(xy)}(i,\alpha)(j,\beta)=T_{(xy)}(i,\alpha)\oplus F_n(T_{(xy)}(i,\alpha))\oplus F_n(T_{(xy)}(j,\beta))\]
This means that $H_{(xy)}(i,\alpha)(j,\beta)$ is the directed graph obtained by taking the arcs of $T_{(xy)}(i,\alpha)$ between parts $t$ and $t+1$ for $0\leq t\leq n-2$, and the arcs between parts $n-1$ and $0$ from $T_{(xy)}(j,\beta)$.
\end{definition}
\begin{example}\label{examplemultipartite}
Figure \ref{figdef71} illustrates the first part of Definition \ref{definition71} by showing $T_{x}(i)$, $T_y(\alpha)$ and $T_{(xy)}(i,\alpha)$, for $x=3$, $y=5$, $i=1$ and $\alpha=2$, with $3$ partite sets.

Figure \ref{figlemma72} illustrates the second part of Definition \ref{definition71} by showing $H_{(xy)}(i,\alpha)(j,\beta)$, for $x=3$, $y=5$, $i=1$, $\alpha=2$, $j=2$, $\beta=4$, with $3$ partite sets.  Figure \ref{figlemma72} also shows $H_x(i,j)$ and $H_y(\alpha,\beta)$, to illustrate Lemma \ref{lemma72}.

Notice that in both figures instead of giving all the coordinates in each vertex, we give the first two coordinates of all the vertices in each row (the third coordinate would specify which partite set the vertex belongs to).
\end{example}
\begin{lemma}\label{lemma72}
Let $x$, $y$ and $n$ be odd. Then:
\[
H_{(xy)}(i,\alpha)(j,\beta)=H_x(i,j)\otimes   H_y(\alpha,\beta)
\]
\end{lemma}
\begin{proof}
Notice that 
\[
F_n(T_{(xy)}(i,\alpha))=F_n(T_x(i)\otimes   T_y(\alpha))=F_n(T_x(i))\otimes   F_n(T_y(\alpha))
\]

Notice also that 
\[
F_n(T_x(i)\otimes   T_y(\alpha))=F_n(T_x(i))\otimes  T_y(\alpha)=T_x(i)\otimes   F_n(T_y(\alpha))
\]

Then we have 
\begin{align*}
H_x(i,j)\otimes   H_y(\alpha,\beta)&=\left(T_x(i)\oplus F_n(T_x(i))\oplus F_n(T_x(j))\right) \otimes \left(T_y(\alpha)\oplus F_n(T_y(\alpha))\oplus F_n(T_y(\beta))\right)\\
&=T_x(i)\otimes T_y(\alpha)\oplus F_n(T_x(i))\otimes F_n(T_y(\alpha))\oplus F_n(T_x(j))\otimes F_n(T_y(\beta))\\
&=T_{(xy)}(i,\alpha)\oplus F_n(T_{(xy)}(i,\alpha))\oplus F_n( T_{(xy)}(j,\beta))\\
&=H_{(xy)}(i,\alpha)(j,\beta)
\end{align*}

\end{proof}

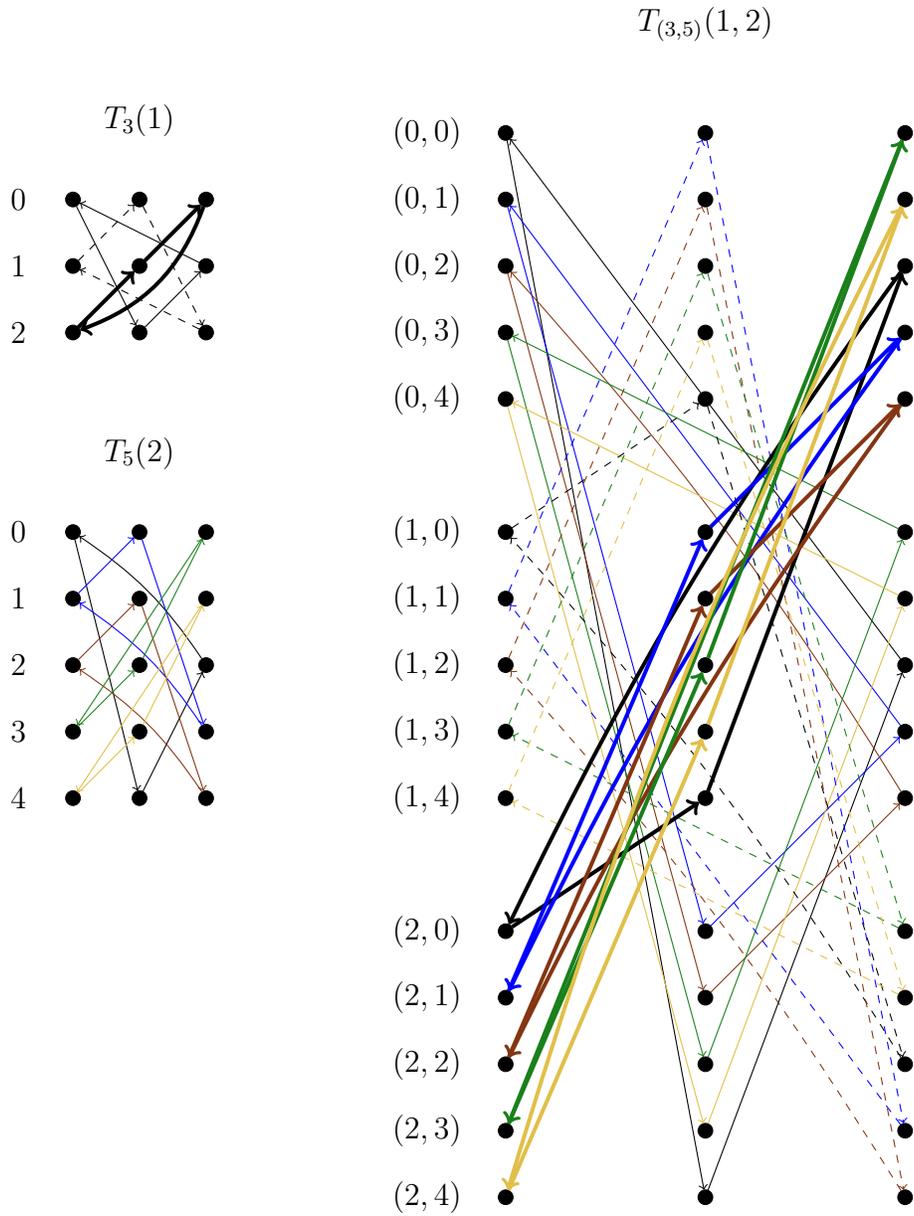
\begin{figure}
\begin{center}
\begin{tikzpicture}[every node/.style={draw,shape=circle,fill=black},scale=.885,every path/.style={->}]
\draw (-5,-7) node [scale=.5](1){};
\draw (-4,-7) node [scale=.5](2){};
\draw (-3,-7) node [scale=.5](3){};
\draw (-5,-8) node[scale=.5] (4){};
\draw (-4,-8) node [scale=.5](5){};
\draw (-3,-8) node [scale=.5](6){};
\draw (-5,-9) node[scale=.5] (7){};
\draw (-4,-9) node [scale=.5](8){};
\draw (-3,-9) node [scale=.5](9){};

\draw (1) to (8);
\draw (8) to (6);
\draw (6) to (1);
\draw (4) [dashed] to (2);
\draw (2) [dashed] to (9);
\draw (9) [dashed] to (4);
\draw (7) [line width=1.5] to (5);
\draw (5) [line width=1.5] to (3);
\draw (3) [bend left=25,line width=1.5] to (7);

\draw (-5,-12) node [scale=.5](a1){};
\draw (-4,-12) node [scale=.5](a2){};
\draw (-3,-12) node [scale=.5](a3){};
\draw (-5,-13) node[scale=.5] (a4){};
\draw (-4,-13) node [scale=.5](a5){};
\draw (-3,-13) node [scale=.5](a6){};
\draw (-5,-14) node[scale=.5] (a7){};
\draw (-4,-14) node [scale=.5](a8){};
\draw (-3,-14) node [scale=.5](a9){};
\draw (-5,-15) node[scale=.5] (a10){};
\draw (-4,-15) node [scale=.5](a11){};
\draw (-3,-15) node [scale=.5](a12){};
\draw (-5,-16) node[scale=.5] (a13){};
\draw (-4,-16) node [scale=.5](a14){};
\draw (-3,-16) node [scale=.5](a15){};

\draw (a1)  to (a14);
\draw (a14) to (a9);
\draw (a9) [bend right=11] to (a1);

\draw (a4) [blue] to (a2);
\draw (a2) [blue] to (a12);
\draw (a12) [blue,bend right=11] to (a4);

\draw (a7) [Castano] to (a5);
\draw (a5) [Castano] to (a15);
\draw (a15) [Castano, bend right=11] to (a7);

\draw (a10) [Green] to (a8);
\draw (a8) [Green] to (a3);
\draw (a3) [Green] to (a10);

\draw (a13) [Amarillo] to (a11);
\draw (a11) [Amarillo] to (a6);
\draw (a6) [Amarillo] to (a13);

\node[above=10pt,fill=none,draw=none] at (-4,-7) {$T_3(1)$};
\node[above=10pt,fill=none,draw=none] at (-4,-12) {$T_5(2)$};

\node[left=10pt,fill=none,draw=none] at (-5,-7) {$0$};
\node[left=10pt,fill=none,draw=none] at (-5,-12) {$0$};
\node[left=10pt,fill=none,draw=none] at (-5,-8) {$1$};
\node[left=10pt,fill=none,draw=none] at (-5,-13) {$1$};
\node[left=10pt,fill=none,draw=none] at (-5,-9) {$2$};
\node[left=10pt,fill=none,draw=none] at (-5,-14) {$2$};
\node[left=10pt,fill=none,draw=none] at (-5,-15) {$3$};
\node[left=10pt,fill=none,draw=none] at (-5,-16) {$4$};

\draw (1.5,-6) node [scale=.5](11){};
\draw (4.5,-6) node [scale=.5](22){};
\draw (7.5,-6) node [scale=.5](33){};

\draw (1.5,-7) node [scale=.5](14){};
\draw (4.5,-7) node [scale=.5](25){};
\draw (7.5,-7) node [scale=.5](36){};

\draw (1.5,-8) node [scale=.5](17){};
\draw (4.5,-8) node [scale=.5](28){};
\draw (7.5,-8) node [scale=.5](39){};

\draw (1.5,-9) node [scale=.5](110){};
\draw (4.5,-9) node [scale=.5](211){};
\draw (7.5,-9) node [scale=.5](312){};

\draw (1.5,-10) node [scale=.5](113){};
\draw (4.5,-10) node [scale=.5](214){};
\draw (7.5,-10) node [scale=.5](315){};

\draw (1.5,-12) node [scale=.5](41){};
\draw (4.5,-12) node [scale=.5](52){};
\draw (7.5,-12) node [scale=.5](63){};

\draw (1.5,-13) node [scale=.5](44){};
\draw (4.5,-13) node [scale=.5](55){};
\draw (7.5,-13) node [scale=.5](66){};

\draw (1.5,-14) node [scale=.5](47){};
\draw (4.5,-14) node [scale=.5](58){};
\draw (7.5,-14) node [scale=.5](69){};

\draw (1.5,-15) node [scale=.5](410){};
\draw (4.5,-15) node [scale=.5](511){};
\draw (7.5,-15) node [scale=.5](612){};

\draw (1.5,-16) node [scale=.5](413){};
\draw (4.5,-16) node [scale=.5](514){};
\draw (7.5,-16) node [scale=.5](615){};

\draw (1.5,-18) node [scale=.5](71){};
\draw (4.5,-18) node [scale=.5](82){};
\draw (7.5,-18) node [scale=.5](93){};

\draw (1.5,-19) node [scale=.5](74){};
\draw (4.5,-19) node [scale=.5](85){};
\draw (7.5,-19) node [scale=.5](96){};

\draw (1.5,-20) node [scale=.5](77){};
\draw (4.5,-20) node [scale=.5](88){};
\draw (7.5,-20) node [scale=.5](99){};

\draw (1.5,-21) node [scale=.5](710){};
\draw (4.5,-21) node [scale=.5](811){};
\draw (7.5,-21) node [scale=.5](912){};

\draw (1.5,-22) node [scale=.5](713){};
\draw (4.5,-22) node [scale=.5](814){};
\draw (7.5,-22) node [scale=.5](915){};

\node[above=10pt,fill=none,draw=none] at (4.5,-6) {$T_{(3,5)}(1,2)$};

\node[left=10pt,fill=none,draw=none] at (1.5,-6) {$(0,0)$};
\node[left=10pt,fill=none,draw=none] at (1.5,-7) {$(0,1)$};
\node[left=10pt,fill=none,draw=none] at (1.5,-8) {$(0,2)$};
\node[left=10pt,fill=none,draw=none] at (1.5,-9) {$(0,3)$};
\node[left=10pt,fill=none,draw=none] at (1.5,-10) {$(0,4)$};
\node[left=10pt,fill=none,draw=none] at (1.5,-12) {$(1,0)$};
\node[left=10pt,fill=none,draw=none] at (1.5,-13) {$(1,1)$};
\node[left=10pt,fill=none,draw=none] at (1.5,-14) {$(1,2)$};
\node[left=10pt,fill=none,draw=none] at (1.5,-15) {$(1,3)$};
\node[left=10pt,fill=none,draw=none] at (1.5,-16) {$(1,4)$};
\node[left=10pt,fill=none,draw=none] at (1.5,-18) {$(2,0)$};
\node[left=10pt,fill=none,draw=none] at (1.5,-19) {$(2,1)$};
\node[left=10pt,fill=none,draw=none] at (1.5,-20) {$(2,2)$};
\node[left=10pt,fill=none,draw=none] at (1.5,-21) {$(2,3)$};
\node[left=10pt,fill=none,draw=none] at (1.5,-22) {$(2,4)$};

\draw (11) to (814);
\draw (814) to (69);
\draw (69) to (11);

\draw (14) [blue] to (82);
\draw (82) [blue] to (612);
\draw (612) [blue] to (14);

\draw (17) [Castano] to (85);
\draw (85) [Castano] to (615);
\draw (615) [Castano,bend right=5] to (17);

\draw (110) [Green] to (88);
\draw (88) [Green] to (63);
\draw (63) [Green] to (110);

\draw (113) [Amarillo] to (811);
\draw (811) [Amarillo] to (66);
\draw (66) [Amarillo] to (113);

\draw (41) [dashed] to (214);
\draw (214)[dashed] to (99);
\draw (99) [dashed,bend right=5]to (41);

\draw (44)[dashed,blue] to (22);
\draw (22)[dashed,blue] to (912);
\draw (912)[dashed,blue] to (44);

\draw (47) [dashed,Castano]to (25);
\draw (25) [dashed,Castano]to (915);
\draw (915) [dashed,Castano,bend right=5]to (47);

\draw (410)[dashed,Green] to (28);
\draw (28) [dashed,Green]to (93);
\draw (93) [dashed,Green]to (410);

\draw (413)[dashed,Amarillo] to (211);
\draw (211)[dashed,Amarillo] to (96);
\draw (96) [dashed,Amarillo]to (413);

\draw (71)[line width=1.5] to (514);
\draw (514)[line width=1.5] to (39);
\draw (39) [line width=1.5,bend right=5]to (71);
\draw (74) [line width=1.5,blue]to (52);
\draw (52) [line width=1.5,blue]to (312);
\draw (312) [line width=1.5,blue,bend right=5]to (74);
\draw (77) [line width=1.5,Castano]to (55);
\draw (55) [line width=1.5,Castano]to (315);
\draw (315) [line width=1.5,Castano,bend right=5]to (77);
\draw (710) [line width=1.5,Green]to (58);
\draw (58) [line width=1.5,Green]to (33);
\draw (33) [line width=1.5,Green]to (710);
\draw (713) [line width=1.5,Amarillo]to (511);
\draw (511) [line width=1.5,Amarillo]to (36);
\draw (36) [line width=1.5,Amarillo,bend right=5]to (713);

\end{tikzpicture}

\caption{$T_3(1)$, $T_5(2)$ and $T_{(3,5)}(1,2)$}
\label{figdef71}
\end{center}
\end{figure}
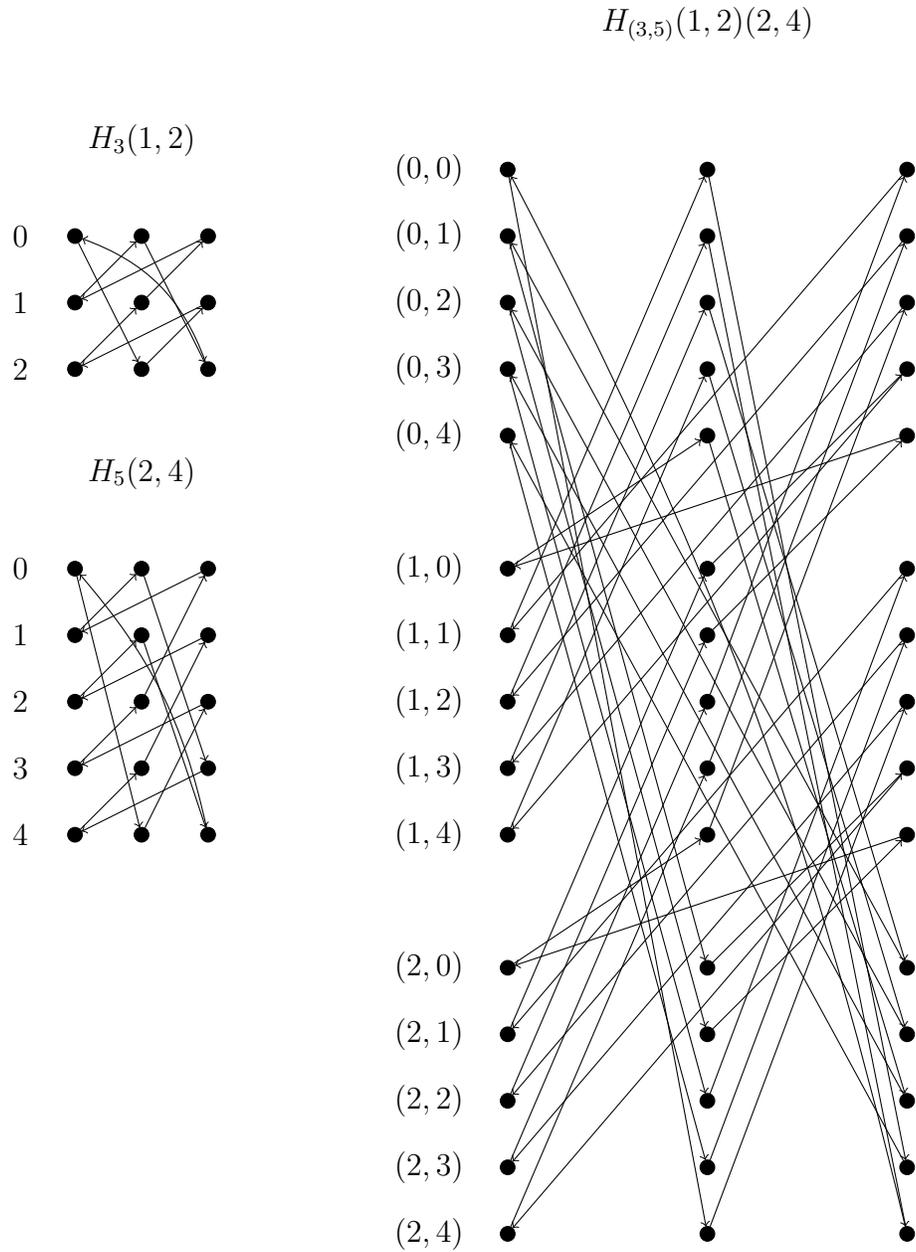
\begin{figure}
\begin{center}
\begin{tikzpicture}[every node/.style={draw,shape=circle,fill=black},scale=.885,every path/.style={->}]
\draw (-5,-7) node [scale=.5](1){};
\draw (-4,-7) node [scale=.5](2){};
\draw (-3,-7) node [scale=.5](3){};
\draw (-5,-8) node[scale=.5] (4){};
\draw (-4,-8) node [scale=.5](5){};
\draw (-3,-8) node [scale=.5](6){};
\draw (-5,-9) node[scale=.5] (7){};
\draw (-4,-9) node [scale=.5](8){};
\draw (-3,-9) node [scale=.5](9){};

\draw (1) to (8);
\draw (8) to (6);
\draw (6) to (7);
\draw (4) to (2);
\draw (2) to (9);
\draw (9) to [bend right=25] (1);
\draw (7) to (5);
\draw (5) to (3);
\draw (3) to (4);

\draw (-5,-12) node [scale=.5](a1){};
\draw (-4,-12) node [scale=.5](a2){};
\draw (-3,-12) node [scale=.5](a3){};
\draw (-5,-13) node[scale=.5] (a4){};
\draw (-4,-13) node [scale=.5](a5){};
\draw (-3,-13) node [scale=.5](a6){};
\draw (-5,-14) node[scale=.5] (a7){};
\draw (-4,-14) node [scale=.5](a8){};
\draw (-3,-14) node [scale=.5](a9){};
\draw (-5,-15) node[scale=.5] (a10){};
\draw (-4,-15) node [scale=.5](a11){};
\draw (-3,-15) node [scale=.5](a12){};
\draw (-5,-16) node[scale=.5] (a13){};
\draw (-4,-16) node [scale=.5](a14){};
\draw (-3,-16) node [scale=.5](a15){};

\draw (a1) to (a14);
\draw (a14) to (a9);
\draw (a9) to (a10);
\draw (a4) to (a2);
\draw (a2) to (a12);
\draw (a12) to (a13);

\draw (a7) to (a5);
\draw (a5) to (a15);
\draw (a15) to [bend right=14] (a1);
\draw (a10) to (a8);
\draw (a8) to (a3);
\draw (a3) to (a4);
\draw (a13) to (a11);
\draw (a11) to (a6);
\draw (a6) to (a7);

\node[above=10pt,fill=none,draw=none] at (-4,-7) {$H_3(1,2)$};
\node[above=10pt,fill=none,draw=none] at (-4,-12) {$H_5(2,4)$};

\node[left=10pt,fill=none,draw=none] at (-5,-7) {$0$};
\node[left=10pt,fill=none,draw=none] at (-5,-12) {$0$};
\node[left=10pt,fill=none,draw=none] at (-5,-8) {$1$};
\node[left=10pt,fill=none,draw=none] at (-5,-13) {$1$};
\node[left=10pt,fill=none,draw=none] at (-5,-9) {$2$};
\node[left=10pt,fill=none,draw=none] at (-5,-14) {$2$};
\node[left=10pt,fill=none,draw=none] at (-5,-15) {$3$};
\node[left=10pt,fill=none,draw=none] at (-5,-16) {$4$};

\draw (1.5,-6) node [scale=.5](11){};
\draw (4.5,-6) node [scale=.5](22){};
\draw (7.5,-6) node [scale=.5](33){};

\draw (1.5,-7) node [scale=.5](14){};
\draw (4.5,-7) node [scale=.5](25){};
\draw (7.5,-7) node [scale=.5](36){};

\draw (1.5,-8) node [scale=.5](17){};
\draw (4.5,-8) node [scale=.5](28){};
\draw (7.5,-8) node [scale=.5](39){};

\draw (1.5,-9) node [scale=.5](110){};
\draw (4.5,-9) node [scale=.5](211){};
\draw (7.5,-9) node [scale=.5](312){};

\draw (1.5,-10) node [scale=.5](113){};
\draw (4.5,-10) node [scale=.5](214){};
\draw (7.5,-10) node [scale=.5](315){};

\draw (1.5,-12) node [scale=.5](41){};
\draw (4.5,-12) node [scale=.5](52){};
\draw (7.5,-12) node [scale=.5](63){};

\draw (1.5,-13) node [scale=.5](44){};
\draw (4.5,-13) node [scale=.5](55){};
\draw (7.5,-13) node [scale=.5](66){};

\draw (1.5,-14) node [scale=.5](47){};
\draw (4.5,-14) node [scale=.5](58){};
\draw (7.5,-14) node [scale=.5](69){};

\draw (1.5,-15) node [scale=.5](410){};
\draw (4.5,-15) node [scale=.5](511){};
\draw (7.5,-15) node [scale=.5](612){};

\draw (1.5,-16) node [scale=.5](413){};
\draw (4.5,-16) node [scale=.5](514){};
\draw (7.5,-16) node [scale=.5](615){};

\draw (1.5,-18) node [scale=.5](71){};
\draw (4.5,-18) node [scale=.5](82){};
\draw (7.5,-18) node [scale=.5](93){};

\draw (1.5,-19) node [scale=.5](74){};
\draw (4.5,-19) node [scale=.5](85){};
\draw (7.5,-19) node [scale=.5](96){};

\draw (1.5,-20) node [scale=.5](77){};
\draw (4.5,-20) node [scale=.5](88){};
\draw (7.5,-20) node [scale=.5](99){};

\draw (1.5,-21) node [scale=.5](710){};
\draw (4.5,-21) node [scale=.5](811){};
\draw (7.5,-21) node [scale=.5](912){};

\draw (1.5,-22) node [scale=.5](713){};
\draw (4.5,-22) node [scale=.5](814){};
\draw (7.5,-22) node [scale=.5](915){};

\node[above=10pt,fill=none,draw=none] at (4.5,-6) {$H_{(3,5)}(1,2)(2,4)$};

\node[left=10pt,fill=none,draw=none] at (1.5,-6) {$(0,0)$};
\node[left=10pt,fill=none,draw=none] at (1.5,-7) {$(0,1)$};
\node[left=10pt,fill=none,draw=none] at (1.5,-8) {$(0,2)$};
\node[left=10pt,fill=none,draw=none] at (1.5,-9) {$(0,3)$};
\node[left=10pt,fill=none,draw=none] at (1.5,-10) {$(0,4)$};
\node[left=10pt,fill=none,draw=none] at (1.5,-12) {$(1,0)$};
\node[left=10pt,fill=none,draw=none] at (1.5,-13) {$(1,1)$};
\node[left=10pt,fill=none,draw=none] at (1.5,-14) {$(1,2)$};
\node[left=10pt,fill=none,draw=none] at (1.5,-15) {$(1,3)$};
\node[left=10pt,fill=none,draw=none] at (1.5,-16) {$(1,4)$};
\node[left=10pt,fill=none,draw=none] at (1.5,-18) {$(2,0)$};
\node[left=10pt,fill=none,draw=none] at (1.5,-19) {$(2,1)$};
\node[left=10pt,fill=none,draw=none] at (1.5,-20) {$(2,2)$};
\node[left=10pt,fill=none,draw=none] at (1.5,-21) {$(2,3)$};
\node[left=10pt,fill=none,draw=none] at (1.5,-22) {$(2,4)$};

\draw (11) to (814);
\draw (814) to (69);
\draw (69) to (710);
\draw (14) to (82);
\draw (82) to (612);
\draw (612) to (713);

\draw (17) to (85);
\draw (85) to (615);
\draw (615) to (71);
\draw (110) to (88);
\draw (88) to (63);
\draw (63) to (74);
\draw (113) to (811);
\draw (811) to (66);
\draw (66) to (77);

\draw (41) to (214);
\draw (214) to (99);
\draw (99) to (110);
\draw (44) to (22);
\draw (22) to (912);
\draw (912) to (113);

\draw (47) to (25);
\draw (25) to (915);
\draw (915) to [bend right=7](11);
\draw (410) to (28);
\draw (28) to (93);
\draw (93) to (14);
\draw (413) to (211);
\draw (211) to (96);
\draw (96) to (17);

\draw (71) to (514);
\draw (514) to (39);
\draw (39) to (410);
\draw (74) to (52);
\draw (52) to (312);
\draw (312) to (413);
\draw (77) to (55);
\draw (55) to (315);
\draw (315) to (41);
\draw (710) to (58);
\draw (58) to (33);
\draw (33) to (44);
\draw (713) to (511);
\draw (511) to (36);
\draw (36) to (47);

\end{tikzpicture}

\caption{$H_3(1,2)$, $H_5(2,4)$ and $H_{(3,5)}(1,2)(2,4)$}
\label{figlemma72}
\end{center}
\end{figure}

\begin{lemma}
Let $\psi$ be a bijection on the set $\{(i,\alpha)|0\leq i \leq x-1, 0\leq \alpha \leq y-1\}$. Then
\[
\overrightarrow{C}_{(xy:n)}=\bigoplus_{(i,\alpha)}H_{(xy)}(i,\alpha)\psi(i,\alpha).\]
\end{lemma}
\begin{proof}
We know that $\overrightarrow{C}_{(xy:n)}=\overrightarrow{C}_{(x:n)}\otimes   \overrightarrow{C}_{(y:n)}=\left(\bigoplus_i T_{x}(i)\right)\otimes  \left(\bigoplus_{\alpha}T_{y}(\alpha)\right)$. By definition of $T_{(xy)}(i,\alpha)$ we get 
\[
\overrightarrow{C}_{(xy:n)}=\bigoplus_{(i,\alpha)}T_{(x,y)}(i,\alpha)\]

We also have
\[
\bigoplus_{(i,\alpha)}T_{(xy)}(i,\alpha)=\bigoplus_{(i,\alpha)}H_{(xy)}(i,\alpha)\psi(i,\alpha)
\]

Combining both we get:
\[
\overrightarrow{C}_{(xy:n)}=\bigoplus_{(i,\alpha)}H_{(xy)}(i,\alpha)\psi(i,\alpha)
\]
as we wanted to prove.
\end{proof}

If $\psi(i,\alpha)=(j,\beta)$ we will denote $\psi_1(i,\alpha)=j$ and $\psi_2(i,\alpha)=\beta$.
If $\gcd(x,i-j)=1$ and $\alpha=\beta$, then $H_{(xy)}(i,\alpha)(j,\beta)$ is a $\overrightarrow{C}_{xn}$-factor. This is because
\[
H_{(xy)}(i,\alpha)(j,\beta)=H_x(i,j)\otimes   H_y(\alpha,\alpha)=H_x(i,j)\otimes   T_y(\alpha)
\]
By Lemma \ref{HisforHamilton} $H_x(i,j)$ is a $\overrightarrow{C}_{xn}$-factor. By Lemma \ref{bbb} $T_y(\alpha)$ is a $\overrightarrow{C}_n$-factor. Then by Lemma \ref{productofbalanced} $H_x(i,j)\otimes   T_y(\alpha)$ is a $\overrightarrow{C}_{xn}$-factor.
Thus to obtain a decomposition of $\overrightarrow{C}_{(xy:n)}$ into $\overrightarrow{C}_{xn}$-factors and $\overrightarrow{C}_{yn}$-factors we need a bijection $\psi$ that satisfies the following set of conditions
\begin{conditions}\label{doubleCheese}
\begin{enumerate}[a)]
\item For all $(i,\alpha)$, $\gcd(x,i-\psi_1(i,\alpha))=1$ and $\psi_2(i,\alpha)=\alpha$, or
\item $\gcd(y,\alpha-\psi_2(i,\alpha))=1$ and $\psi_1(i,\alpha)=i$.
\end{enumerate}
\end{conditions}

\begin{lemma}\label{xnyn}
Let $x$, $y$, and $n$ be odd. Let $s_p\neq 1,xy-1$. Then there is a decomposition of $\overrightarrow{C}_{(xy:n)}$ into $s_p$ $\overrightarrow{C}_{xn}$-factors and $r_p=xy-s_p$ $\overrightarrow{C}_{yn}$-factors.
\end{lemma}
\begin{proof}
We will describe a bijection $\psi$ that satisfies conditions \ref{doubleCheese} with $r_p$ pairs $(i,\alpha)$ that satisfy $i=\psi_1(i,\alpha)$.

Let $r_\alpha$, $0\leq \alpha\leq y-1$ be such that:
\begin{itemize}
\item $\sum_\alpha r_\alpha=r_p$,
\item $r_i\geq r_j$ if $i\leq j$,
\item $r_{0}=r_{1}$,
\item $0\leq r_{\alpha}\leq x$, $r_\alpha\neq x-1$. 
\end{itemize}
Define the function $\phi_\alpha (i)=\phi_s(i)$, with $\phi_s(i)$ as the phi-function over the set $\{0,1,\ldots, x-1\}$ with $s=x-r_\alpha$.
Let $\pi(i)=\left|\{\alpha|\phi_\alpha(i)=i\}\right|$.
Let $\sigma_i(\alpha)=\phi_s(\alpha)$, with $\phi_s(\alpha)$ as the phi-function over the set $\{0,1,\ldots, y-1\}$ with $s=\pi(i)$.
Notice that
\[
\psi(i,\alpha)=(\phi_{\alpha}(i),\sigma_i(\alpha))
\]
is a function satisfying conditions \ref{doubleCheese} because if $\alpha\leq \pi(i)$, then $\psi_{\alpha}(i)=i$ and $\gcd(y,\alpha-\sigma_i(\alpha))=1$. If, on the other hand $\alpha \geq \pi(i)$, then we have $\sigma_i(\alpha)=\alpha$, and $\gcd(x,i-\psi_{\alpha}(i))=1$. Finally, notice that there are $r_p$ pairs $(i,\alpha)$ that satisfy $i=\varphi_1(i,\alpha)$. Therefore there is a decomposition of $\overrightarrow{C}_{(xy:n)}$ into $s_p$ $\overrightarrow{C}_{xn}$-factors and $r_p=xy-s_p$ $\overrightarrow{C}_{yn}$-factors.
\end{proof}

We can work with $\Gamma(i)$ and $\Lambda(i)$ in a similar fashion as to what we did with $T_{x}(i)$.

\begin{definition}
Let $x$ be odd. We define $T_{(2x)}(i,\alpha)$ to be the directed subgraph of $\overrightarrow{C}_{(4x:n)}$ obtained by taking $T_{(xy)}(i,\alpha)=T_x(i)\otimes   \Gamma(\alpha)$.
We also define $H_{(2x)}(i,\alpha)(j,\beta)=T_{(2x)}(i,\alpha)\oplus F_n(T_{(2x)}(i,\alpha))\oplus F_n(T_{(2x)}(j,\beta))$. This is the directed graph obtained by taking the arcs of 
$T_{(2x)}(i,\alpha)$ between parts 
$t$ and $t+1$ for $0\leq t\leq n-2$, and the arcs between parts $n-1$ and $0$ from $T_{(2x)}(j,\beta)$.
\end{definition}

Now we can apply the same techniques that we did to $T_{(xy)}$ and $H_{(xy)}$.
\begin{lemma}
Let $x$ and $n$ be odd. Then:
\[
H_{(2x)}(i,\alpha)(j,\beta)=H_x(i,j)\otimes   \Lambda(\alpha,\beta)
\]
\end{lemma}
\begin{proof}
Notice that 
\[
F_n(T_{(2x)}(i,\alpha))=F_n(T_x(i)\otimes   \Gamma(\alpha))=F_n(T_x(i))\otimes   F_n(\Gamma(\alpha))
\]
and
\[
F_n(T_x(i)\otimes \Gamma(\alpha))=F_n(T_x(i))\otimes \Gamma(\alpha)=T_x(i)\otimes F_n(\Gamma(\alpha))
\]
Using this the result is trivial.
\end{proof}

\begin{lemma}
Let $\varphi$ by a bijection on the set $\{(i,\alpha)|0\leq i \leq x-1, 0\leq \alpha \leq 3\}$. Then
\[
\overrightarrow{C}_{(4x:n)}=\bigoplus_{(i,\alpha)}H_{(2x)}(i,\alpha)\varphi(i,\alpha).\]
\end{lemma}
\begin{proof}
We know that $\overrightarrow{C}_{(4x:n)}=\overrightarrow{C}_{(x:n)}\otimes   \overrightarrow{C}_{(4:n)}=\left(\bigoplus_i T_x(i)\right)\otimes  \left(\bigoplus_{\alpha}\Gamma(\alpha)\right)$. By definition of $T_{(2x)}(i,\alpha)$ we get 
\[
\overrightarrow{C}_{(4x:n)}=\bigoplus_{(i,\alpha)}T_{(2x)}(i,\alpha)\]

We also have
\[
\bigoplus_{(i,\alpha)}T_{(2x)}(i,\alpha)=\bigoplus_{(i,\alpha)}H_{(2x)}(i,\alpha)\varphi(i,\alpha)
\]

Combining both we get:
\[
\overrightarrow{C}_{(4x:n)}=\bigoplus_{(i,\alpha)}H_{(2x)}(i,\alpha)\varphi(i,\alpha)
\]
as we wanted to prove.
\end{proof}

Next we develop the conditions needed for our decompositions. Recall that if 
$\varphi(i,\alpha)=(j,\beta)$ we will denote $\varphi_1(i,\alpha)=j$ and $\varphi_2(i,\alpha)=\beta$.
\begin{itemize}
\item If $\alpha\neq \beta$ and $\gcd(x,i-j)=1$ then 
\[
H_{(2x)}(i,\alpha)(j,\beta)=H_x(i,j)\otimes   \Lambda(\alpha,\beta)
\]
is a $\overrightarrow{C}_{2xn}$-factor by Lemmas \ref{HisforHamilton}, \ref{LisforLambda}, and \ref{productofbalanced}. 

\item If $i=j$ and $\alpha\neq \beta$, then $H_{(2x)}(i,\alpha)(j,\beta)$ is a $\overrightarrow{C}_{2n}$-factor by Lemmas \ref{bbb}, \ref{LisforLambda}, and \ref{productofbalanced}.

\item If $\alpha=\beta$ and $\gcd(x,i-j)=1$, then $H_{(2x)}(i,\alpha)(j,\beta)$ is a $\overrightarrow{C}_{xn}$-factor by Lemmas \ref{HisforHamilton}, \ref{GisforGamma}, and \ref{productofbalanced}.

\item If $i=j$ and $\alpha=\beta$, then $H_{(xy)}(i,\alpha)(j,\beta)$ is a $\overrightarrow{C}_{n}$-factor by Lemmas \ref{bbb}, \ref{GisforGamma}, and \ref{productofbalanced}.
\end{itemize}

So for a decomposition of $\overrightarrow{C}_{(4x:n)}$ into $\overrightarrow{C}_{2xn}$-factors and $\overrightarrow{C}_{n}$-factors we need a bijection $\varphi$ that satisfies:
\begin{conditions}\label{2HamandTrig}
For all $(i,\alpha)$ such that $\varphi(i,\alpha)\neq (i,\alpha)$, $\alpha\neq \varphi_2(i,\alpha)$ and $\gcd(x,i-\varphi_1(i,\alpha))=1$.
\end{conditions}

For a decomposition of $\overrightarrow{C}_{(4x:n)}$ into $\overrightarrow{C}_{2n}$-factors and $\overrightarrow{C}_{xn}$-factors we need a bijection $\varphi$ that satisfies:
\begin{conditions}\label{2doubleCheese}
\begin{enumerate}[a)]
\item For all $(i,\alpha)$ either $\alpha\neq \varphi_2(i,\alpha)$ and $i=j$, or
\item $\alpha=\varphi_2(i,\alpha)$ and $\gcd(x,i-\varphi_1(i,\alpha))= 1$.
\end{enumerate}
\end{conditions}

We define a new family of functions $\theta_s:\mathbb{Z}_x\times\mathbb{Z}_4\rightarrow \mathbb{Z}_x\times\mathbb{Z}_4$. These functions will be referred as \textit{theta-functions}. 
Let $s\in \{0,1,\ldots,4x\}$, $s\notin \{1,4x-1\}$, and write $s=4k+2a+3b$, with $a,b\in\{0,1\}$, $k\leq x$. We define:
\begin{center}
\textbf{Theta-functions}:
\[
\theta_s(i,\alpha)\coloneqq\left\lbrace\begin{array}{ccc}
(i+1,\alpha+1)&\text{ if }& 0\leq i \leq k-1, \alpha= 0,2\\
(i-1,\alpha-1)&\text{ if }& 1\leq i \leq k, \alpha= 1,3 \\
(i+1,\alpha+1)&\text{ if }&  i=k,\, a=1,\, \alpha=0\\
(i-1,\alpha-1)&\text{ if }&  i=k+1,\, a=1,\, \alpha=1\\
(i+1,\alpha+1)&\text{ if }&  i=k,\, b=1,\, \alpha=2\\
(i+1,\alpha-2)&\text{ if }&  i=k+1,\, b=1,\, \alpha=3\\
(i-2,\alpha+1)&\text{ if }&  i=k+2,\, b=1,\, \alpha=1\\
(i,\alpha)& &\text{ otherwise}\\
\end{array}\right.
\]
\end{center}
If $s=4x-1=4(x-1)+3$, we define $\theta_s$ in a similar way, with a small change:
\[
\theta_{4x-1}(i,\alpha)\coloneqq\left\lbrace\begin{array}{ccc}
(i+1,\alpha+1)&\text{ if }& 0\leq i \leq x-2, \alpha= 0,2 \\
(i-1,\alpha-1)&\text{ if }& 1\leq i \leq x-2, \alpha= 1,3 \\
(x-2,2)&\text{ if }&  i=x-1,\, \alpha=1\\
(0,1)&\text{ if }&  i=x-1,\, \alpha=3\\
(x-1,0)&\text{ if }&  i=0,\, \alpha=1\\
(0,3)&\text{ if }&  i=x-1,\, \alpha=0\\
(x-2,0)&\text{ if }&  i=0,\, \alpha=3\\
(i,\alpha)& &\text{ otherwise}\\
\end{array}\right.
\]

We give a visual example of $\theta_9$ and $\theta_{19}$, for $x=5$ in Figure \ref{thetaexample}.

\begin{figure}
\begin{center}
\begin{tikzpicture}[every node/.style={shape=circle},outer sep=0pt, inner sep=0pt]
\draw (0,0) node [scale=.7](00){$(0,0)$};
\draw (1,5) node [scale=.7](01){$(0,1)$};
\draw (2,0) node [scale=.7](02){$(0,2)$};
\draw (3,5) node [scale=.7](03){$(0,3)$};
\draw (0,1) node [scale=.7](10){$(1,0)$};
\draw (1,1) node [scale=.7](11){$(1,1)$};
\draw (2,1) node [scale=.7](12){$(1,2)$};
\draw (3,1) node [scale=.7](13){$(1,3)$};
\draw (0,2) node [scale=.7](20){$(2,0)$};
\draw (1,2) node [scale=.7](21){$(2,1)$};
\draw (2,2) node [scale=.7](22){$(2,2)$};
\draw (3,2) node [scale=.7](23){$(2,3)$};
\draw (0,3) node [scale=.7](30){$(3,0)$};
\draw (1,3) node [scale=.7](31){$(3,1)$};
\draw (2,3) node [scale=.7](32){$(3,2)$};
\draw (3,3) node [scale=.7](33){$(3,3)$};
\draw (0,4) node [scale=.7](40){$(4,0)$};
\draw (1,4) node [scale=.7](41){$(4,1)$};
\draw (2,4) node [scale=.7](42){$(4,2)$};
\draw (3,4) node [scale=.7](43){$(4,3)$};

\draw (00) [->,thick] to (11);
\draw (11) [->,thick,bend right=15] to (00);
\draw (02) [->,thick] to (13);
\draw (13) [->,thick,bend right=15] to (02);
\draw (10) [->,thick] to (21);
\draw (21) [->,thick,bend right=15] to (10);
\draw (12) [->,thick] to (23);
\draw (23) [->,thick,bend right=10] to (31);
\draw (31) [->,thick] to (12);
\draw (22) [->,out=90,in=135,looseness=4] to (22);
\draw (30) [->,out=90,in=135,looseness=4] to (30);
\draw (32) [->,out=90,in=135,looseness=4] to (32);
\draw (33) [->,out=90,in=135,looseness=4] to (33);
\draw (40) [->,out=90,in=135,looseness=4] to (40);
\draw (41) [->,out=90,in=135,looseness=4] to (41);
\draw (42) [->,out=90,in=135,looseness=4] to (42);
\draw (43) [->,out=90,in=135,looseness=4] to (43);
\draw (01) [->,out=90,in=135,looseness=4] to (01);
\draw (03) [->,out=90,in=135,looseness=4] to (03);
\draw (20) [->,out=90,in=135,looseness=4] to (20);

\draw (8,0) node [scale=.7](100){$(0,0)$};
\draw (9,5) node [scale=.7](101){$(0,1)$};
\draw (10,0) node [scale=.7](102){$(0,2)$};
\draw (11,5) node [scale=.7](103){$(0,3)$};
\draw (8,1) node [scale=.7](110){$(1,0)$};
\draw (9,1) node [scale=.7](111){$(1,1)$};
\draw (10,1) node [scale=.7](112){$(1,2)$};
\draw (11,1) node [scale=.7](113){$(1,3)$};
\draw (8,2) node [scale=.7](120){$(2,0)$};
\draw (9,2) node [scale=.7](121){$(2,1)$};
\draw (10,2) node [scale=.7](122){$(2,2)$};
\draw (11,2) node [scale=.7](123){$(2,3)$};
\draw (8,3) node [scale=.7](130){$(3,0)$};
\draw (9,3) node [scale=.7](131){$(3,1)$};
\draw (10,3) node [scale=.7](132){$(3,2)$};
\draw (11,3) node [scale=.7](133){$(3,3)$};
\draw (8,4) node [scale=.7](140){$(4,0)$};
\draw (9,4) node [scale=.7](141){$(4,1)$};
\draw (10,4) node [scale=.7](142){$(4,2)$};
\draw (11,4) node [scale=.7](143){$(4,3)$};
\draw (100) [->,thick] to (111);
\draw (111) [->,thick,bend right=15] to (100);
\draw (102) [->,thick] to (113);
\draw (113) [->,thick,bend right=15] to (102);
\draw (110) [->,thick] to (121);
\draw (121) [->,thick,bend right=15] to (110);
\draw (112) [->,thick] to (123);
\draw (123) [->,thick,bend right=15] to (112);

\draw (120) [->,thick] to (131);
\draw (131) [->,thick,bend right=15] to (120);
\draw (122) [->,thick] to (133);
\draw (133) [->,thick,bend right=15] to (122);
\draw (120) [->,thick] to (131);
\draw (131) [->,thick,bend right=15] to (120);
\draw (122) [->,thick] to (133);
\draw (133) [->,thick,bend right=15] to (122);

\draw (130) [->,thick] to (141);
\draw (141) [->,thick] to (132);
\draw (132) [->,thick] to (143);
\draw (143) [->,thick] to (101);
\draw (101) [->,thick] to (140);
\draw (140) [->,thick] to (103);
\draw (103) [->,thick,bend right=35,dotted] to (130);
\draw (142) [->,out=90,in=135,looseness=4] to (142);

\node[below=10pt,fill=none,draw=none] at (1.5,2) (aaaa){$x=5$, $s=9$, $k=1,a=1,b=1$};
\node[right=190pt,fill=none,draw=none] at (aaaa) {$x=5$, $s=19$};
\end{tikzpicture}
\vspace*{-1in}

\caption{Example of Theta-functions}
\label{thetaexample}
\end{center}
\end{figure}

The following lemma is a generalization of a result given in $\cite{AKKPO}$. 

\begin{lemma}\label{2xnandntheorem}
Let $s_p\in \{0,2,\ldots,4x-1,4x\}$, $x$ odd. Then there exists a decomposition of $C_{(4x:n)}$ into $s_p$ $C_{2xn}$-factors and $r_p=4x-s_p$ $C_n$-factors.
\end{lemma}
\begin{proof}
The bijection $\psi=\theta_{s_p}$ satisfies Conditions $7.9$. In particular if $\psi(i,\alpha)\neq (i,\alpha)$, then $\alpha\neq \psi_{2}(i,\alpha)$ and $i-\psi_1 (i,\alpha)\in\{\pm 1, \pm 2\}$; and as $x$ is odd, $\gcd(x,i-\psi_1(i,\alpha))=1$. Furthermore, $\psi$ has $s_p$ non-fixed points.
Therefore there exists a decomposition of $C_{(4x:n)}$ into $s_p$ $C_{2xn}$-factors and $4x-s_p$ $C_n$-factors.
\end{proof}

\begin{lemma}\label{xnand2ntheorem}
Let $s_p\in \{0,2,3,\ldots,4x-3,4x-2,4x\}$. Then there exists a decomposition of $\overrightarrow{C}_{(4x:n)}$ into $s_p$ $\overrightarrow{C}_{xn}$-factors and $r_p=4x-s_p$ $\overrightarrow{C}_{2n}$-factors.
\end{lemma}
\begin{proof}
We provide a bijection $\varphi$ that satisfies Conditions \ref{2doubleCheese} with $r_p$ pairs $(i,\alpha)$ that satisfy $i=\varphi_1(i,\alpha)$.

Let $r_\alpha$, $0\leq \alpha\leq 3$ be such that:
\begin{itemize}
\item $\sum_\alpha r_\alpha=r_p$,
\item $r_i\geq r_j$ if $i\leq j$, 
\item $r_0=r_1$, 
\item $r_\alpha\leq x$, $r_\alpha\neq x-1$.
\end{itemize}
The only case where such a choice of $r_\alpha$ cannot be made is when $x=3$, $s=7$. This case is covered in Lemma \ref{theissuewith12} in the Appendix.

Define the function $\psi_\alpha (i)=\phi_s(i)$, with $\phi_s(i)$ as the phi-functions over the set $\{0,1,\ldots, x-1\}$ with $s=x-r_\alpha$.
Let $\pi(i)=\left|\{\alpha|\psi_\alpha(i)=i\}\right|$.
Let $\sigma_i(\alpha)$ be the permutation on the set $\{0,1,2,3\}$ that cyclically permutes the first $\pi(i)$ elements and fixes the rest.
Notice that
\[
\varphi(i,\alpha)=(\psi_{\alpha}(i),\sigma_i(\alpha))
\]
is a function satisfying conditions \ref{2doubleCheese} because if $\alpha\leq \pi(i)$, then $\psi_{\alpha}(i)=i$ and $\sigma_i(\alpha)\neq \alpha$. If, on the other hand $\alpha \geq \pi(i)$, then we have $\sigma_i(\alpha)=\alpha$, and $\gcd(x,i-\psi_{\alpha}(i))=1$. Finally, notice that there are $r_p$ pairs $(i,\alpha)$ that satisfy $i=\varphi_1(i,\alpha)$.
Therefore there exists a decomposition of $\overrightarrow{C}_{(4x:n)}$ into $s_p$ $\overrightarrow{C}_{xn}$-factors and $r_p=4x-s_p$ $\overrightarrow{C}_{2n}$-factors.
\end{proof}
We are interested in one more type of decomposition, into $\overrightarrow{C}_{2xn}$ and $\overrightarrow{C}_{yn}$ factors. To do this we introduce the following:

\begin{definition}
Let $x$ and $y$ be odd. Define $T_{(2xy)}(i,\alpha,\gamma)$ to be the directed subgraph of $\overrightarrow{C}_{(4xy:n)}$ obtained by taking $T_{(2xy)}(i,\alpha,\gamma)=T_{(2x)}(i,\alpha)\otimes   T_y(\gamma)$. We also define 
\[
H_{(2xy)}(i,\alpha,\gamma)(j,\beta,\delta)=T_{(2xy)}(i,\alpha,\gamma)\oplus F_n(T_{(2xy)}(i,\alpha,\gamma))\oplus F_n(T_{(2xy)}(j,\beta,\delta)\]
This means that $H_{(2xy)}(i,\alpha,\gamma)(j,\beta,\delta)$ is the directed graph obtained by taking the arcs of $T_{(2xy)}(i,\alpha,\gamma)$ between parts $t$ and $t+1$ for $0\leq t\leq n-2$, and the arcs between parts $n-1$ and $0$ from $T_{(2xy)}(j,\beta,\delta)$.
\end{definition}

Now we have all the usual results:

\begin{lemma}
Let $x$, $y$ and $n$ be odd. Then:
\[
H_{(2xy)}(i,\alpha,\gamma)(j,\beta,\delta)=H_{(2x)}(i,\alpha)(j,\beta)\otimes   H_y(\gamma,\delta)
\]
\end{lemma}
\begin{proof}
Notice that 
\[
F_n(T_{(2xy)}(i,\alpha,\gamma))=F_n(T_{(2x)}(i,\alpha)\otimes   T_y(\gamma))=F_n(T_{(2x)}(i,\alpha))\otimes   F_n(T_y(\gamma))
\]

Using this the result is trivial.

\end{proof}

\begin{lemma}
Let $\varphi$ be a bijection on the set $\{(i,\alpha,\gamma)|0\leq i \leq x-1, 0\leq \alpha \leq 3,0\leq \gamma \leq y-1\}$. Then
\[
\overrightarrow{C}_{(4xy:n)}=\bigoplus_{(i,\alpha,\gamma)}H_{(2xy)}(i,\alpha,\gamma)\varphi(i,\alpha,\gamma).\]
\end{lemma}
\begin{proof}
We know that $\overrightarrow{C}_{(4xy:n)}=\overrightarrow{C}_{(4x:n)}\otimes   \overrightarrow{C}_{(y:n)}=\left(\bigoplus_{(i,\alpha)} T_{(2x)}(i,\alpha)\right)\otimes  \left(\bigoplus_{\gamma}T_{y}(\gamma)\right)$. By the definition of $T_{(2xy)}(i,\alpha,\gamma)$ we get 
\[
\overrightarrow{C}_{(4xy:n)}=\bigoplus_{(i,\alpha,\gamma)}T_{(2xy)}(i,\alpha)\]

We also have
\[
\bigoplus_{(i,\alpha,\gamma)}T_{(2xy)}(i,\alpha,\gamma)=\bigoplus_{(i,\alpha,\gamma)}H_{(2xy)}(i,\alpha,\gamma)\varphi(i,\alpha,\gamma)
\]

Combining both we get:
\[
\overrightarrow{C}_{(4xy:n)}=\bigoplus_{(i,\alpha,\gamma)}H_{(2xy)}(i,\alpha,\gamma)\varphi(i,\alpha,\gamma)
\]
as we wanted to prove.
\end{proof}

We have the following properties:
\begin{itemize}
\item If $\alpha\neq \beta$, $\gamma=\delta$, and $\gcd(x,i-j)=1$, then
\begin{align*}
H_{(2xy)}(i,\alpha,\gamma)(j,\beta,\delta)&=H_{(2x)}(i,\alpha)(j,\beta)\otimes H_{y}(\gamma,\gamma)\\
&=H_x(i,j)\otimes   \Lambda(\alpha,\beta) \otimes T_{y}(\gamma)
\end{align*}
is a $\overrightarrow{C}_{2xn}$-factor by Lemmas \ref{HisforHamilton}, \ref{LisforLambda}, \ref{bbb} and \ref{productofbalanced}. 

\item If $i=j$, $\alpha=\beta$, and $\gcd(y,\gamma-\delta)=1$ then 

\begin{align*}
H_{(2xy)}(i,\alpha,\gamma)(j,\beta,\delta)&=H_{(2x)}(i,\alpha)(i,\alpha)\otimes H_{y}(\gamma,\delta)\\
&=T_x(i)\otimes   \Gamma(\alpha) \otimes H_{y}(\gamma,\delta)
\end{align*}
is a $\overrightarrow{C}_{yn}$-factor by Lemmas \ref{bbb}, \ref{GisforGamma}, \ref{HisforHamilton}, and \ref{productofbalanced}.
\end{itemize}
To get a decomposition of $\overrightarrow{C}_{(4xy:n)}$ into $\overrightarrow{C}_{2xn}$-factors and  $\overrightarrow{C}_{yn}$-factors we need a bijection $\varphi$ that satisfies:
\begin{conditions}\label{2xnandyn}
\begin{enumerate}[a)]
\item For all $(i,\alpha,\gamma)$, $\gcd(y,\gamma-\varphi_3(i,\alpha,\gamma))=1$ or $\gamma=\varphi_3(i,\alpha,\gamma)$.
\item If $\gamma=\varphi_3(i,\alpha,\gamma)$, then $\gcd(x,i-\varphi_{1}(i,\alpha,\gamma))=1$ and $\alpha\neq \varphi_{2}(i,\alpha,\gamma)$.
\item If $\gcd(y,\gamma-\varphi_3(i,\alpha,\gamma))=1$, then $i=\varphi_1(i,\alpha,\gamma)$ and $\alpha=\varphi_2(i,\alpha,\gamma)$.
\end{enumerate}
\end{conditions}

Now we can write our lemma:
\begin{lemma}\label{2xnyn}
Let $x,y$, and $n$ be odd. Let $s_p\neq 1, 4xy-1$. Then there is a decomposition of $C_{(4xy:n)}$ into $s_p$ $C_{2xn}$-factors and $r_p=4xy-s_p$ $C_{yn}$-factors.
\end{lemma}
\begin{proof}
We give a bijection $\varphi$ that satisfies Conditions $7.16$ with $r_p$ elements $(i,\alpha,\gamma)$ that satisfy $i=\varphi_1(i,\alpha,\gamma)$.
Let $s_p=4xk+q$, with $0\leq q \leq 4x-1$. We have two cases,  $k\leq y-3$, and $k\geq y-2$.

\begin{description}
\item [Case 1]
If $k\leq y-3$, let $s_p=4xk+a-\epsilon$, with $2\leq a \leq 4x-1$, $0\leq \epsilon \leq 2$. For $y-k+1\leq \gamma \leq y-1$, let $\psi_{\gamma}(i,\alpha)=\theta_{4x}(i,\alpha)$, the theta-function with $s=4x$. Let $\psi_{y-k}(i,\alpha)=\theta_{4x-\epsilon}(i,\alpha)$, the theta-function with $s=4x-\epsilon$. Let $\psi_{y-k-1}(i,\alpha)=\theta_{a}(i,\alpha)$, the theta function with $s=a$. For $0\leq \gamma \leq y-k-2$ let $\psi_{\gamma}(i,\alpha)=(i,\alpha)$, the identity function.

\item [Case 2] 
If $k\geq y-2$, let $s_p=4xk'+2a-\epsilon$, with $2\leq a \leq 4x-\epsilon$, $0\leq \epsilon \leq 4$, where $k'\in\{y-2,y-3\}$ because $2a$ may be greater than $4x$. For $y-k+1\leq \gamma \leq y-1$, let $\psi_{\gamma}(i,\alpha)=\theta_{4x}(i,\alpha)$, the theta-function with $s=4x$. Let $\psi_{y-k}(i,\alpha)=\theta_{4x-\epsilon}(i,\alpha)$, the theta-function with $s=4x-\epsilon$.
Let $\psi_{y-k-1}(i,\alpha)=\psi_{y-k-2}(i,\alpha)=\theta_{a}(i,\alpha)$, the theta function with $s=a$. For $0\leq \gamma \leq y-k-3$ let $\psi_{\gamma}(i,\alpha)=(i,\alpha)$, the identity function.
\end{description}

Notice that the fixed point of $\theta_{4x-1}$ is $(x-1,2)$, the fixed points of $\theta_{4x-2}$ are $\{(x-1,2),(0,3)\}$, the fixed points of $\theta_{4x-3}$ are $\{(x-1,2),(0,3),(x-1,0)\}$, and the fixed points of $\theta_{4x-4}$ are $\{(x-1,2),(0,3),(x-1,0),(0,1)\}$. This means that if $0\leq \epsilon\leq 4$ and $a\leq 4x-\epsilon$, the fixed points of $\theta_{4x-\epsilon}$ are a subset of the fixed points of $\theta_{a}$. Hence if $\psi_{\delta}(i,\alpha)=(i,\alpha)$, then $\psi_{\gamma}(i,\alpha)=(i,\alpha)$ for all $\gamma\leq \delta$. Notice also that $\psi_{0}(i,\alpha)=\psi_{1}(i,\alpha)$, hence $\max\{\delta \in \{0,\ldots,y-1\}|\psi_{\delta}(i,\alpha)=(i,\alpha)\}\neq 1$. Therefore we can define $\sigma_{i,\alpha}(\gamma)=\phi_s(\gamma)$, the phi-function over the set $\{0,1,\ldots,y-1\}$ with $s=\max\{\delta \in \{0,\ldots,y-1\}|\psi_{\delta}(i,\alpha)=(i,\alpha)\}$. Then:
\[
\rho(i,\alpha,\gamma)=(\psi_{\gamma}(i,\alpha),\sigma_{i,\alpha}(\gamma))
\]
is a function satisfying conditions $7.16$.
\end{proof}

\begin{example}
We provide two visual examples, with $x=3$ and $y=5$. To make the picture easier to understand, the points satisfying $i=\varphi_1(i,\alpha,\gamma)$ have been boxed and underlined, and rearranged with their images at the right side.

In Figure \ref{717examplea} we have $s_p=25=2\cdot 12+1=2\cdot 12+2 -1$, giving us $r_p=60-25=35$, $k=2$, $a=2$, $\epsilon=1$.

In Figure \ref{717exampleb} we have $s_p=37=3\cdot 12+1=3\cdot 12+2\cdot2-3$, giving us $r_p=60-37=23$, $k'=3$, $a=2$, $\epsilon=3$.
\end{example}
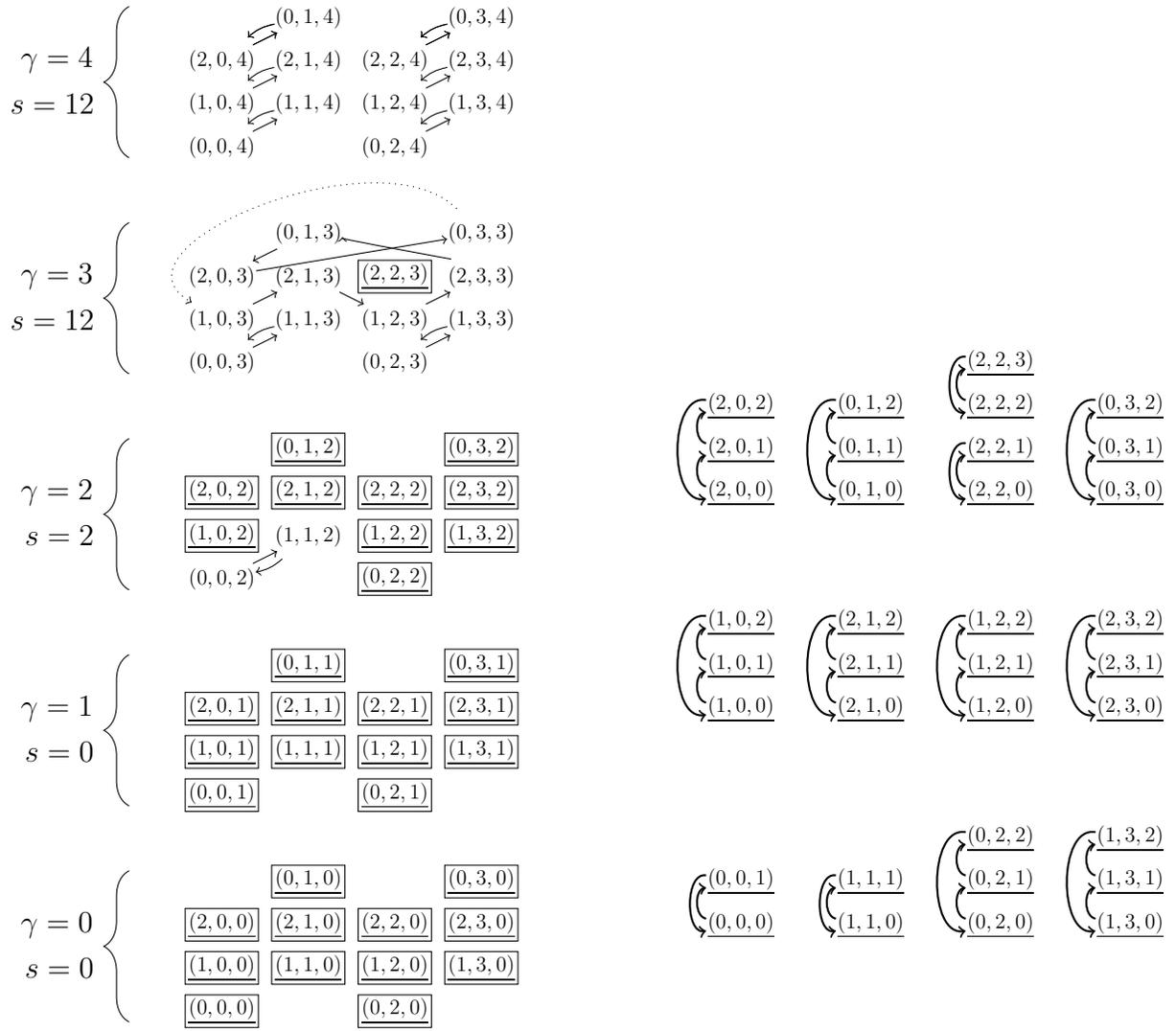
\begin{figure}
\begin{center}
\begin{tikzpicture}[scale=.6,every node/.style={shape=circle},outer sep=0pt, inner sep=0pt]
\draw (0,0) node [scale=.7,shape=rectangle,draw,inner sep=2pt](000){$\underline{(0,0,0)}$};
\draw (2,3) node [scale=.7,shape=rectangle,draw,inner sep=2pt](001){$\underline{(0,1,0)}$};
\draw (4,0) node [scale=.7,shape=rectangle,draw,inner sep=2pt](002){$\underline{(0,2,0)}$};
\draw (6,3) node [scale=.7,shape=rectangle,draw,inner sep=2pt](003){$\underline{(0,3,0)}$};
\draw (0,1) node [scale=.7,shape=rectangle,draw,inner sep=2pt](010){$\underline{(1,0,0)}$};
\draw (2,1) node [scale=.7,shape=rectangle,draw,inner sep=2pt](011){$\underline{(1,1,0)}$};
\draw (4,1) node [scale=.7,shape=rectangle,draw,inner sep=2pt](012){$\underline{(1,2,0)}$};
\draw (6,1) node [scale=.7,shape=rectangle,draw,inner sep=2pt](013){$\underline{(1,3,0)}$};
\draw (0,2) node [scale=.7,shape=rectangle,draw,inner sep=2pt](020){$\underline{(2,0,0)}$};
\draw (2,2) node [scale=.7,shape=rectangle,draw,inner sep=2pt](021){$\underline{(2,1,0)}$};
\draw (4,2) node [scale=.7,shape=rectangle,draw,inner sep=2pt](022){$\underline{(2,2,0)}$};
\draw (6,2) node [scale=.7,shape=rectangle,draw,inner sep=2pt](023){$\underline{(2,3,0)}$};

\node[left=50pt,fill=none,draw=none] at (0,2) (aaaa){$\gamma=0$};
\node[left=50pt,fill=none,draw=none] at (0,1) (aaaa){$s=0$};
\draw [decorate,decoration={brace,amplitude=10pt},xshift=-4pt,yshift=0pt]
(-2,-.25) -- (-2,3.25)node [black,midway,xshift=9pt]{};

\draw (0,5) node [scale=.7,shape=rectangle,draw,inner sep=2pt](100){$\underline{(0,0,1)}$};
\draw (2,8) node [scale=.7,shape=rectangle,draw,inner sep=2pt](101){$\underline{(0,1,1)}$};
\draw (4,5) node [scale=.7,shape=rectangle,draw,inner sep=2pt](102){$\underline{(0,2,1)}$};
\draw (6,8) node [scale=.7,shape=rectangle,draw,inner sep=2pt](103){$\underline{(0,3,1)}$};
\draw (0,6) node [scale=.7,shape=rectangle,draw,inner sep=2pt](110){$\underline{(1,0,1)}$};
\draw (2,6) node [scale=.7,shape=rectangle,draw,inner sep=2pt](111){$\underline{(1,1,1)}$};
\draw (4,6) node [scale=.7,shape=rectangle,draw,inner sep=2pt](112){$\underline{(1,2,1)}$};
\draw (6,6) node [scale=.7,shape=rectangle,draw,inner sep=2pt](113){$\underline{(1,3,1)}$};
\draw (0,7) node [scale=.7,shape=rectangle,draw,inner sep=2pt](120){$\underline{(2,0,1)}$};
\draw (2,7) node [scale=.7,shape=rectangle,draw,inner sep=2pt](121){$\underline{(2,1,1)}$};
\draw (4,7) node [scale=.7,shape=rectangle,draw,inner sep=2pt](122){$\underline{(2,2,1)}$};
\draw (6,7) node [scale=.7,shape=rectangle,draw,inner sep=2pt](123){$\underline{(2,3,1)}$};

\node[left=50pt,fill=none,draw=none] at (0,7) (aaaa){$\gamma=1$};
\node[left=50pt,fill=none,draw=none] at (0,6) (aaaa){$s=0$};
\draw [decorate,decoration={brace,amplitude=10pt},xshift=-4pt,yshift=0pt]
(-2,4.75) -- (-2,8.25)node [black,midway,xshift=9pt]{};

\draw (0,10) node [scale=.7](200){$(0,0,2)$};
\draw (2,13) node [scale=.7,shape=rectangle,draw,inner sep=2pt](201){$\underline{(0,1,2)}$};
\draw (4,10) node [scale=.7,shape=rectangle,draw,inner sep=2pt](202){$\underline{(0,2,2)}$};
\draw (6,13) node [scale=.7,shape=rectangle,draw,inner sep=2pt](203){$\underline{(0,3,2)}$};
\draw (0,11) node [scale=.7,shape=rectangle,draw,inner sep=2pt](210){$\underline{(1,0,2)}$};
\draw (2,11) node [scale=.7](211){$(1,1,2)$};
\draw (4,11) node [scale=.7,shape=rectangle,draw,inner sep=2pt](212){$\underline{(1,2,2)}$};
\draw (6,11) node [scale=.7,shape=rectangle,draw,inner sep=2pt](213){$\underline{(1,3,2)}$};
\draw (0,12) node [scale=.7,shape=rectangle,draw,inner sep=2pt](220){$\underline{(2,0,2)}$};
\draw (2,12) node [scale=.7,shape=rectangle,draw,inner sep=2pt](221){$\underline{(2,1,2)}$};
\draw (4,12) node [scale=.7,shape=rectangle,draw,inner sep=2pt](222){$\underline{(2,2,2)}$};
\draw (6,12) node [scale=.7,shape=rectangle,draw,inner sep=2pt](223){$\underline{(2,3,2)}$};

\node[left=50pt,fill=none,draw=none] at (0,12) (aaaa){$\gamma=2$};
\node[left=50pt,fill=none,draw=none] at (0,11) (aaaa){$s=2$};
\draw [decorate,decoration={brace,amplitude=10pt},xshift=-4pt,yshift=0pt]
(-2,9.75) -- (-2,13.25)node [black,midway,xshift=9pt]{};

\draw (0,15) node [scale=.7](300){$(0,0,3)$};
\draw (2,18) node [scale=.7](301){$(0,1,3)$};
\draw (4,15) node [scale=.7](302){$(0,2,3)$};
\draw (6,18) node [scale=.7](303){$(0,3,3)$};
\draw (0,16) node [scale=.7](310){$(1,0,3)$};
\draw (2,16) node [scale=.7](311){$(1,1,3)$};
\draw (4,16) node [scale=.7](312){$(1,2,3)$};
\draw (6,16) node [scale=.7](313){$(1,3,3)$};
\draw (0,17) node [scale=.7](320){$(2,0,3)$};
\draw (2,17) node [scale=.7](321){$(2,1,3)$};
\draw (4,17) node [scale=.7,shape=rectangle,draw,inner sep=2pt](322){$\underline{(2,2,3)}$};
\draw (6,17) node [scale=.7](323){$(2,3,3)$};

\node[left=50pt,fill=none,draw=none] at (0,17) (aaaa){$\gamma=3$};
\node[left=50pt,fill=none,draw=none] at (0,16) (aaaa){$s=12$};
\draw [decorate,decoration={brace,amplitude=10pt},xshift=-4pt,yshift=0pt]
(-2,14.75) -- (-2,18.25)node [black,midway,xshift=9pt]{};

\draw (0,20) node [scale=.7](400){$(0,0,4)$};
\draw (2,23) node [scale=.7](401){$(0,1,4)$};
\draw (4,20) node [scale=.7](402){$(0,2,4)$};
\draw (6,23) node [scale=.7](403){$(0,3,4)$};
\draw (0,21) node [scale=.7](410){$(1,0,4)$};
\draw (2,21) node [scale=.7](411){$(1,1,4)$};
\draw (4,21) node [scale=.7](412){$(1,2,4)$};
\draw (6,21) node [scale=.7](413){$(1,3,4)$};
\draw (0,22) node [scale=.7](420){$(2,0,4)$};
\draw (2,22) node [scale=.7](421){$(2,1,4)$};
\draw (4,22) node [scale=.7](422){$(2,2,4)$};
\draw (6,22) node [scale=.7](423){$(2,3,4)$};

\node[left=50pt,fill=none,draw=none] at (0,22) (aaaa){$\gamma=4$};
\node[left=50pt,fill=none,draw=none] at (0,21) (aaaa){$s=12$};
\draw [decorate,decoration={brace,amplitude=10pt},xshift=-4pt,yshift=0pt]
(-2,19.75) -- (-2,23.25)node [black,midway,xshift=9pt]{};

\draw (400) [->] to (411);
\draw (411) [->,bend right=15] to (400);
\draw (402) [->] to (413);
\draw (413) [->,bend right=15] to (402);
\draw (410) [->] to (421);
\draw (421) [->,bend right=15] to (410);
\draw (412) [->] to (423);
\draw (423) [->,bend right=15] to (412);

\draw (420) [->] to (401);
\draw (401) [->,bend right=15] to (420);
\draw (422) [->] to (403);
\draw (403) [->,bend right=15] to (422);
\draw (420) [->] to (401);
\draw (401) [->,bend right=15] to (420);
\draw (422) [->] to (403);
\draw (403) [->,bend right=15] to (422);

\draw (300) [->] to (311);
\draw (311) [->,bend right=15] to (300);
\draw (302) [->] to (313);
\draw (313) [->,bend right=15] to (302);

\draw (310) [->] to (321);
\draw (321) [->] to (312);
\draw (312) [->] to (323);
\draw (323) [->,out=150,in=350,looseness=0] to (301);
\draw (301) [->] to (320);
\draw (320) [->] to (303);
\draw (303) [->,out=135, in=150, looseness=1,dotted] to (310);

\draw (200) [->] to (211);
\draw (211) [->,bend left=15] to (200);

\draw (12,14) node [scale=.7](a220){$\underline{(2,0,2)}$};
\draw (12,13) node [scale=.7](a120){$\underline{(2,0,1)}$};
\draw (12,12) node [scale=.7](a020){$\underline{(2,0,0)}$};

\draw (12,9) node [scale=.7](a210){$\underline{(1,0,2)}$};
\draw (12,8) node [scale=.7](a110){$\underline{(1,0,1)}$};
\draw (12,7) node [scale=.7](a010){$\underline{(1,0,0)}$};

\draw (12,3) node [scale=.7](a100){$\underline{(0,0,1)}$};
\draw (12,2) node [scale=.7](a000){$\underline{(0,0,0)}$};

\draw (15,14) node [scale=.7](a201){$\underline{(0,1,2)}$};
\draw (15,13) node [scale=.7](a101){$\underline{(0,1,1)}$};
\draw (15,12) node [scale=.7](a001){$\underline{(0,1,0)}$};

\draw (15,9) node [scale=.7](a221){$\underline{(2,1,2)}$};
\draw (15,8) node [scale=.7](a121){$\underline{(2,1,1)}$};
\draw (15,7) node [scale=.7](a021){$\underline{(2,1,0)}$};

\draw (15,3) node [scale=.7](a111){$\underline{(1,1,1)}$};
\draw (15,2) node [scale=.7](a011){$\underline{(1,1,0)}$};

\draw (18,15) node [scale=.7](a322){$\underline{(2,2,3)}$};
\draw (18,14) node [scale=.7](a222){$\underline{(2,2,2)}$};
\draw (18,13) node [scale=.7](a122){$\underline{(2,2,1)}$};
\draw (18,12) node [scale=.7](a022){$\underline{(2,2,0)}$};

\draw (18,9) node [scale=.7](a212){$\underline{(1,2,2)}$};
\draw (18,8) node [scale=.7](a112){$\underline{(1,2,1)}$};
\draw (18,7) node [scale=.7](a012){$\underline{(1,2,0)}$};

\draw (18,4) node [scale=.7](a202){$\underline{(0,2,2)}$};
\draw (18,3) node [scale=.7](a102){$\underline{(0,2,1)}$};
\draw (18,2) node [scale=.7](a002){$\underline{(0,2,0)}$};

\draw (21,14) node [scale=.7](a203){$\underline{(0,3,2)}$};
\draw (21,13) node [scale=.7](a103){$\underline{(0,3,1)}$};
\draw (21,12) node [scale=.7](a003){$\underline{(0,3,0)}$};

\draw (21,9) node [scale=.7](a223){$\underline{(2,3,2)}$};
\draw (21,8) node [scale=.7](a123){$\underline{(2,3,1)}$};
\draw (21,7) node [scale=.7](a023){$\underline{(2,3,0)}$};

\draw (21,4) node [scale=.7](a213){$\underline{(1,3,2)}$};
\draw (21,3) node [scale=.7](a113){$\underline{(1,3,1)}$};
\draw (21,2) node [scale=.7](a013){$\underline{(1,3,0)}$};

\draw (a020) [->,out=170,in=190, looseness=1,thick] to (a120);
\draw (a120) [->,out=170,in=190, looseness=1,thick] to (a220);
\draw (a220) [->,out=170,in=190, looseness=1,thick] to (a020);

\draw (a010) [->,out=170,in=190, looseness=1,thick] to (a110);
\draw (a110) [->,out=170,in=190, looseness=1,thick] to (a210);
\draw (a210) [->,out=170,in=190, looseness=1,thick] to (a010);

\draw (a000) [->,out=170,in=190, looseness=1,thick] to (a100);
\draw (a100) [->,out=170,in=190, looseness=1,thick] to (a000);

\draw (a021) [->,out=170,in=190, looseness=1,thick] to (a121);
\draw (a121) [->,out=170,in=190, looseness=1,thick] to (a221);
\draw (a221) [->,out=170,in=190, looseness=1,thick] to (a021);

\draw (a011) [->,out=170,in=190, looseness=1,thick] to (a111);
\draw (a111) [->,out=170,in=190, looseness=1,thick] to (a011);

\draw (a001) [->,out=170,in=190, looseness=1,thick] to (a101);
\draw (a101) [->,out=170,in=190, looseness=1,thick] to (a201);
\draw (a201) [->,out=170,in=190, looseness=1,thick] to (a001);

\draw (a012) [->,out=170,in=190, looseness=1,thick] to (a112);
\draw (a112) [->,out=170,in=190, looseness=1,thick] to (a212);
\draw (a212) [->,out=170,in=190, looseness=1,thick] to (a012);

\draw (a002) [->,out=170,in=190, looseness=1,thick] to (a102);
\draw (a102) [->,out=170,in=190, looseness=1,thick] to (a202);
\draw (a202) [->,out=170,in=190, looseness=1,thick] to (a002);

\draw (a023) [->,out=170,in=190, looseness=1,thick] to (a123);
\draw (a123) [->,out=170,in=190, looseness=1,thick] to (a223);
\draw (a223) [->,out=170,in=190, looseness=1,thick] to (a023);

\draw (a013) [->,out=170,in=190, looseness=1,thick] to (a113);
\draw (a113) [->,out=170,in=190, looseness=1,thick] to (a213);
\draw (a213) [->,out=170,in=190, looseness=1,thick] to (a013);

\draw (a003) [->,out=170,in=190, looseness=1,thick] to (a103);
\draw (a103) [->,out=170,in=190, looseness=1,thick] to (a203);
\draw (a203) [->,out=170,in=190, looseness=1,thick] to (a003);

\draw (a322) [->,out=170,in=190, looseness=1,thick] to (a222);
\draw (a222) [->,out=170,in=190, looseness=1,thick] to (a322);
\draw (a122) [->,out=170,in=190, looseness=1,thick] to (a022);
\draw (a022) [->,out=170,in=190, looseness=1,thick] to (a122);

\node[below=10pt,fill=none,draw=none] at (9,5.5) (aaaa){$x=3$, $y=5$, $s_p=25$, $r_p=35$, $k=2$, $a=2$, $\epsilon=1$};
\end{tikzpicture}
\vspace*{-1in}

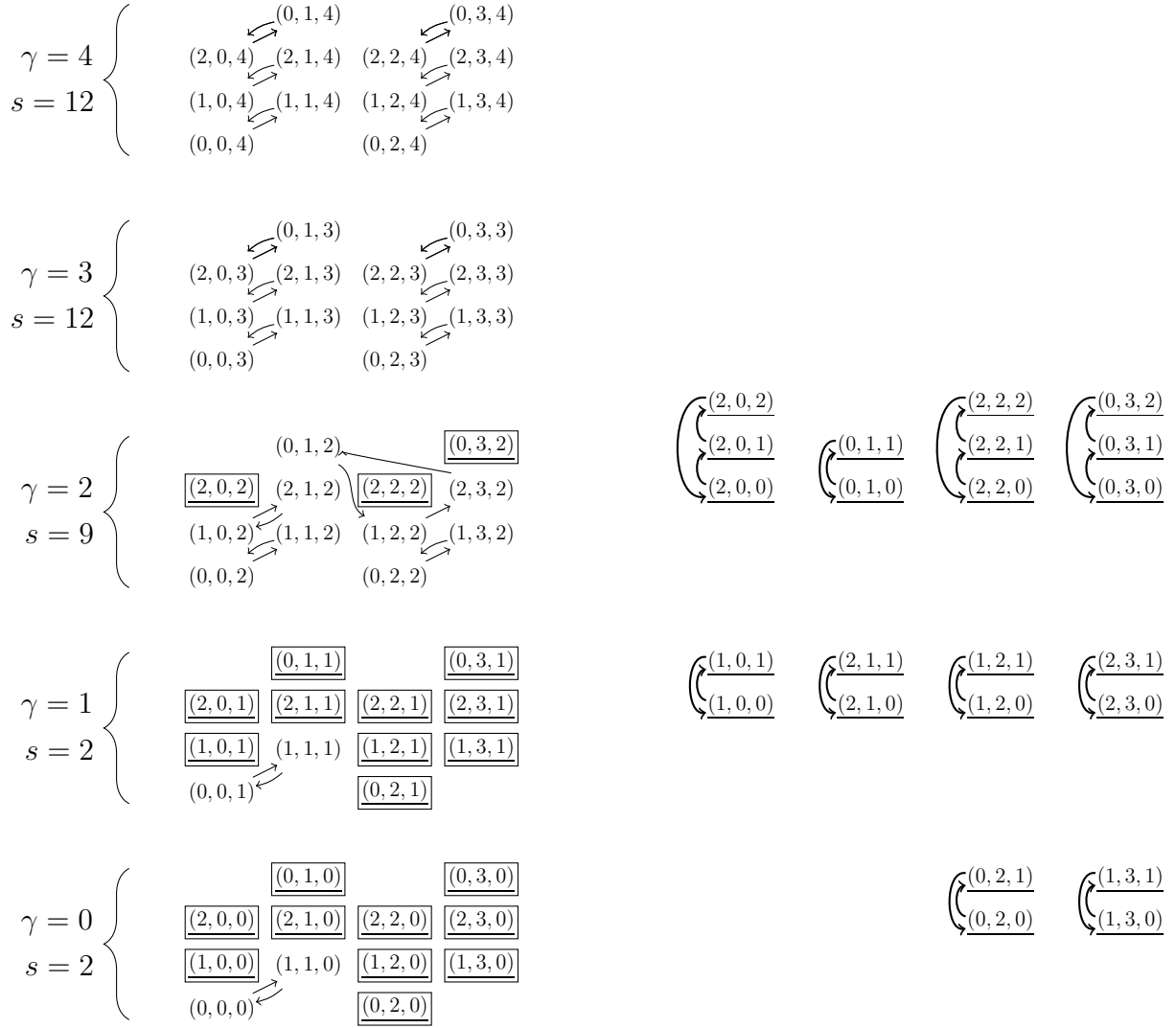
\caption{Example of Lemma $7.17$ Case 1}
\label{717examplea}
\end{center}
\end{figure}

\begin{figure}
\begin{center}
\begin{tikzpicture}[scale=.6,every node/.style={shape=circle},outer sep=0pt, inner sep=0pt]
\draw (0,0) node [scale=.7](000){$(0,0,0)$};
\draw (2,3) node [scale=.7,shape=rectangle,draw,inner sep=2pt](001){$\underline{(0,1,0)}$};
\draw (4,0) node [scale=.7,shape=rectangle,draw,inner sep=2pt](002){$\underline{(0,2,0)}$};
\draw (6,3) node [scale=.7,shape=rectangle,draw,inner sep=2pt](003){$\underline{(0,3,0)}$};
\draw (0,1) node [scale=.7,shape=rectangle,draw,inner sep=2pt](010){$\underline{(1,0,0)}$};
\draw (2,1) node [scale=.7](011){$(1,1,0)$};
\draw (4,1) node [scale=.7,shape=rectangle,draw,inner sep=2pt](012){$\underline{(1,2,0)}$};
\draw (6,1) node [scale=.7,shape=rectangle,draw,inner sep=2pt](013){$\underline{(1,3,0)}$};
\draw (0,2) node [scale=.7,shape=rectangle,draw,inner sep=2pt](020){$\underline{(2,0,0)}$};
\draw (2,2) node [scale=.7,shape=rectangle,draw,inner sep=2pt](021){$\underline{(2,1,0)}$};
\draw (4,2) node [scale=.7,shape=rectangle,draw,inner sep=2pt](022){$\underline{(2,2,0)}$};
\draw (6,2) node [scale=.7,shape=rectangle,draw,inner sep=2pt](023){$\underline{(2,3,0)}$};

\node[left=50pt,fill=none,draw=none] at (0,2) (aaaa){$\gamma=0$};
\node[left=50pt,fill=none,draw=none] at (0,1) (aaaa){$s=2$};
\draw [decorate,decoration={brace,amplitude=10pt},xshift=-4pt,yshift=0pt]
(-2,-.25) -- (-2,3.25)node [black,midway,xshift=9pt]{};

\draw (0,5) node [scale=.7](100){$(0,0,1)$};
\draw (2,8) node [scale=.7,shape=rectangle,draw,inner sep=2pt](101){$\underline{(0,1,1)}$};
\draw (4,5) node [scale=.7,shape=rectangle,draw,inner sep=2pt](102){$\underline{(0,2,1)}$};
\draw (6,8) node [scale=.7,shape=rectangle,draw,inner sep=2pt](103){$\underline{(0,3,1)}$};
\draw (0,6) node [scale=.7,shape=rectangle,draw,inner sep=2pt](110){$\underline{(1,0,1)}$};
\draw (2,6) node [scale=.7](111){$(1,1,1)$};
\draw (4,6) node [scale=.7,shape=rectangle,draw,inner sep=2pt](112){$\underline{(1,2,1)}$};
\draw (6,6) node [scale=.7,shape=rectangle,draw,inner sep=2pt](113){$\underline{(1,3,1)}$};
\draw (0,7) node [scale=.7,shape=rectangle,draw,inner sep=2pt](120){$\underline{(2,0,1)}$};
\draw (2,7) node [scale=.7,shape=rectangle,draw,inner sep=2pt](121){$\underline{(2,1,1)}$};
\draw (4,7) node [scale=.7,shape=rectangle,draw,inner sep=2pt](122){$\underline{(2,2,1)}$};
\draw (6,7) node [scale=.7,shape=rectangle,draw,inner sep=2pt](123){$\underline{(2,3,1)}$};

\node[left=50pt,fill=none,draw=none] at (0,7) (aaaa){$\gamma=1$};
\node[left=50pt,fill=none,draw=none] at (0,6) (aaaa){$s=2$};
\draw [decorate,decoration={brace,amplitude=10pt},xshift=-4pt,yshift=0pt]
(-2,4.75) -- (-2,8.25)node [black,midway,xshift=9pt]{};

\draw (0,10) node [scale=.7](200){$(0,0,2)$};
\draw (2,13) node [scale=.7](201){$(0,1,2)$};
\draw (4,10) node [scale=.7](202){$(0,2,2)$};
\draw (6,13) node [scale=.7,shape=rectangle,draw,inner sep=2pt](203){$\underline{(0,3,2)}$};
\draw (0,11) node [scale=.7](210){$(1,0,2)$};
\draw (2,11) node [scale=.7](211){$(1,1,2)$};
\draw (4,11) node [scale=.7](212){$(1,2,2)$};
\draw (6,11) node [scale=.7](213){$(1,3,2)$};
\draw (0,12) node [scale=.7,shape=rectangle,draw,inner sep=2pt](220){$\underline{(2,0,2)}$};
\draw (2,12) node [scale=.7](221){$(2,1,2)$};
\draw (4,12) node [scale=.7,shape=rectangle,draw,inner sep=2pt](222){$\underline{(2,2,2)}$};
\draw (6,12) node [scale=.7](223){$(2,3,2)$};

\node[left=50pt,fill=none,draw=none] at (0,12) (aaaa){$\gamma=2$};
\node[left=50pt,fill=none,draw=none] at (0,11) (aaaa){$s=9$};
\draw [decorate,decoration={brace,amplitude=10pt},xshift=-4pt,yshift=0pt]
(-2,9.75) -- (-2,13.25)node [black,midway,xshift=9pt]{};

\draw (0,15) node [scale=.7](300){$(0,0,3)$};
\draw (2,18) node [scale=.7](301){$(0,1,3)$};
\draw (4,15) node [scale=.7](302){$(0,2,3)$};
\draw (6,18) node [scale=.7](303){$(0,3,3)$};
\draw (0,16) node [scale=.7](310){$(1,0,3)$};
\draw (2,16) node [scale=.7](311){$(1,1,3)$};
\draw (4,16) node [scale=.7](312){$(1,2,3)$};
\draw (6,16) node [scale=.7](313){$(1,3,3)$};
\draw (0,17) node [scale=.7](320){$(2,0,3)$};
\draw (2,17) node [scale=.7](321){$(2,1,3)$};
\draw (4,17) node [scale=.7](322){$(2,2,3)$};
\draw (6,17) node [scale=.7](323){$(2,3,3)$};

\node[left=50pt,fill=none,draw=none] at (0,17) (aaaa){$\gamma=3$};
\node[left=50pt,fill=none,draw=none] at (0,16) (aaaa){$s=12$};
\draw [decorate,decoration={brace,amplitude=10pt},xshift=-4pt,yshift=0pt]
(-2,14.75) -- (-2,18.25)node [black,midway,xshift=9pt]{};

\draw (0,20) node [scale=.7](400){$(0,0,4)$};
\draw (2,23) node [scale=.7](401){$(0,1,4)$};
\draw (4,20) node [scale=.7](402){$(0,2,4)$};
\draw (6,23) node [scale=.7](403){$(0,3,4)$};
\draw (0,21) node [scale=.7](410){$(1,0,4)$};
\draw (2,21) node [scale=.7](411){$(1,1,4)$};
\draw (4,21) node [scale=.7](412){$(1,2,4)$};
\draw (6,21) node [scale=.7](413){$(1,3,4)$};
\draw (0,22) node [scale=.7](420){$(2,0,4)$};
\draw (2,22) node [scale=.7](421){$(2,1,4)$};
\draw (4,22) node [scale=.7](422){$(2,2,4)$};
\draw (6,22) node [scale=.7](423){$(2,3,4)$};

\node[left=50pt,fill=none,draw=none] at (0,22) (aaaa){$\gamma=4$};
\node[left=50pt,fill=none,draw=none] at (0,21) (aaaa){$s=12$};
\draw [decorate,decoration={brace,amplitude=10pt},xshift=-4pt,yshift=0pt]
(-2,19.75) -- (-2,23.25)node [black,midway,xshift=9pt]{};

\draw (400) [->] to (411);
\draw (411) [->,bend right=15] to (400);
\draw (402) [->] to (413);
\draw (413) [->,bend right=15] to (402);
\draw (410) [->] to (421);
\draw (421) [->,bend right=15] to (410);
\draw (412) [->] to (423);
\draw (423) [->,bend right=15] to (412);

\draw (420) [->] to (401);
\draw (401) [->,bend right=15] to (420);
\draw (422) [->] to (403);
\draw (403) [->,bend right=15] to (422);
\draw (420) [->] to (401);
\draw (401) [->,bend right=15] to (420);
\draw (422) [->] to (403);
\draw (403) [->,bend right=15] to (422);

\draw (300) [->] to (311);
\draw (311) [->,bend right=15] to (300);
\draw (302) [->] to (313);
\draw (313) [->,bend right=15] to (302);
\draw (310) [->] to (321);
\draw (321) [->,bend right=15] to (310);
\draw (312) [->] to (323);
\draw (323) [->,bend right=15] to (312);

\draw (320) [->] to (301);
\draw (301) [->,bend right=15] to (320);
\draw (322) [->] to (303);
\draw (303) [->,bend right=15] to (322);
\draw (320) [->] to (301);
\draw (301) [->,bend right=15] to (320);
\draw (322) [->] to (303);
\draw (303) [->,bend right=15] to (322);

\draw (200) [->] to (211);
\draw (211) [->,bend right=15] to (200);
\draw (202) [->] to (213);
\draw (213) [->,bend right=15] to (202);

\draw (210) [->] to (221);
\draw (221) [->,bend left=15] to (210);
\draw (212) [->] to (223);
\draw (201) [->,out=330,in=150] to (212);
\draw (223) [->,out=150,in=350,looseness=0] to (201);

\draw (100) [->] to (111);
\draw (111) [->,bend left=15] to (100);

\draw (000) [->] to (011);
\draw (011) [->,bend left=15] to (000);

\draw (12,14) node [scale=.7](a220){$\underline{(2,0,2)}$};
\draw (12,13) node [scale=.7](a120){$\underline{(2,0,1)}$};
\draw (12,12) node [scale=.7](a020){$\underline{(2,0,0)}$};

\draw (12,8) node [scale=.7](a110){$\underline{(1,0,1)}$};
\draw (12,7) node [scale=.7](a010){$\underline{(1,0,0)}$};

\draw (15,13) node [scale=.7](a101){$\underline{(0,1,1)}$};
\draw (15,12) node [scale=.7](a001){$\underline{(0,1,0)}$};

\draw (15,8) node [scale=.7](a121){$\underline{(2,1,1)}$};
\draw (15,7) node [scale=.7](a021){$\underline{(2,1,0)}$};

\draw (18,14) node [scale=.7](a222){$\underline{(2,2,2)}$};
\draw (18,13) node [scale=.7](a122){$\underline{(2,2,1)}$};
\draw (18,12) node [scale=.7](a022){$\underline{(2,2,0)}$};

\draw (18,8) node [scale=.7](a112){$\underline{(1,2,1)}$};
\draw (18,7) node [scale=.7](a012){$\underline{(1,2,0)}$};

\draw (18,3) node [scale=.7](a102){$\underline{(0,2,1)}$};
\draw (18,2) node [scale=.7](a002){$\underline{(0,2,0)}$};

\draw (21,14) node [scale=.7](a203){$\underline{(0,3,2)}$};
\draw (21,13) node [scale=.7](a103){$\underline{(0,3,1)}$};
\draw (21,12) node [scale=.7](a003){$\underline{(0,3,0)}$};

\draw (21,8) node [scale=.7](a123){$\underline{(2,3,1)}$};
\draw (21,7) node [scale=.7](a023){$\underline{(2,3,0)}$};

\draw (21,3) node [scale=.7](a113){$\underline{(1,3,1)}$};
\draw (21,2) node [scale=.7](a013){$\underline{(1,3,0)}$};

\draw (a020) [->,out=170,in=190, looseness=1,thick] to (a120);
\draw (a120) [->,out=170,in=190, looseness=1,thick] to (a220);
\draw (a220) [->,out=170,in=190, looseness=1,thick] to (a020);

\draw (a010) [->,out=170,in=190, looseness=1,thick] to (a110);
\draw (a110) [->,out=170,in=190, looseness=1,thick] to (a010);


\draw (a021) [->,out=170,in=190, looseness=1,thick] to (a121);
\draw (a121) [->,out=170,in=190, looseness=1,thick] to (a021);


\draw (a001) [->,out=170,in=190, looseness=1,thick] to (a101);
\draw (a101) [->,out=170,in=190, looseness=1,thick] to (a001);

\draw (a012) [->,out=170,in=190, looseness=1,thick] to (a112);
\draw (a112) [->,out=170,in=190, looseness=1,thick] to (a012);

\draw (a002) [->,out=170,in=190, looseness=1,thick] to (a102);
\draw (a102) [->,out=170,in=190, looseness=1,thick] to (a002);

\draw (a023) [->,out=170,in=190, looseness=1,thick] to (a123);
\draw (a123) [->,out=170,in=190, looseness=1,thick] to (a023);

\draw (a013) [->,out=170,in=190, looseness=1,thick] to (a113);
\draw (a113) [->,out=170,in=190, looseness=1,thick] to (a013);

\draw (a003) [->,out=170,in=190, looseness=1,thick] to (a103);
\draw (a103) [->,out=170,in=190, looseness=1,thick] to (a203);
\draw (a203) [->,out=170,in=190, looseness=1,thick] to (a003);

\draw (a222) [->,out=170,in=190, looseness=1,thick] to (a022);
\draw (a122) [->,out=170,in=190, looseness=1,thick] to (a222);
\draw (a022) [->,out=170,in=190, looseness=1,thick] to (a122);

\node[below=10pt,fill=none,draw=none] at (9,5.5) (aaaa){$x=3$, $y=5$, $s_p=37$, $r_p=23$, $k=4$, $a=2$, $\epsilon=3$, $k'=3$};
\end{tikzpicture}
\vspace*{-1in}

\caption{Example of Lemma $7.17$ Case 2}
\label{717exampleb}
\end{center}
\end{figure}

\section{Product}
In this section the partite product will be applied to the decompositions obtained in the previous section, to obtain decompositions of larger graphs.
\begin{lemma}\label{zwlemma}
Let $m=zw$ with $z$ odd and $w\not\in\{2,6\}$. Then there is a decomposition of $\overrightarrow{C}_{(m:n)}$ into $\overrightarrow{C}_{zn}$-factors and a decomposition of $\overrightarrow{C}_{(4m:n)}$ into  $\overrightarrow{C}_{2zn}$-factors.
\end{lemma}
\begin{proof}
We know that $\overrightarrow{C}_{(m:n)}=\overrightarrow{C}_{(z:n)}\otimes \overrightarrow{C}_{(w:n)}$.
By Theorem \ref{1varfun} we can decompose $\overrightarrow{C}_{(z:n)}=\bigoplus_i H_z(i,\phi(i))$ into $\overrightarrow{C}_{zn}$-factors, and by Lemma \ref{bbb} $\overrightarrow{C}_{(w:n)}=\bigoplus_j T_w(j,j)$ into $\overrightarrow{C}_n$-factors.

Then
\begin{align*}
\overrightarrow{C}_{(m:n)}&=\overrightarrow{C}_{(z:n)}\otimes \overrightarrow{C}_{(w:n)}\\
&=\left(\bigoplus_i H_z(i,\phi(i))\right)\otimes\left(\bigoplus_j H_w(j,j)\right)\\
&=\bigoplus_i \bigoplus_j H_z(i,\phi(i))\otimes H_w(j,j)
\end{align*}

By Lemma \ref{productofbalanced} $H_z(i,\phi(i))\otimes H_w(j,j)$ is a $\overrightarrow{C}_{zn}$-factor, and so $\overrightarrow{C}_{(m:n)}$ can be decomposed into $\overrightarrow{C}_{zn}$-factors.

For the result on $\overrightarrow{C}_{(4m:n)}$ we just multiply by $\overrightarrow{C}_{(4:n)}$:

\[
\overrightarrow{C}_{(4m:n)}=\overrightarrow{C}_{(m:n)}\otimes \overrightarrow{C}_{(4:n)}
\]

We can decompose $\overrightarrow{C}_{(4:n)}$ into $\overrightarrow{C}_{2n}$-factors by Theorem \ref{even} and so when multiplying by $\overrightarrow{C}_{(m:n)}$ we obtain a decomposition of $\overrightarrow{C}_{(4m:n)}$ into $\overrightarrow{C}_{2zn}$-factors.

\end{proof}

We may now use Lemma \ref{zwlemma} and the decompositions obtained in Section 7 to get a decomposition of $\overrightarrow{C}_{(xyzw:n)}$ into $\overrightarrow{C}_{xzn}$-factors and $\overrightarrow{C}_{yzn}$-factors:

\begin{lemma}\label{productcontruction}
Suppose $\overrightarrow{C}_{(m:n)}$ can be decomposed into $s_p$ $\overrightarrow{C}_{m_1n}$-factors and $r_p$ $\overrightarrow{C}_{m_2n}$ factors, with $s_p,r_p\neq 1$, $r_p+s_p=m$. Let $z$ be odd with $\gcd(m_1,z)=\gcd(m_2,z)=1$ and $w\not\in\{2,6\}$. Then there is a decomposition of $\overrightarrow{C}_{(mzw:n)}$ into $s$ $\overrightarrow{C}_{m_1zn}$-factors and $r$ $\overrightarrow{C}_{m_2zn}$-factors for any $s,r\neq 1$, $s+r=mzw$. Furthermore, if $m_1$ and $m_2$ are odd, there is a decomposition of $\overrightarrow{C}_{(4mzw:n)}$ into $s'$ $\overrightarrow{C}_{2m_1zn}$-factors and $r'$  $\overrightarrow{C}_{2m_2zn}$-factors for any $s',r'\neq 1$, $s'+r'=4mzw$.
\end{lemma}
\begin{proof}
We start with the product $\overrightarrow{C}_{(mzw:n)}=\overrightarrow{C}_{(m:n)}\otimes   \overrightarrow{C}_{(zw:n)}$.
By Lemma \ref{zwlemma} we can decompose $\overrightarrow{C}_{(zw:n)}=\oplus_{i=1}^{zw} Z_i$, where each $Z_i$ is a $\overrightarrow{C}_{zn}$-factor.

Let $s=mt+u$, with $0\leq u \leq m$. If $u\neq 1,m-1$ we decompose as follows:

\begin{align*}
\overrightarrow{C}_{(mzw:n)}&=\overrightarrow{C}_{(m:n)}\otimes   \overrightarrow{C}_{(zw:n)}\\
&=\overrightarrow{C}_{(m:n)}\otimes \left(\oplus_{i=1}^{zw} Z_i\right)\\
&= \oplus_{i=1}^{zw}\overrightarrow{C}_{(m:n)}\otimes Z_i\\
&= \left(\oplus_{i=1}^{t}\overrightarrow{C}_{(m:n)}\otimes Z_i\right)\oplus \overrightarrow{C}_{(m:n)}\otimes Z_{t+1}\oplus\left(\oplus_{i=t+2}^{zw}\overrightarrow{C}_{(m:n)}\otimes Z_i\right)
\end{align*}

From the theorem hypothesis on $\overrightarrow{C}_{(m:n)}$, we have the following decompositions:
\begin{itemize}
\item 
We decompose $\overrightarrow{C}_{(m:n)}$ into $m$ $\overrightarrow{C}_{m_1n}$-factors for the product $\oplus_{i=1}^{t}\overrightarrow{C}_{(m:n)}\otimes Z_i$.
\item
We decompose $\overrightarrow{C}_{(m:n)}$ into $m$ $\overrightarrow{C}_{m_2n}$-factors for the product $\oplus_{i=t+2}^{zw}C_{(m:n)}\otimes Z_i$.
\item
We decompose $\overrightarrow{C}_{(m:n)}$ into $u$ $\overrightarrow{C}_{m_1n}$-factors and $m-u$ $\overrightarrow{C}_{m_2n}$-factors for the product $\overrightarrow{C}_{(m:n)}\otimes Z_{t+1}$.
\end{itemize}

Because $m_1$ and $m_2$ are coprime with $z$, by Lemma \ref{productofbalanced} there is a decomposition into $mt+u=s$ $\overrightarrow{C}_{m_1zn}$-factors and $r$ $\overrightarrow{C}_{m_2zn}$-factors.

If $u=1$, we decompose as follows:
\begin{align*}
\overrightarrow{C}_{(mzw:n)}=&\left(\oplus_{i=1}^{t-1}\overrightarrow{C}_{(m:n)}\otimes Z_i\right)\oplus \left(\overrightarrow{C}_{(m:n)}\otimes Z_{t}\right)\oplus \\
&\left(\overrightarrow{C}_{(m:n)}\otimes Z_{t+1}\right)\oplus\left(\oplus_{i=t+2}^{zw}\overrightarrow{C}_{(m:n)}\otimes Z_i\right)
\end{align*}

We also have the following decompositions:
\begin{itemize}
\item 
We decompose $\overrightarrow{C}_{(m:n)}$ into $m$ $\overrightarrow{C}_{m_1n}$-factors for the product $\oplus_{i=1}^{t-1}\overrightarrow{C}_{(m:n)}\otimes Z_i$.
\item
We decompose $\overrightarrow{C}_{(m:n)}$ into $m$ $\overrightarrow{C}_{m_2n}$-factors for the product $\oplus_{i=t+2}^{zw}\overrightarrow{C}_{(m:n)}\otimes Z_i$.
\item
We decompose $\overrightarrow{C}_{(m:n)}$ into $m-2$ $\overrightarrow{C}_{m_1n}$-factors and $2$ $\overrightarrow{C}_{m_2n}$-factors for the product $\overrightarrow{C}_{(m:n)}\otimes Z_{t}$.
\item We decompose $\overrightarrow{C}_{(m:n)}$ into $3$ $\overrightarrow{C}_{m_1n}$-factors and $m-3$  $\overrightarrow{C}_{m_2n}$-factors for the product $\overrightarrow{C}_{(m:n)}\otimes Z_{t+1}$.
\end{itemize}

Because $m_1$ and $m_2$ are coprime with $z$, this gives a decomposition into $m(t-1)+m-2+3=mt+1=s$ $\overrightarrow{C}_{m_1zn}$-factors and $r$ $\overrightarrow{C}_{m_2zn}$-factors.
If $u=m-1$ we just change the roles of $m_1$ and $m_2$ and take the decomposition for $u=1$.
Therefore there is a decomposition of $\overrightarrow{C}_{(mzw:n)}$ into $s$ $\overrightarrow{C}_{m_1zn}$ factors and $r$ $\overrightarrow{C}_{m_2zn}$-factors for any $s,r\neq 1$, $r+s=mzw$.
The decomposition of $\overrightarrow{C}_{(4mzw:n)}$ into $s'$  $\overrightarrow{C}_{2m_1zn}$-factors and $r'$ $\overrightarrow{C}_{2m_2zn}$-factors works in the same way.
\end{proof}

We can now combine this result with our decompositions from Section 7 to obtain the following result, which we write using non-directed graphs, as we are getting ready to apply the results from Section 5.
\begin{theorem}\label{cvresult}
Let $x,y,z,n$ be odd numbers with $\gcd(x,z)=\gcd(y,z)=1$ and $w\not\in \{2,6\}$. Then we have the following decompositions:
\begin{enumerate}[a)] 
\item $C_{(xyzw:n)}$ can be decomposed into $s$ $C_{xzn}$-factors and $r$ $C_{yzn}$-factors for any $s,r\neq 1$, $s+r=xyzw$.
\item $C_{(4xzw:n)}$ can be decomposed into $s$ $C_{2xzn}$-factors and $r$ $C_{2zn}$-factors for any $s,r\neq 1$, $s+r=4xzw$.
\item $C_{(4xzw:n)}$ can be decomposed into $s$ $C_{2xzn}$-factors and $r$ $C_{zn}$-factors for any $s,r\neq 1$, $s+r=4xzw$.
\item $C_{(4xzw:n)}$ can be decomposed into $s$ $C_{xzn}$-factors and $r$ $C_{2zn}$-factors for any $s,r\neq 1$, $s+r=4xzw$.
\item $C_{(4xyzw:n)}$ can be decomposed into $s$ $C_{2xzn}$-factors and $r$ $C_{yzn}$-factors for any $s,r\neq 1$, $s+r=4xyzw$.
\item $C_{(4xyzw:n)}$ can be decomposed into $s$ $C_{2xzn}$-factors and $r$ $C_{2yzn}$-factors for any $s,r\neq 1$, $s+r=4xyzw$.
\end{enumerate}
\end{theorem}
\begin{proof}
\begin{enumerate}[a)]
\item $\overrightarrow{C}_{(xy:n)}$ can be decomposed into $s_p$ $\overrightarrow{C}_{xn}$-factors and $r_p$ $\overrightarrow{C}_{yn}$-factors by Lemma \ref{xnyn}. So by Lemma \ref{productcontruction}; $\overrightarrow{C}_{(xywz:n)}$ can be decomposed into $s$ $\overrightarrow{C}_{xzn}$-factors and $r$ $\overrightarrow{C}_{yzn}$-factors.
\item $\overrightarrow{C}_{(x:n)}$ can be decomposed into $s_p$ $\overrightarrow{C}_{xn}$-factors and $r_p$ $\overrightarrow{C}_{n}$-factors by Lemma \ref{1varfun}. Now apply Lemma \ref{productcontruction}.
\item $\overrightarrow{C}_{(4x:n)}$ can be decomposed into $s_p$ $\overrightarrow{C}_{2xn}$-factors and $r_p$ $\overrightarrow{C}_{n}$-factors by Lemma \ref{2xnandntheorem}. Now apply Lemma \ref{productcontruction}.
\item $\overrightarrow{C}_{(4x:n)}$ can be decomposed into $s_p$ $\overrightarrow{C}_{xn}$-factors and $r_p$ $\overrightarrow{C}_{2n}$-factors by Lemma \ref{xnand2ntheorem}. Now apply Lemma \ref{productcontruction}.
\item $\overrightarrow{C}_{(4xy:n)}$ can be decomposed into $s_p$ $\overrightarrow{C}_{2xn}$-factors and $r_p$ $\overrightarrow{C}_{yn}$-factors by Lemma \ref{2xnyn}. Now apply Lemma \ref{productcontruction}.
\item $\overrightarrow{C}_{(xy:n)}$ can be decomposed into $s_p$ $\overrightarrow{C}_{xn}$-factors and $r_p$ $\overrightarrow{C}_{yn}$-factors by Lemma \ref{xnyn}. Now apply Lemma \ref{productcontruction}.
\end{enumerate}
\end{proof}
\section{Main Result}
We now to use the decompositions that we obtained for $C_{(v:n)}$ to obtain decompositions of $K_{(v:m)}$ via Lemmas $\ref{cvntokvm}$ and $\ref{buildcompletegraph}$.

\begin{theorem}\label{maintheorem}
Let $m$ and $n$ be odd, such that $m\equiv 0 \pmod n$. Let $s$ and $r$ be such that $s,r\neq 1$ and $s+r=v\frac{m-1}{2}$. Let $x_1,\ldots x_{m/n}$, $y_1,\ldots y_{m/n}$, $z_1,\ldots,z_{m/n}$ and $w_1,\ldots,w_{m/n}$ be such that:
\begin{itemize}
\item $\gcd(x_i,z_i)=\gcd(y_i,z_i)=1$,
\item $w_i\not\in\{2,6\}$,
\item $2$ divides at most one of $x_i,y_i$ and $z_i$,
\item $v=x_iy_iz_iw_i$ if $2$ divides none of $x_i,y_i,z_i$,
\item $v=2x_iy_iz_iw_i$ if $2$ divides one of $x_i,y_i,z_i$.
\end{itemize}
Furthermore, let $F_1$ be a $[(x_1z_1n)^{\frac{v}{x_1z_1}},\ldots,(x_{m/n}z_{m/n}n)^{\frac{v}{x_{m/n}z_{m/n}}}]$-factor, and let $F_2$ be a $[(y_1z_1n)^{\frac{v}{y_1z_1}},\ldots,(y_{m/n}z_{m/n}n)^{\frac{v}{y_{m/n}z_{m/n}}}]$-factor.
Then there is a decomposition of $K_{(v:m)}$ into $s$ copies of $F_1$ and $r$ copies of $F_2$.
\end{theorem}
\begin{proof}
By Theorem \ref{OP} there is a decomposition of $K_m$ into $C_n$-factors.

Pick $s_1,\ldots, s_{\frac{m-1}{2}}$ such that $s=\bigoplus_{i=1}^{v\frac{m-1}{2}}s_i$, with $0\leq s_i \leq v$ and $s_i\not\in \{1,v-1\}$.

By Theorem \ref{cvresult} there is a decomposition of $C_{(v:n)}$ into $s_i$ $C_{x_iz_in}$-factors and $r_i=v-s_i$ $C_{y_iz_in}$-factors.

Therefore by Theorem \ref{cvntokvm} there is a decomposition of $K_{(v:m)}$ into $s$ copies of the 2-factor $F_1$ and $r$ copies of the $2$-factor $F_2$.
\end{proof}

\section{Applications}
We can use our results to solve many cases of the Hamilton-Waterloo Problem for complete graphs.
For some of them we will need the notion of resolvable group divisible design.

A \emph{resolvable group divisible design} $(k,\l)\-\RGDD(h^u)$ is a triple $(\mathcal{V},\mathcal{G},\mathcal{B})$ where $\mathcal{V}$ is a finite set of size $v=hu$, $\mathcal{G}$ is a partition of $\mathcal{V}$ into $u$ \emph{groups} each containing $h$ elements, and $\mathcal{B}$ is a collection of $k$ element subsets of $\mathcal{V}$ called \emph{blocks} which satisfy the following properties.
\begin{packed_item}
 \item If $B\in \mathcal{B}$, then $|B|=k$.
 \item If a pair of elements from $\mathcal{V}$ appear in the same group, then the pair cannot be in any block.
 \item Two points that are not in the same group, called a \emph{transverse pair}, appear in exactly $\l$ blocks.
 \item $|\mathcal{G}|>1$.
 \item The blocks can be partitioned into parallel classes such that for each element of $\mathcal{V}$ there is exactly one block in each parallel class containing it.
\end{packed_item}
Here we use the term group to indicate an element of $\mathcal{G}$. In this context, group simply means a set of elements without any algebraic structure.
If $\l=1$, we refer to the RGDD as a $k\-\RGDD(h^u)$.

In \cite{R} the following characterization theorem was proven:
\begin{theorem}{\normalfont\cite{R}}\label{3rgdd}
 A $(3,\l)\-\RGDD(h^u)$ exists if and only if $u\geq 3,\l h(u-1)$ is even, $hu\equiv 0\pmod{3}$, and $(\l,h,u)\not\in\{(1,2,6),(1,6,3) \}\bigcup\{(2j+1,2,3), (4j+2,1,6):j\geq 0\}$.
\end{theorem}

In \cite{AKKPO} the authors used resolvable group divisible designs together with Theorem \ref{3rgdd} to decompose complete graphs into $C_3$-factors and $C_{3x}$-factors.
\begin{lemma}{\normalfont\cite{AKKPO}}
\label{main}
Let $x \geq 3$, $y \geq 3$ and $m$ be positive integers such that both $x$ and $y$ divide $3m$. Suppose the following conditions are satisfied:
\begin{packed_item}
\item There exists a 3-$\RGDD(h^{u})$,
\item there exists a decomposition of $K_{(m:3)}$ into $s_{p}$ $C_{x}$-factors and $r_{p}$ $C_{y}$-factors, for\\
$p \in \{1,2, \ldots, \frac{h(u-1)}{2}\}$,
\item there exists a decomposition of $K_{hm}$ into $s_\beta$ $C_{x}$-factors and $r_\beta$ $C_{y}$-factors.
\end{packed_item}
Let
\[s_{\alpha}=\sum_{p=1}^{\frac{h(u-1)}{2}}s_{p} \mbox{ and } r_{\alpha}=\sum_{p=1}^{\frac{h(u-1)}{2}}r_{p}.\]
Then there exists a decomposition of $K_{hum}$ into $s_\alpha+s_\beta$ $C_{x}$-factors and $r_\alpha+r_\beta$ $C_{y}$-factors.
\end{lemma}

We can now apply our decompositions to extend the result from \cite{AKKPO}. We will be concerned with prime numbers whose greatest power that divides $x$ is the same as their greatest power that divides $y$. Thus we give the following definition:
\begin{definition}
Let $x$ and $y$ be natural numbers, with $x=p_1^{a_1}\ldots p_n^{a_n}$ and $y=p_1^{b_1}\ldots p_n^{b_n}$ their prime factorization. Then we define the special greatest common divisor of $x$ and $y$, $\sgcd{x,y}$ as the smallest number such that $p_i^{a_i}$ divides $\sgcd{x,y}$ if and only if $a_i=b_i$.
\end{definition}
\begin{example}
For example if $x=2^33^25^27^2$ and $y=3^25^37^211^4$, then $\sgcd{x,y}=3^27^2$.
\end{example}

\begin{corollary}
Let $x,y,n$ be integers such that
\begin{itemize}
\item $\frac{xy}{\sgcd{x,y}}$ divides $n$,
\item $x,y\not\equiv 0 \pmod 4$, and $4$ divides $n$ if $2$ divides $xy$;
\item $3n=\frac{huxyw}{\sgcd{x,y}}$, with $h\equiv 0 \pmod 3$, $u\geq 3$, $h(u-1)$ even, and $(h,u)\not\in\{(2,6),(6,3)\}$.
\end{itemize}

Then there exists a decomposition of $K_{3n}$ into $s$ $C_{3x}$-factors and $r$ $C_{3y}$-factors for every pair $(s,r)$ such that $s+r=\lfloor \frac{3n-1}{2}\rfloor$, $s,r\neq 1$.
\end{corollary}
\begin{proof}
Let $(s,r)$ be such that $s+r=\lfloor \frac{3n-1}{2}\rfloor$ and $s,r\neq 1$.
If $s\geq r$ let $s_0=\lfloor\frac{3n-1}{2}\rfloor$ and $r_0=0$. Otherwise let $s_0=0$ and $r_0=\lfloor\frac{3n-1}{2}\rfloor$.
Let $s_1,\ldots,s_{\frac{h(u-1)}{2}}$ and $r_1,\ldots,r_{\frac{h(u-1)}{2}}$ be such that $s_i+r_i=\frac{xyw}{\sgcd{x,y}}$, $s_i,r_i\neq 1$ for all $i$ and 
\[r_{\alpha}=\sum_{i=0}^{\frac{h(u-1)}{2}}r_{i} \mbox{ and } s_{\alpha}=\sum_{i=0}^{\frac{h(u-1)}{2}}s_{i}.\]

From Theorem \ref{3rgdd} we know that there is a 3-$\RGDD(h^{u})$. Let $z=\sgcd{x,y}$, $x_1=\frac{x}{z}$ and $y_1=\frac{y}{z}$. Then 
\[
\frac{xyw}{z}=\frac{x_1zy_1zw}{z}=x_1y_1zw.
\]
So we may apply Theorem \ref{maintheorem} to obtain a decomposition of $K_{\left(\frac{xyw}{z}:3\right)}$ into $s_i$ $C_{3x_1z}$-factors and $r_i$ $C_{3y_1z}$-factors for each $i$.

Because $hxyw=hx_1zyw=hxy_1zw$ we have that $3x_1z|(hxyw)$ and $3y_1z|(hxyw)$. From Theorem \ref{OP} there is a decomposition of $K_{\frac{hxyw}{z}}$ into $C_{3x_1z}$-factors, and there is also a decomposition of $K_{\frac{hxyw}{z}}$ into $C_{3y_1z}$-factors.

Thus we may apply Lemma \ref{main} to obtain a decomposition of $K_{\frac{huxyw}{z}}=K_{3n}$ into $s$ $C_{3x}$-factors and $r$ $C_{3y}$-factors.
\end{proof}

We can also use Lemma \ref{buildcompletegraph} to obtain decompositions of complete graphs into $C_x$-factors and $C_y$-factors:

\begin{corollary}
Let $m,x$, and $y$ be integers such that:
\begin{itemize}
\item $z=\sgcd{x,y}$, $w=\frac{\gcd(x,y)}{z}\geq 2$,
\item $\frac{xy}{z}$ divides $m$,
\item $4$ does not divide $x$ nor $y$.
\item Neither $x$ nor $y$ is $3$ if $\frac{m}{w}\in\{6,12\}$.
\end{itemize}

Then there exists a decomposition of $K_m$ into $s$ $C_{x}$-factors and $r$ $C_y$-factors for every $s,r\neq 1$.
\end{corollary}
\begin{proof}
Let $s,r$ be such that $s+r=\lfloor\frac{m-1}{2}\rfloor$ and $s,r\neq 1$.

Let 
\[
k=\frac{mz}{xy}\quad
x'=\frac{x}{zw}\quad
y'=\frac{y}{zw}\quad
m'=\frac{m}{w}=\frac{xyk}{zw}=x'y'zwk
\]
Let $s_\alpha,s_\beta,r_\alpha,r_\beta$ be such that $s_\beta,r_\beta\neq 1$, $\{s_\alpha,r_\alpha\}=\{0,\
\lfloor\frac{m'-1}{2}\rfloor\}$, $s=s_\alpha+s_\beta$ and $r=r_\alpha+r_\beta$.
By Theorem \ref{maintheorem} there is a decomposition of $K_{(m':w)}$ into $s_\beta$ $C_{x'zw}$-factors 
and $r_\beta$ $C_{y'zw}$-factors. This is a decomposition of $K_{(m':w)}$ into $s_\beta$ $C_x$-factors and 
$r_\beta$ $C_y$-factors.
Because $\frac{xy}{z}$ divides $m$, it follows that both $x$ and $y$ divide $m'=\frac{m}{w}=m\frac{z}
{\gcd(x,y)}$. Thus by Theorem \ref{OP} there is decompositon of $K_{m'}$ into $s_\alpha$ $C_x$-factors and $r_\alpha$ $C_y$-factors (keep in mind 
that one of $s_\alpha$ and $r_\alpha$ is $0$).
Then by Lemma \ref{buildcompletegraph} there is a decomposition of $K_{m'w}$ into $s$ $C_{x}$-factors and $r$ $C_{y}$-factors.
Therefore there is a decompostion of $K_{m}$ into $s$ $C_{x}$-factors and $r$ $C_{y}$-factors.
\end{proof}
Notice that by asking $\frac{xy}{z}$ to divide $m$ we cover some of the cases left open in \cite{BDT} for the odd order case.
\begin{example}
Let $m=3^35^3$, $x=3^15^2$ and $y=3^25^2$. We have:
\begin{align*}
z=\sgcd{x,y}&=5^2,&w&=3^1,&k&=5^1
\end{align*}

And $\frac{xy}{z}=3^35^2$ divides $m$. So we can decompose $K_{m}$ into $s$ $C_x$-factors and $r$ $C_y
$-factors for any $s,r\neq 1$. Note that $l=\lcm(x,y)=3^25^2$. The number of vertices, $m$, is a multiple of $l$, however $xy\!\!\not | \,m$. Thus, Theorem \ref{theoremBDT} cannot be applied here.
\end{example}
\begin{example}
Let $m=4^13^45^37^1$, $x=2^13^15^27^1$ and $y=3^35^2$. We have:
\begin{align*}
z=\sgcd{x,y}&=5^2,&w&=3^1,&k&=2^15^1
\end{align*}

And $\frac{xy}{z}=2^13^45^27^1$ divides $m$. So we can decompose $K_{m}$ into $s$ $C_x$-factors and $r$ $C_y
$-factors for any $s,r\neq 1$. Note that because $x$ is even Theorem \ref{theoremBDT} cannot be applied here.
\end{example}

We can also use our Lemma \ref{buildcompletegraphnonuniform} to obtain non-uniform decompositions of 
complete graphs:

\begin{corollary}
Let $v,m,n,x_1,\ldots,x_k,y_1,\ldots,y_k$ be integers such that:
\begin{itemize}
\item $m\geq 3$ is odd,
\item $n$ divides $m$, $x_i$, and $y_i$ for every $i$,
\item $k=\frac{m}{n}$,
\item $z_i=\sgcd{x_i,y_i}$,
\item $\frac{x_iy_i}{z_in}$ divides $v$ for each $i$,
\item $x_i$ divides $v$ for each $i$,
\item $4$ does not divide $x_i$ nor $y_i$ for any $i$,
\item $3\not\in\{x_1,\ldots,x_k,y_1,\ldots,y_k\}$ if $k\in\{6,12\}$.
\end{itemize}

Then there exists a decomposition of $K_{vm}$ into $s$ $[x_1^{vn/x_1},\ldots,x_k^{vn/x_k}]
$-factors and $r$ $[y_1^{vn/y_1},\ldots,y_k^{vn/y_k}]$-factors for every $s,r\neq 1$.
\end{corollary}
\begin{proof}
Let $s,r$ be such that $s+r=\lfloor\frac{vm-1}{2}\rfloor$ and $s,r\neq 1$.

Let 
\[
x_i'=\frac{x_i}{z_in}\quad
y_i'=\frac{y_i}{z_in}\quad k_i=\frac{vz_in}{x_iy_i}
\]

Then
\[
v=\frac{x_iy_ik_i}{z_in}=\frac{x_i'z_iny_i'z_ink_i}{z_in}=x_i'y_i'z_ink_i
\]

Let $s_\alpha,s_\beta,r_\alpha,r_\beta$ be such that $s_\beta,r_\beta\neq 1$, $\{s_\alpha,r_\alpha\}=\{0,
\lfloor\frac{v-1}{2}\rfloor\}$, $s=s_\alpha+s_\beta$ and $r=r_\alpha+r_\beta$.
By Theorem \ref{maintheorem} there is a decomposition of $K_{(v:m)}$ into $s_\beta$ $[x_1^{vn/x_1},\ldots,x_k^{vn/x_k}]$-factors and $r_\beta$ $[y_1^{vn/y_1},\ldots,y_k^{vn/y_k}]$-factors.
We partition $mK_{v}$ into $k$ copies of $nK_{v}$, labeled $\kappa_1,\ldots\kappa_k$. Because $\frac{x_iy_i}{z_in}$ and $x_i$ divide $v$, we get that both $x_i$ and $y_i$ divide $v$, by Theorem \ref{OP} 
there is a decompostion of $K_v$ into $s_{\alpha}$ $C_{x_i}$-factors and $r_\alpha$ $C_{y_i}$-factors (keep in mind that one of $s_\alpha$ and $r_\alpha$ is $0$). This means that $\kappa_i$ can be decomposed into $s_\alpha$ $[x_i^{vn/x_i}]$-factors 
and $r_\alpha$ $[y_i^{vn/y_i}]$-factors. Combining these decompositions we get a decomposition of
$mK_{v}$ into $s_\alpha$ $[x_1^{vn/x_1},\ldots,x_k^{vn/x_k}]$-factors and $r_\alpha$ 
$[y_1^{vn/y_1},\ldots,y_k^{vn/y_k}]$-factors

Then by Lemma \ref{buildcompletegraphnonuniform} there is a decomposition $K_{vm}$ into $s$ $
[x_1^{vn/x_1},\ldots,x_k^{vn/x_k}]$-factors and $r$ $[y_1^{vn/y_1},
\ldots,y_k^{vn/y_k}]$-factors for every $s,r\neq 1$.
\end{proof}
\begin{example}
Let $v=5^37^211^313^4$, $m=3^15^1$, $n=5$, $k=3$, $x_1=5^27^2$, $y_1=5^17^211^113^1$, $x_2=5^111^113^1$, $y_2=5^17^213^3$, $x_3=5^27^111^113^4$, and $y_3=5^1$. We have:
\begin{align*}
z_1=\sgcd{x_1,y_1}&=7^2, &\frac{x_1y_1}{z_1n}=5^27^2\\
z_2=\sgcd{x_2,y_2}&=5^1, &\frac{x_2y_2}{z_2n}=7^211^113^4\\
z_3=\sgcd{x_3,y_3}&=1, &\frac{x_3y_3}{z_3n}=5^27^111^113^4\\
\end{align*}

We can see that $\frac{x_iy_i}{z_in}$ and $x_i$ divide $v$ for each $i$. 

Let $F_1$ be the $2$-factor consisting of:
\begin{itemize}
\item $\frac{vn}{x_1}=5^211^313^4$ cycles of size $x_1=5^27^2$,
\item $\frac{vn}{x_2}=5^37^211^213^3$ cycles of size $x_2=5^111^113^1$,
\item and $\frac{vn}{x_3}=5^26^111^2$ cycles of size $x_3=5^27^111^113^4$.
\end{itemize}

Let $F_2$ be the $2$-factor consisting of:
\begin{itemize}
\item $\frac{vn}{y_1}=5^311^213^3$ cycles of size $y_1=5^17^211^113^1$,
\item $\frac{vn}{y_2}=5^311^313^1$ cycles of size $y_2=5^17^213^3$,
\item and $\frac{vn}{y_3}=5^37^211^313^4$ cycles of size $y_3=5^1$.
\end{itemize}

Then we can decompose $K_{vm}$ into $s$ copies of $F_1$ and $r$ copies of $F_2$ for any $s,r\neq 1$.
\end{example}

\newpage
\section{Appendix}
\begin{lemma}
There is a decomposition of $C_{(12:3)}$ into $7$ $C_{12}$-factors and $5$ $C_{6}$-factors.
\end{lemma}
\begin{proof}
Let $G_0$, $G_1$, and $G_2$ be the partite sets. We will construct each factors by taking a difference between each pair of partite sets. Notice that if the sum of the three differences is congruent to $6$ modulo $12$, this gives a $C_{6}$-factor. Likewise, if the sum of the differences is congruent to $4$ or $8$ modulo $12$ we get a $C_{12}$-factor.

The $C_{6}$ factors are obtained by taking differences:
\[
\begin{array}{|c|c|c|c|}
\hline
\text{$G_0$ to $G_1$}&\text{$G_1$ to $G_2$}&\text{$G_2$ to $G_3$}&\text{total mod 12}\\
\hline
0 & 3 & 3 & 6 \\
\hline
1 & 1 & 4& 6 \\
\hline
2 & 2 & 2 & 6\\
\hline
3 & 9 & 6 & 6 \\
\hline
4 & 4 & 10 &  6\\
\hline
\end{array}
\]

The $C_{12}$-factors are obtained by taking differences:
\[
\begin{array}{|c|c|c|c|}
\hline
\text{$G_0$ to $G_1$}&\text{$G_1$ to $G_2$}&\text{$G_2$ to $G_3$}&\text{total mod 12}\\
\hline
5 & 10 & 5 &  8\\
\hline
6 & 6 & 8 & 8\\
\hline
7 & 0 & 9 &  4\\
\hline
8 & 5 & 7 &  8\\
\hline
9 & 7 & 0 &  4\\
\hline
10 & 11 & 11 & 8\\
\hline
11 & 8 & 1 &  8\\
\hline
\end{array}
\]

Because all possible difference between each pair of partite sets has been taken once, this provides the desired decomposition.
\end{proof}
\begin{lemma}\label{theissuewith12}
There is a decomposition of $C_{(12:n)}$ into $7$ $C_{3n}$-factors and $5$ $C_{2n}$-factors, for any $n\geq 5$.
\end{lemma}
\begin{proof}
We have to pick differences between $n$ pairs of partite sets. For the first $n-3$ differences sets choose differences that add up to $0$. This can be achieved by taking the difference one would take to decompose $C_{(12:n-3)}$ into $C_{n-3}$-factors. For the last three differences, choose in the same way as we did in the previous lemma.
\end{proof}

\section{Bibliography}
\begin{thebibliography}{99}
\bibitem{AKKPO}
J. Asplund, D. Kamin, M. Keranen,  A. Pastine, and S. \"{O}zkan, On the Hamilton-Waterloo problem with triangle factors and $C_{3x}$-factors, {\em Australas. J. Combin.} {\bf 64} (2016), 458-474.
\bibitem{AH}
B. Alspach and R. Haggkvist, Some observations on the Oberwolfach problem, {\em Journal of Graph Theory}
{\bf 9} (1985), 177-187.
\bibitem{ASSW}
B. Alspach, P. Schellenberg, D.R. Stinson, and D. Wagner, The Oberwolfach problem and factors of uniform length,
{\em Journal of Combinatorial Theory, Ser. A} {\bf 52} (1989), 20-43.
\bibitem{BHR}
E.J. Billington, D.G. Hoffman, C.A. Rodger, Resolvable gregarious cycle decompositions of complete equipartite graphs,
{\em Discrete Math} {\bf 308} (2008), 2844-2853.
\bibitem{BS}
R.C. Bose and S.S. Shrikhande, On the Construction of Sets of Mutually Orthogonal Latin Squares and the Falsity of a Conjecture of Euler, {\em Transactions of the American Mathematical Society} {\bf 95} (1960), 191-209.
\bibitem{BSP}
R.C. Bose, S.S. Shrikhande, E.T. Parker, Further Results on the Construction of Mutually Orthogonal Latin Squares and the Falsity of Euler's Conjecture, {\em Canadian Journal of Mathematics} {\bf 12}  (1960), 189-203.
\bibitem{BD}
D. Bryant and P.Danziger, On bipartite $2$-factorisations of $K_n−I$ and the Oberwolfach problem,
{\em Journal of Graph Theory}, {\bf 68} (2011), 22–37.
\bibitem{BDD}
D. Bryant, P. Danziger, M. Dean, On the Hamilton-Waterloo Problem for Bipartite $2$-Factors, {\em Journal of Combinatorial Designs}, {\bf 21} (2013), 60-80.
\bibitem{BDP}
D. Bryant, P. Danziger, W. Pettersson, Bipartite $2$-Factorizations of Complete Multipartite Graphs,
{\em Journal of Graph Theory}, {\bf 78} (2015), 287-294.
\bibitem{BR}
D. Bryant and C.A. Rodger, ```Cycle decompositions,'''in The CRC Handbook of Combinatorial Designs, 2nd Edition, C.J. Colbourn, J.H. Dinitz (Editors), CRC Press, Boca Raton, 2007, pp. 373-382.
\bibitem{BDT}
A.C. Burgess, P. Danziger, T. Traetta, On the Hamilton-Waterloo Problem with odd orders, Preprint (2016), arXiv:1510.07079v2 [math.CO].
\bibitem{H}
R. H¨aggkvist, A lemma on cycle decompositions, {\em Ann Discrete Math} {\bf 27} (1985), 227–232.
\bibitem{HH} D.G. Hoffman and S.H. Holliday, Resolvably Decomposing Complete Equipartite Graphs Minus a One-Factor into Cycles of Uniform Even Length, {\em ARS Combinatori} {\bf 110} (2013), 435-445.
\bibitem{Liu1}
J. Liu, A Generalization of the Oberwolfach Problem and $C_t$-factorizations of Complete Equipartite Graphs, {\em Journal of Combinatorial Designs} {\bf 8} (2000), 42-49.
\bibitem{Liu2}
J. Liu, A Complete Solution to the Generalized Oberwolfach Problem with Uniform Table Sizes, {\em J. Combin. Theory Ser. A.} {\bf 101} (2003), no. 1, 20-34.
\bibitem{R}
R.S. Rees, Two new direct product-type constructions for resolvable group-divisible designs, {\em Journal of Combinatorial Designs} {\bf 1} (1993), 15-26.
\end{thebibliography}
\end{document}